\documentclass[a4paper]{article}

\usepackage[dvipsnames]{xcolor}
\usepackage[colorlinks=true,linkcolor=Maroon,citecolor=YellowOrange!85!black,urlcolor=black]{hyperref} 
\usepackage{tikz}
\usepackage[utf8]{inputenc}
\usepackage[T1]{fontenc}
\usepackage{lmodern}
\usepackage{newunicodechar}
\newunicodechar{Φ}{\ensuremath{\Phi}}
\newunicodechar{≥}{\ensuremath{\ge}}

\usepackage[backend=bibtex,style=alphabetic,firstinits=true,isbn=false,doi=false,url=false,eprint=false]{biblatex}
\definecolor{lime}{HTML}{A6CE39}
\newcommand{\orcidicon}{%
	\begin{tikzpicture}
	\draw[lime, fill=lime] (0,0) 
		circle [radius=0.16] 
		node[white] {{\fontfamily{qag}\selectfont \tiny iD}};
	\end{tikzpicture}
}

\usepackage{amssymb,amsmath,amsthm}

\usepackage{todonotes}
\usepackage{verbatim} 
\usepackage{enumerate}
\usepackage{mathtools}
\usepackage[normalem]{ulem}

\usepackage{bbm}

\usepackage{lineno}

\makeatletter
\def\lstAZ{A, B, C, D, E, F, G, H, I, J, K, L, M, N, O, P, Q, R, S, T, U, V, W, X, Y, Z}

\newcommand{\MkScr}[1]{\expandafter\def\csname s#1\endcsname{\mathscr{#1}}}
\newcommand{\MkFrak}[1]{\expandafter\def\csname f#1\endcsname{\mathfrak{#1}}}
\newcommand{\MkCal}[1]{\expandafter\def\csname c#1\endcsname{\mathcal{#1}}}
\newcommand{\MkBB}[1]{\expandafter\def\csname b#1\endcsname{\mathbb{#1}}}
\@for\i:=\lstAZ\do{%
	\expandafter\MkScr \i     }    
\@for\i:=\lstAZ\do{%
	\expandafter\MkCal \i     }    
\@for\i:=\lstAZ\do{%
	\expandafter\MkFrak \i     }    
\@for\i:=\lstAZ\do{%
	\expandafter\MkBB \i     }    
\makeatother

\newcommand{\PR}{\mathbb{P}}

\newcommand{\bONE}{\mathbf{1}}

\newcommand{\dd}{ \mathrm{d}}
\newcommand{\e}{{\mathrm{e}}}

\DeclareMathOperator*{\grad}{grad}

\renewcommand{\epsilon}{\varepsilon}
\renewcommand{\phi}{\varphi}
\renewcommand{\tilde}{\widetilde}

\newcommand{\vn}[1]{\left| \! \left| #1\right| \! \right|} 
\newcommand{\tn}[1]{\left| \! \left| \! \left|#1\right| \! \right| \! \right|}

\newcommand{\ip}[2]{\langle #1,#2\rangle}

\newcommand{\st}{\text{st}}

\numberwithin{equation}{section}

\newtheorem{theorem}{Theorem}[section]
\newtheorem{lemma}[theorem]{Lemma}
\newtheorem{proposition}[theorem]{Proposition}
\newtheorem{corollary}[theorem]{Corollary}

\theoremstyle{definition}
\newtheorem{definition}[theorem]{Definition}

\newtheorem{remark}[theorem]{Remark}

\newtheorem{assumption}[theorem]{Assumption}
\newtheorem{example}[theorem]{Example}
\newtheorem{condition}[theorem]{Condition}

\newcommand\Item[1][]{%
  \ifx\relax#1\relax  \item \else \item[#1] \fi
  \abovedisplayskip=0pt\abovedisplayshortskip=0pt~\vspace*{-\baselineskip}}

\definecolor{DarkGoldenrod}{rgb}{0.78, 0.45, 0.04}

\newcounter{subass}

\makeatletter
\def\blx@bblfile@biber{%
  \blx@secinit
  \begingroup
  \blx@bblstart
\begingroup
\makeatletter
\@ifundefined{ver@biblatex.sty}
  {\@latex@error
     {Missing 'biblatex' package}
     {The bibliography requires the 'biblatex' package.}
      \aftergroup\endinput}
  {}
\endgroup

\sortlist[entry]{anyt/global/}
  \entry{AB78}{book}{}
    \name{author}{2}{}{%
      {{hash=ACD}{%
         family={Aliprantis},
         family_i={A\bibinitperiod},
         given={Charalambos\bibnamedelima D.},
         given_i={C\bibinitperiod\bibinitdelim D\bibinitperiod},
      }}%
      {{hash=BO}{%
         family={Burkinshaw},
         family_i={B\bibinitperiod},
         given={Owen},
         given_i={O\bibinitperiod},
      }}%
    }
    \list{publisher}{1}{%
      {Academic Press [Harcourt Brace Jovanovich, Publishers], New
  York-London}%
    }
    \strng{namehash}{ACDBO1}
    \strng{fullhash}{ACDBO1}
    \field{labelnamesource}{author}
    \field{labeltitlesource}{title}
    \field{labelalpha}{AB78}
    \field{sortinit}{A}
    \field{sortinithash}{A}
    \field{title}{Locally Solid {R}iesz Spaces}
    \field{year}{1978}
  \endentry

  \entry{Ar87}{article}{}
    \name{author}{1}{}{%
      {{hash=AW}{%
         family={Arendt},
         family_i={A\bibinitperiod},
         given={Wolfgang},
         given_i={W\bibinitperiod},
      }}%
    }
    \strng{namehash}{AW1}
    \strng{fullhash}{AW1}
    \field{labelnamesource}{author}
    \field{labeltitlesource}{title}
    \field{labelalpha}{Are87}
    \field{sortinit}{A}
    \field{sortinithash}{A}
    \verb{doi}
    \verb 10.1112/plms/s3-54.2.321
    \endverb
    \verb{eprint}
    \verb http://plms.oxfordjournals.org/content/s3-54/2/321.full.pdf+html
    \endverb
    \field{number}{2}
    \field{pages}{321\bibrangedash 349}
    \field{title}{Resolvent Positive Operators}
    \field{volume}{s3-54}
    \field{journaltitle}{Proceedings of the London Mathematical Society}
    \field{year}{1987}
  \endentry

  \entry{BaBlMu2008}{article}{}
    \name{author}{3}{}{%
      {{hash=BN}{%
         family={B{\"a}uerle},
         family_i={B\bibinitperiod},
         given={Nicole},
         given_i={N\bibinitperiod},
      }}%
      {{hash=BA}{%
         family={Blatter},
         family_i={B\bibinitperiod},
         given={Anja},
         given_i={A\bibinitperiod},
      }}%
      {{hash=MA}{%
         family={M{\"u}ller},
         family_i={M\bibinitperiod},
         given={Alfred},
         given_i={A\bibinitperiod},
      }}%
    }
    \strng{namehash}{BNBAMA1}
    \strng{fullhash}{BNBAMA1}
    \field{labelnamesource}{author}
    \field{labeltitlesource}{title}
    \field{labelalpha}{BBM08}
    \field{sortinit}{B}
    \field{sortinithash}{B}
    \verb{doi}
    \verb 10.1007/s00186-007-0185-6
    \endverb
    \field{issn}{1432-5217}
    \field{number}{1}
    \field{pages}{161\bibrangedash 186}
    \field{title}{Dependence properties and comparison results for L{\'e}vy
  processes}
    \verb{url}
    \verb http://dx.doi.org/10.1007/s00186-007-0185-6
    \endverb
    \field{volume}{67}
    \field{journaltitle}{Mathematical Methods of Operations Research}
    \field{year}{2008}
  \endentry

  \entry{Bo07}{book}{}
    \name{author}{1}{}{%
      {{hash=BVI}{%
         family={Bogachev},
         family_i={B\bibinitperiod},
         given={Vladimir\bibnamedelima I.},
         given_i={V\bibinitperiod\bibinitdelim I\bibinitperiod},
      }}%
    }
    \list{publisher}{1}{%
      {Springer-Verlag}%
    }
    \strng{namehash}{BVI1}
    \strng{fullhash}{BVI1}
    \field{labelnamesource}{author}
    \field{labeltitlesource}{title}
    \field{labelalpha}{Bog07}
    \field{sortinit}{B}
    \field{sortinithash}{B}
    \field{title}{Measure Theory}
    \field{year}{2007}
  \endentry

  \entry{BeRu2006}{article}{}
    \name{author}{2}{}{%
      {{hash=BJ}{%
         family={Bergenthum},
         family_i={B\bibinitperiod},
         given={Jan},
         given_i={J\bibinitperiod},
      }}%
      {{hash=RL}{%
         family={R{\"u}schendorf},
         family_i={R\bibinitperiod},
         given={Ludger},
         given_i={L\bibinitperiod},
      }}%
    }
    \strng{namehash}{BJRL1}
    \strng{fullhash}{BJRL1}
    \field{labelnamesource}{author}
    \field{labeltitlesource}{title}
    \field{labelalpha}{BR06}
    \field{sortinit}{B}
    \field{sortinithash}{B}
    \field{abstract}{%
    In this paper we generalize the recent comparison results of El Karoui
  et al. (Math Finance 8:93--126, 1998), Bellamy and Jeanblanc (Finance Stoch
  4:209--222, 2000) and Gushchin and Mordecki (Proc Steklov Inst Math
  237:73--113, 2002) to d-dimensional exponential semimartingales. Our main
  result gives sufficient conditions for the comparison of European options
  with respect to martingale pricing measures. The comparison is with respect
  to convex and also with respect to directionally convex functions. Sufficient
  conditions for these orderings are formulated in terms of the predictable
  characteristics of the stochastic logarithm of the stock price processes. As
  examples we discuss the comparison of exponential semimartingales to
  multivariate diffusion processes, to stochastic volatility models, to
  L{\'e}vy processes, and to diffusions with jumps. We obtain extensions of
  several recent results on nontrivial price intervals. A crucial property in
  this approach is the propagation of convexity property. We develop a new
  approach to establish this property for several further examples of
  univariate and multivariate processes.%
    }
    \verb{doi}
    \verb 10.1007/s00780-006-0001-9
    \endverb
    \field{issn}{1432-1122}
    \field{number}{2}
    \field{pages}{222}
    \field{title}{Comparison of Option Prices in Semimartingale Models}
    \verb{url}
    \verb http://dx.doi.org/10.1007/s00780-006-0001-9
    \endverb
    \field{volume}{10}
    \field{journaltitle}{Finance and Stochastics}
    \field{year}{2006}
  \endentry

  \entry{BeRu2007}{article}{}
    \name{author}{2}{}{%
      {{hash=BJ}{%
         family={Bergenthum},
         family_i={B\bibinitperiod},
         given={Jan},
         given_i={J\bibinitperiod},
      }}%
      {{hash=RL}{%
         family={Rüschendorf},
         family_i={R\bibinitperiod},
         given={Ludger},
         given_i={L\bibinitperiod},
      }}%
    }
    \list{publisher}{1}{%
      {The Institute of Mathematical Statistics}%
    }
    \strng{namehash}{BJRL2}
    \strng{fullhash}{BJRL2}
    \field{labelnamesource}{author}
    \field{labeltitlesource}{title}
    \field{labelalpha}{BR07}
    \field{sortinit}{B}
    \field{sortinithash}{B}
    \verb{doi}
    \verb 10.1214/009117906000000386
    \endverb
    \field{number}{1}
    \field{pages}{228\bibrangedash 254}
    \field{title}{Comparison of semimartingales and Lévy processes}
    \verb{url}
    \verb http://dx.doi.org/10.1214/009117906000000386
    \endverb
    \field{volume}{35}
    \field{journaltitle}{Ann. Probab.}
    \field{month}{01}
    \field{year}{2007}
  \endentry

  \entry{BoSa2002}{book}{}
    \name{author}{2}{}{%
      {{hash=BAN}{%
         family={Borodin},
         family_i={B\bibinitperiod},
         given={Andrei\bibnamedelima N.},
         given_i={A\bibinitperiod\bibinitdelim N\bibinitperiod},
      }}%
      {{hash=SP}{%
         family={Salminen},
         family_i={S\bibinitperiod},
         given={Paavo},
         given_i={P\bibinitperiod},
      }}%
    }
    \list{publisher}{1}{%
      {Birkh\"auser Verlag, Basel}%
    }
    \strng{namehash}{BANSP1}
    \strng{fullhash}{BANSP1}
    \field{labelnamesource}{author}
    \field{labeltitlesource}{title}
    \field{labelalpha}{BS02}
    \field{sortinit}{B}
    \field{sortinithash}{B}
    \field{edition}{Second}
    \field{isbn}{3-7643-6705-9}
    \field{pages}{xvi+672}
    \field{series}{Probability and its Applications}
    \field{title}{Handbook of {B}rownian motion---facts and formulae}
    \verb{url}
    \verb https://doi.org/10.1007/978-3-0348-8163-0
    \endverb
    \field{year}{2002}
  \endentry

  \entry{Bu58}{article}{}
    \name{author}{1}{}{%
      {{hash=BRC}{%
         family={Buck},
         family_i={B\bibinitperiod},
         given={R.\bibnamedelima Creighton},
         given_i={R\bibinitperiod\bibinitdelim C\bibinitperiod},
      }}%
    }
    \list{publisher}{1}{%
      {University of Michigan, Department of Mathematics}%
    }
    \strng{namehash}{BRC1}
    \strng{fullhash}{BRC1}
    \field{labelnamesource}{author}
    \field{labeltitlesource}{title}
    \field{labelalpha}{Buc58}
    \field{sortinit}{B}
    \field{sortinithash}{B}
    \verb{doi}
    \verb 10.1307/mmj/1028998054
    \endverb
    \field{number}{2}
    \field{pages}{95\bibrangedash 104}
    \field{title}{Bounded continuous functions on a locally compact space.}
    \field{volume}{5}
    \field{journaltitle}{Michigan Math. J.}
    \field{year}{1958}
  \endentry

  \entry{Ca11}{book}{}
    \name{author}{1}{}{%
      {{hash=vCJA}{%
         prefix={van},
         prefix_i={v\bibinitperiod},
         family={Casteren},
         family_i={C\bibinitperiod},
         given={Jan\bibnamedelima A.},
         given_i={J\bibinitperiod\bibinitdelim A\bibinitperiod},
      }}%
    }
    \list{publisher}{1}{%
      {World Scientific Publishing Co. Pte. Ltd., Hackensack, NJ}%
    }
    \strng{namehash}{CJAv1}
    \strng{fullhash}{CJAv1}
    \field{labelnamesource}{author}
    \field{labeltitlesource}{title}
    \field{labelalpha}{Cas11}
    \field{sortinit}{C}
    \field{sortinithash}{C}
    \field{isbn}{978-981-4322-18-8; 981-4322-18-0}
    \field{pages}{xviii+805}
    \field{series}{Series on Concrete and Applicable Mathematics}
    \field{title}{Markov processes, {F}eller semigroups and evolution
  equations}
    \field{volume}{12}
    \field{year}{2011}
  \endentry

  \entry{CoFlGr96}{article}{}
    \name{author}{3}{}{%
      {{hash=CJT}{%
         family={Cox},
         family_i={C\bibinitperiod},
         given={J.\bibnamedelima Theodore},
         given_i={J\bibinitperiod\bibinitdelim T\bibinitperiod},
      }}%
      {{hash=FK}{%
         family={Fleischmann},
         family_i={F\bibinitperiod},
         given={Klaus},
         given_i={K\bibinitperiod},
      }}%
      {{hash=GA}{%
         family={Greven},
         family_i={G\bibinitperiod},
         given={Andreas},
         given_i={A\bibinitperiod},
      }}%
    }
    \strng{namehash}{CJTFKGA1}
    \strng{fullhash}{CJTFKGA1}
    \field{labelnamesource}{author}
    \field{labeltitlesource}{title}
    \field{labelalpha}{CFG96}
    \field{sortinit}{C}
    \field{sortinithash}{C}
    \verb{doi}
    \verb 10.1007/BF01191911
    \endverb
    \field{issn}{1432-2064}
    \field{number}{4}
    \field{pages}{513\bibrangedash 528}
    \field{title}{Comparison of interacting diffusions and an application to
  their ergodic theory}
    \field{volume}{105}
    \field{journaltitle}{Probability Theory and Related Fields}
    \field{year}{1996}
  \endentry

  \entry{CoKu15}{article}{}
    \name{author}{2}{}{%
      {{hash=CC}{%
         family={Costantini},
         family_i={C\bibinitperiod},
         given={Cristina},
         given_i={C\bibinitperiod},
      }}%
      {{hash=KT}{%
         family={Kurtz},
         family_i={K\bibinitperiod},
         given={Thomas},
         given_i={T\bibinitperiod},
      }}%
    }
    \strng{namehash}{CCKT1}
    \strng{fullhash}{CCKT1}
    \field{labelnamesource}{author}
    \field{labeltitlesource}{title}
    \field{labelalpha}{CK15}
    \field{sortinit}{C}
    \field{sortinithash}{C}
    \verb{doi}
    \verb 10.1214/EJP.v20-3624
    \endverb
    \field{issn}{1083-6489}
    \field{pages}{no. 67, 1\bibrangedash 27}
    \field{title}{Viscosity methods giving uniqueness for martingale problems}
    \field{volume}{20}
    \field{journaltitle}{Electron. J. Probab.}
    \field{year}{2015}
  \endentry

  \entry{ChWa93}{article}{}
    \name{author}{2}{}{%
      {{hash=CMF}{%
         family={Chen},
         family_i={C\bibinitperiod},
         given={Mu-Fa},
         given_i={M\bibinitperiod-F\bibinitperiod},
      }}%
      {{hash=WFY}{%
         family={Wang},
         family_i={W\bibinitperiod},
         given={Feng-Yu},
         given_i={F\bibinitperiod-Y\bibinitperiod},
      }}%
    }
    \strng{namehash}{CMFWFY1}
    \strng{fullhash}{CMFWFY1}
    \field{labelnamesource}{author}
    \field{labeltitlesource}{title}
    \field{labelalpha}{CW93}
    \field{sortinit}{C}
    \field{sortinithash}{C}
    \field{number}{95}
    \field{pages}{421\bibrangedash 428}
    \field{title}{On order-preservation and positive correlations for
  multidimensional diffusion processes}
    \field{journaltitle}{Probab.~Theory Relat.~Fields}
    \field{year}{1993}
  \endentry

  \entry{Da68b}{article}{}
    \name{author}{1}{}{%
      {{hash=DDJ}{%
         family={Daley},
         family_i={D\bibinitperiod},
         given={D.\bibnamedelima J.},
         given_i={D\bibinitperiod\bibinitdelim J\bibinitperiod},
      }}%
    }
    \strng{namehash}{DDJ1}
    \strng{fullhash}{DDJ1}
    \field{labelnamesource}{author}
    \field{labeltitlesource}{title}
    \field{labelalpha}{Dal68}
    \field{sortinit}{D}
    \field{sortinithash}{D}
    \verb{doi}
    \verb 10.1007/BF00531852
    \endverb
    \field{issn}{1432-2064}
    \field{number}{4}
    \field{pages}{305\bibrangedash 317}
    \field{title}{Stochastically monotone Markov Chains}
    \verb{url}
    \verb http://dx.doi.org/10.1007/BF00531852
    \endverb
    \field{volume}{10}
    \field{journaltitle}{Zeitschrift f{\"u}r Wahrscheinlichkeitstheorie und
  Verwandte Gebiete}
    \field{year}{1968}
  \endentry

  \entry{DaSz2006}{article}{}
    \name{author}{2}{}{%
      {{hash=DH}{%
         family={Daduna},
         family_i={D\bibinitperiod},
         given={Hans},
         given_i={H\bibinitperiod},
      }}%
      {{hash=SR}{%
         family={Szekli},
         family_i={S\bibinitperiod},
         given={Ryszard},
         given_i={R\bibinitperiod},
      }}%
    }
    \strng{namehash}{DHSR1}
    \strng{fullhash}{DHSR1}
    \field{labelnamesource}{author}
    \field{labeltitlesource}{title}
    \field{labelalpha}{DS06}
    \field{sortinit}{D}
    \field{sortinithash}{D}
    \verb{doi}
    \verb 10.1239/jap/1158784947
    \endverb
    \field{issn}{0021-9002}
    \field{number}{3}
    \field{pages}{793\bibrangedash 814}
    \field{title}{Dependence ordering for {M}arkov processes on partially
  ordered spaces}
    \verb{url}
    \verb http://dx.doi.org/10.1239/jap/1158784947
    \endverb
    \field{volume}{43}
    \field{journaltitle}{J. Appl. Probab.}
    \field{year}{2006}
  \endentry

  \entry{EK86}{book}{}
    \name{author}{2}{}{%
      {{hash=ESN}{%
         family={Ethier},
         family_i={E\bibinitperiod},
         given={Stewart\bibnamedelima N.},
         given_i={S\bibinitperiod\bibinitdelim N\bibinitperiod},
      }}%
      {{hash=KTG}{%
         family={Kurtz},
         family_i={K\bibinitperiod},
         given={Thomas\bibnamedelima G.},
         given_i={T\bibinitperiod\bibinitdelim G\bibinitperiod},
      }}%
    }
    \list{publisher}{1}{%
      {Wiley}%
    }
    \strng{namehash}{ESNKTG1}
    \strng{fullhash}{ESNKTG1}
    \field{labelnamesource}{author}
    \field{labeltitlesource}{title}
    \field{labelalpha}{EK86}
    \field{sortinit}{E}
    \field{sortinithash}{E}
    \field{title}{Markov processes: Characterization and Convergence}
    \field{year}{1986}
  \endentry

  \entry{EN00}{book}{}
    \name{author}{2}{}{%
      {{hash=EKJ}{%
         family={Engel},
         family_i={E\bibinitperiod},
         given={Klaus-Jochen},
         given_i={K\bibinitperiod-J\bibinitperiod},
      }}%
      {{hash=NR}{%
         family={Nagel},
         family_i={N\bibinitperiod},
         given={Rainer},
         given_i={R\bibinitperiod},
      }}%
    }
    \list{publisher}{1}{%
      {Springer-Verlag}%
    }
    \strng{namehash}{EKJNR1}
    \strng{fullhash}{EKJNR1}
    \field{labelnamesource}{author}
    \field{labeltitlesource}{title}
    \field{labelalpha}{EN00}
    \field{sortinit}{E}
    \field{sortinithash}{E}
    \field{title}{One-Parameter Semigroups for Linear Evolution Equations}
    \field{year}{2000}
  \endentry

  \entry{En89}{book}{}
    \name{author}{1}{}{%
      {{hash=ER}{%
         family={Engelking},
         family_i={E\bibinitperiod},
         given={Ryszard},
         given_i={R\bibinitperiod},
      }}%
    }
    \list{publisher}{1}{%
      {Heldermann Verlag, Berlin}%
    }
    \strng{namehash}{ER1}
    \strng{fullhash}{ER1}
    \field{labelnamesource}{author}
    \field{labeltitlesource}{title}
    \field{labelalpha}{Eng89}
    \field{sortinit}{E}
    \field{sortinithash}{E}
    \field{edition}{Second}
    \field{title}{General topology}
    \field{year}{1989}
  \endentry

  \entry{Ev10}{book}{}
    \name{author}{1}{}{%
      {{hash=ELC}{%
         family={Evans},
         family_i={E\bibinitperiod},
         given={Lawrence\bibnamedelima C.},
         given_i={L\bibinitperiod\bibinitdelim C\bibinitperiod},
      }}%
    }
    \list{publisher}{1}{%
      {American Mathematical Society, Providence, RI}%
    }
    \strng{namehash}{ELC1}
    \strng{fullhash}{ELC1}
    \field{labelnamesource}{author}
    \field{labeltitlesource}{title}
    \field{labelalpha}{Eva10}
    \field{sortinit}{E}
    \field{sortinithash}{E}
    \verb{doi}
    \verb 10.1090/gsm/019
    \endverb
    \field{edition}{Second}
    \field{isbn}{978-0-8218-4974-3}
    \field{pages}{xxii+749}
    \field{series}{Graduate Studies in Mathematics}
    \field{title}{Partial differential equations}
    \field{volume}{19}
    \field{year}{2010}
  \endentry

  \entry{Gr79}{book}{}
    \name{author}{1}{}{%
      {{hash=GD}{%
         family={Griffeath},
         family_i={G\bibinitperiod},
         given={David},
         given_i={D\bibinitperiod},
      }}%
    }
    \list{publisher}{1}{%
      {Springer, Berlin}%
    }
    \strng{namehash}{GD1}
    \strng{fullhash}{GD1}
    \field{labelnamesource}{author}
    \field{labeltitlesource}{title}
    \field{labelalpha}{Gri79}
    \field{sortinit}{G}
    \field{sortinithash}{G}
    \field{isbn}{3-540-09508-X}
    \field{pages}{iv+108}
    \field{series}{Lecture Notes in Mathematics}
    \field{title}{Additive and cancellative interacting particle systems}
    \field{volume}{724}
    \field{year}{1979}
  \endentry

  \entry{Ha76}{article}{}
    \name{author}{1}{}{%
      {{hash=HRG}{%
         family={Haydon},
         family_i={H\bibinitperiod},
         given={R.\bibnamedelima G.},
         given_i={R\bibinitperiod\bibinitdelim G\bibinitperiod},
      }}%
    }
    \strng{namehash}{HRG1}
    \strng{fullhash}{HRG1}
    \field{labelnamesource}{author}
    \field{labeltitlesource}{title}
    \field{labelalpha}{Hay76}
    \field{sortinit}{H}
    \field{sortinithash}{H}
    \field{number}{2}
    \field{pages}{273\bibrangedash 278}
    \field{title}{On the {S}tone-{W}eierstrass theorem for the strict and
  superstrict topologies}
    \field{volume}{59}
    \field{journaltitle}{Proc. Amer. Math. Soc.}
    \field{year}{1976}
  \endentry

  \entry{HePi91}{article}{}
    \name{author}{2}{}{%
      {{hash=HI}{%
         family={Herbst},
         family_i={H\bibinitperiod},
         given={Ira},
         given_i={I\bibinitperiod},
      }}%
      {{hash=PL}{%
         family={Pitt},
         family_i={P\bibinitperiod},
         given={Loren},
         given_i={L\bibinitperiod},
      }}%
    }
    \strng{namehash}{HIPL1}
    \strng{fullhash}{HIPL1}
    \field{labelnamesource}{author}
    \field{labeltitlesource}{title}
    \field{labelalpha}{HP91}
    \field{sortinit}{H}
    \field{sortinithash}{H}
    \verb{doi}
    \verb 10.1007/BF01312211
    \endverb
    \field{issn}{0178-8051}
    \field{number}{3}
    \field{pages}{275\bibrangedash 312}
    \field{title}{Diffusion equation techniques in stochastic monotonicity and
  positive correlations}
    \field{volume}{87}
    \field{journaltitle}{Probab. Theory Related Fields}
    \field{year}{1991}
  \endentry

  \entry{Ka1962}{article}{}
    \name{author}{1}{}{%
      {{hash=KGI}{%
         family={Kalmykov},
         family_i={K\bibinitperiod},
         given={G.\bibnamedelima I.},
         given_i={G\bibinitperiod\bibinitdelim I\bibinitperiod},
      }}%
    }
    \strng{namehash}{KGI1}
    \strng{fullhash}{KGI1}
    \field{labelnamesource}{author}
    \field{labeltitlesource}{title}
    \field{labelalpha}{Kal62}
    \field{sortinit}{K}
    \field{sortinithash}{K}
    \verb{doi}
    \verb https://doi.org/10.1137/1107043
    \endverb
    \field{issn}{0040-585X (print), 1095-7219 (electronic)}
    \field{number}{4}
    \field{pages}{456\bibrangedash 459}
    \field{title}{On the Partial Ordering of One-Dimensional {Markov}
  Processes}
    \field{volume}{7}
    \field{journaltitle}{Theory Probab. Appl.}
    \field{year}{1962}
    \warn{\item Invalid format of field 'month'}
  \endentry

  \entry{KeKe77}{article}{}
    \name{author}{2}{}{%
      {{hash=KJ}{%
         family={Keilson},
         family_i={K\bibinitperiod},
         given={Julian},
         given_i={J\bibinitperiod},
      }}%
      {{hash=KA}{%
         family={Kester},
         family_i={K\bibinitperiod},
         given={Adri},
         given_i={A\bibinitperiod},
      }}%
    }
    \strng{namehash}{KJKA1}
    \strng{fullhash}{KJKA1}
    \field{labelnamesource}{author}
    \field{labeltitlesource}{title}
    \field{labelalpha}{KK77}
    \field{sortinit}{K}
    \field{sortinithash}{K}
    \field{issn}{0304-4149}
    \field{number}{3}
    \field{pages}{231\bibrangedash 241}
    \field{title}{Monotone matrices and monotone {M}arkov processes}
    \verb{url}
    \verb https://www.sciencedirect.com/science/article/pii/0304414977900333
    \endverb
    \field{volume}{5}
    \field{journaltitle}{Stochastic Processes Appl.}
    \field{year}{1977}
  \endentry

  \entry{Kr16}{article}{}
    \name{author}{1}{}{%
      {{hash=KRC}{%
         family={Kraaij},
         family_i={K\bibinitperiod},
         given={Richard\bibnamedelima C.},
         given_i={R\bibinitperiod\bibinitdelim C\bibinitperiod},
      }}%
    }
    \strng{namehash}{KRC1}
    \strng{fullhash}{KRC1}
    \field{labelnamesource}{author}
    \field{labeltitlesource}{title}
    \field{labelalpha}{Kra16}
    \field{sortinit}{K}
    \field{sortinithash}{K}
    \verb{doi}
    \verb 10.1007/s00233-015-9689-1
    \endverb
    \field{issn}{1432-2137}
    \field{number}{1}
    \field{pages}{158\bibrangedash 185}
    \field{title}{Strongly continuous and locally equi-continuous semigroups on
  locally convex spaces}
    \field{volume}{92}
    \field{journaltitle}{Semigroup Forum}
    \field{year}{2016}
  \endentry

  \entry{Kr96}{book}{}
    \name{author}{1}{}{%
      {{hash=KN}{%
         family={Krylov},
         family_i={K\bibinitperiod},
         given={N.V.},
         given_i={N\bibinitperiod},
      }}%
    }
    \list{publisher}{1}{%
      {American Mathematical Society}%
    }
    \strng{namehash}{KN1}
    \strng{fullhash}{KN1}
    \field{labelnamesource}{author}
    \field{labeltitlesource}{title}
    \field{labelalpha}{Kry96}
    \field{sortinit}{K}
    \field{sortinithash}{K}
    \field{isbn}{0-8218-0569-X}
    \field{title}{Lectures on Elliptic and Parabolic equations in H{\"o}lder
  spaces}
    \field{year}{1996}
  \endentry

  \entry{Ku09}{article}{}
    \name{author}{1}{}{%
      {{hash=KM}{%
         family={Kunze},
         family_i={K\bibinitperiod},
         given={Markus},
         given_i={M\bibinitperiod},
      }}%
    }
    \strng{namehash}{KM1}
    \strng{fullhash}{KM1}
    \field{labelnamesource}{author}
    \field{labeltitlesource}{title}
    \field{labelalpha}{Kun09}
    \field{sortinit}{K}
    \field{sortinithash}{K}
    \field{number}{3}
    \field{pages}{540\bibrangedash 560}
    \field{title}{Continuity and equicontinuity of semigroups on norming dual
  pairs}
    \field{volume}{79}
    \field{journaltitle}{Semigroup Forum}
    \field{year}{2009}
  \endentry

  \entry{Ko69}{book}{}
    \name{author}{1}{}{%
      {{hash=KG}{%
         family={K{\"o}the},
         family_i={K\bibinitperiod},
         given={Gottfried},
         given_i={G\bibinitperiod},
      }}%
    }
    \list{publisher}{1}{%
      {Springer-Verlag}%
    }
    \strng{namehash}{KG1}
    \strng{fullhash}{KG1}
    \field{labelnamesource}{author}
    \field{labeltitlesource}{title}
    \field{labelalpha}{K{\"o}t69}
    \field{sortinit}{K}
    \field{sortinithash}{K}
    \field{title}{Topological Vector Spaces I}
    \field{year}{1969}
  \endentry

  \entry{Ko79}{book}{}
    \name{author}{1}{}{%
      {{hash=KG}{%
         family={K{\"o}the},
         family_i={K\bibinitperiod},
         given={Gottfried},
         given_i={G\bibinitperiod},
      }}%
    }
    \list{publisher}{1}{%
      {Springer-Verlag}%
    }
    \strng{namehash}{KG1}
    \strng{fullhash}{KG1}
    \field{labelnamesource}{author}
    \field{labeltitlesource}{title}
    \field{labelalpha}{K{\"o}t79}
    \field{sortinit}{K}
    \field{sortinithash}{K}
    \field{title}{Topological Vector Spaces II}
    \field{year}{1979}
  \endentry

  \entry{Li85}{book}{}
    \name{author}{1}{}{%
      {{hash=LTM}{%
         family={Liggett},
         family_i={L\bibinitperiod},
         given={Thomas\bibnamedelima M.},
         given_i={T\bibinitperiod\bibinitdelim M\bibinitperiod},
      }}%
    }
    \list{publisher}{1}{%
      {Springer-Verlag}%
    }
    \strng{namehash}{LTM1}
    \strng{fullhash}{LTM1}
    \field{labelnamesource}{author}
    \field{labeltitlesource}{title}
    \field{labelalpha}{Lig85}
    \field{sortinit}{L}
    \field{sortinithash}{L}
    \field{title}{Interacting Particle Systems}
    \field{year}{1985}
  \endentry

  \entry{Ma87}{article}{}
    \name{author}{1}{}{%
      {{hash=MWA}{%
         family={Massey},
         family_i={M\bibinitperiod},
         given={William\bibnamedelima A.},
         given_i={W\bibinitperiod\bibinitdelim A\bibinitperiod},
      }}%
    }
    \strng{namehash}{MWA1}
    \strng{fullhash}{MWA1}
    \field{labelnamesource}{author}
    \field{labeltitlesource}{title}
    \field{labelalpha}{Mas87}
    \field{sortinit}{M}
    \field{sortinithash}{M}
    \verb{doi}
    \verb 10.1287/moor.12.2.350
    \endverb
    \field{issn}{0364-765X}
    \field{number}{2}
    \field{pages}{350\bibrangedash 367}
    \field{title}{Stochastic orderings for {M}arkov processes on partially
  ordered spaces}
    \verb{url}
    \verb http://dx.doi.org/10.1287/moor.12.2.350
    \endverb
    \field{volume}{12}
    \field{journaltitle}{Math. Oper. Res.}
    \field{year}{1987}
  \endentry

  \entry{MuSt02}{book}{}
    \name{author}{2}{}{%
      {{hash=MA}{%
         family={M{\"u}ller},
         family_i={M\bibinitperiod},
         given={Alfred},
         given_i={A\bibinitperiod},
      }}%
      {{hash=SD}{%
         family={Stoyan},
         family_i={S\bibinitperiod},
         given={Dietrich},
         given_i={D\bibinitperiod},
      }}%
    }
    \list{publisher}{1}{%
      {Wiley}%
    }
    \strng{namehash}{MASD1}
    \strng{fullhash}{MASD1}
    \field{labelnamesource}{author}
    \field{labeltitlesource}{title}
    \field{labelalpha}{MS02}
    \field{sortinit}{M}
    \field{sortinithash}{M}
    \field{title}{Comparison methods for stochastic models and risks}
    \field{volume}{389}
    \field{year}{2002}
  \endentry

  \entry{RuScWo2015}{article}{}
    \name{author}{3}{}{%
      {{hash=RL}{%
         family={R{\"u}schendorf},
         family_i={R\bibinitperiod},
         given={Ludger},
         given_i={L\bibinitperiod},
      }}%
      {{hash=SA}{%
         family={Schnurr},
         family_i={S\bibinitperiod},
         given={Alexander},
         given_i={A\bibinitperiod},
      }}%
      {{hash=WV}{%
         family={Wolf},
         family_i={W\bibinitperiod},
         given={Viktor},
         given_i={V\bibinitperiod},
      }}%
    }
    \list{publisher}{1}{%
      {Cambridge University Press}%
    }
    \strng{namehash}{RLSAWV1}
    \strng{fullhash}{RLSAWV1}
    \field{labelnamesource}{author}
    \field{labeltitlesource}{title}
    \field{labelalpha}{RSW16}
    \field{sortinit}{R}
    \field{sortinithash}{R}
    \field{number}{4}
    \field{pages}{1015\bibrangedash 1044}
    \field{title}{Comparison of time-inhomogeneous Markov processes}
    \field{volume}{48}
    \field{journaltitle}{Advances in Applied Probability}
    \field{year}{2016}
  \endentry

  \entry{RuWo2011}{article}{}
    \name{author}{2}{}{%
      {{hash=RL}{%
         family={R{\"u}schendorf},
         family_i={R\bibinitperiod},
         given={Ludger},
         given_i={L\bibinitperiod},
      }}%
      {{hash=WV}{%
         family={Wolf},
         family_i={W\bibinitperiod},
         given={Viktor},
         given_i={V\bibinitperiod},
      }}%
    }
    \strng{namehash}{RLWV1}
    \strng{fullhash}{RLWV1}
    \field{labelnamesource}{author}
    \field{labeltitlesource}{title}
    \field{labelalpha}{RW11}
    \field{sortinit}{R}
    \field{sortinithash}{R}
    \verb{doi}
    \verb 10.1524/stnd.2011.1068
    \endverb
    \field{issn}{0721-2631}
    \field{number}{2}
    \field{pages}{151\bibrangedash 168}
    \field{title}{Comparison of {M}arkov processes via infinitesimal
  generators}
    \verb{url}
    \verb http://dx.doi.org/10.1524/stnd.2011.1068
    \endverb
    \field{volume}{28}
    \field{journaltitle}{Statist. Decisions}
    \field{year}{2011}
  \endentry

  \entry{Se72}{article}{}
    \name{author}{1}{}{%
      {{hash=SFD}{%
         family={Sentilles},
         family_i={S\bibinitperiod},
         given={F.\bibnamedelima Dennis},
         given_i={F\bibinitperiod\bibinitdelim D\bibinitperiod},
      }}%
    }
    \strng{namehash}{SFD1}
    \strng{fullhash}{SFD1}
    \field{labelnamesource}{author}
    \field{labeltitlesource}{title}
    \field{labelalpha}{Sen72}
    \field{sortinit}{S}
    \field{sortinithash}{S}
    \field{pages}{311\bibrangedash 336}
    \field{title}{Bounded continuous functions on a completely regular space}
    \field{volume}{168}
    \field{journaltitle}{Trans. Amer. Math. Soc.}
    \field{year}{1972}
  \endentry

  \entry{So15}{inproceedings}{}
    \name{author}{1}{}{%
      {{hash=SA}{%
         family={Sootla},
         family_i={S\bibinitperiod},
         given={Aivar},
         given_i={A\bibinitperiod},
      }}%
    }
    \strng{namehash}{SA1}
    \strng{fullhash}{SA1}
    \field{labelnamesource}{author}
    \field{labeltitlesource}{title}
    \field{labelalpha}{Soo15}
    \field{sortinit}{S}
    \field{sortinithash}{S}
    \field{booktitle}{American Control Conference (ACC), 2015}
    \field{pages}{3144\bibrangedash 3149}
    \field{title}{On monotonicity and propagation of order properties}
    \field{year}{2015}
  \endentry

  \entry{Su72}{article}{}
    \name{author}{1}{}{%
      {{hash=SWH}{%
         family={Summers},
         family_i={S\bibinitperiod},
         given={W.\bibnamedelima H.},
         given_i={W\bibinitperiod\bibinitdelim H\bibinitperiod},
      }}%
    }
    \strng{namehash}{SWH1}
    \strng{fullhash}{SWH1}
    \field{labelnamesource}{author}
    \field{labeltitlesource}{title}
    \field{labelalpha}{Sum72}
    \field{sortinit}{S}
    \field{sortinithash}{S}
    \field{issn}{0002-9939}
    \field{pages}{507\bibrangedash 514}
    \field{title}{Separability in the strict and substrict topologies}
    \field{volume}{35}
    \field{journaltitle}{Proc. Amer. Math. Soc.}
    \field{year}{1972}
  \endentry

  \entry{SV79}{book}{}
    \name{author}{2}{}{%
      {{hash=SDW}{%
         family={Stroock},
         family_i={S\bibinitperiod},
         given={Daniel\bibnamedelima W.},
         given_i={D\bibinitperiod\bibinitdelim W\bibinitperiod},
      }}%
      {{hash=VSRS}{%
         family={Varadhan},
         family_i={V\bibinitperiod},
         given={S.\bibnamedelima R.\bibnamedelima Srinivasa},
         given_i={S\bibinitperiod\bibinitdelim R\bibinitperiod\bibinitdelim
  S\bibinitperiod},
      }}%
    }
    \list{publisher}{1}{%
      {Springer-Verlag, Berlin-New York}%
    }
    \strng{namehash}{SDWVSRS1}
    \strng{fullhash}{SDWVSRS1}
    \field{labelnamesource}{author}
    \field{labeltitlesource}{title}
    \field{labelalpha}{SV79}
    \field{sortinit}{S}
    \field{sortinithash}{S}
    \field{isbn}{3-540-90353-4}
    \field{pages}{xii+338}
    \field{series}{Grundlehren der Mathematischen Wissenschaften [Fundamental
  Principles of Mathematical Sciences]}
    \field{title}{Multidimensional diffusion processes}
    \field{volume}{233}
    \field{year}{1979}
  \endentry

  \entry{We68}{article}{}
    \name{author}{1}{}{%
      {{hash=WJH}{%
         family={Webb},
         family_i={W\bibinitperiod},
         given={J.\bibnamedelima H.},
         given_i={J\bibinitperiod\bibinitdelim H\bibinitperiod},
      }}%
    }
    \strng{namehash}{WJH1}
    \strng{fullhash}{WJH1}
    \field{labelnamesource}{author}
    \field{labeltitlesource}{title}
    \field{labelalpha}{Web68}
    \field{sortinit}{W}
    \field{sortinithash}{W}
    \field{issue}{02}
    \field{pages}{341\bibrangedash 364}
    \field{title}{Sequential convergence in locally convex spaces}
    \field{volume}{64}
    \field{journaltitle}{Mathematical Proceedings of the Cambridge
  Philosophical Society}
    \field{month}{04}
    \field{year}{1968}
  \endentry

  \entry{Wi81}{article}{}
    \name{author}{1}{}{%
      {{hash=WA}{%
         family={Wilansky},
         family_i={W\bibinitperiod},
         given={Albert},
         given_i={A\bibinitperiod},
      }}%
    }
    \list{publisher}{1}{%
      {Hindawi Publishing Corporation, New York, NY}%
    }
    \strng{namehash}{WA1}
    \strng{fullhash}{WA1}
    \field{labelnamesource}{author}
    \field{labeltitlesource}{title}
    \field{labelalpha}{Wil81}
    \field{sortinit}{W}
    \field{sortinithash}{W}
    \verb{doi}
    \verb 10.1155/S0161171281000021
    \endverb
    \field{pages}{39\bibrangedash 53}
    \field{title}{Mazur spaces}
    \field{volume}{4}
    \field{journaltitle}{{Int. J. Math. Math. Sci.}}
    \field{year}{1981}
  \endentry

  \entry{Wi61}{article}{}
    \name{author}{1}{}{%
      {{hash=WA}{%
         family={Wiweger},
         family_i={W\bibinitperiod},
         given={A.},
         given_i={A\bibinitperiod},
      }}%
    }
    \strng{namehash}{WA2}
    \strng{fullhash}{WA2}
    \field{labelnamesource}{author}
    \field{labeltitlesource}{title}
    \field{labelalpha}{Wiw61}
    \field{sortinit}{W}
    \field{sortinithash}{W}
    \field{pages}{47\bibrangedash 68}
    \field{title}{Linear spaces with mixed topology}
    \field{volume}{20}
    \field{journaltitle}{Studia Math.}
    \field{year}{1961}
  \endentry

  \entry{Cr16}{article}{}
    \name{author}{1}{}{%
      {{hash=CD}{%
         family={{Criens}},
         family_i={C\bibinitperiod},
         given={D.},
         given_i={D\bibinitperiod},
      }}%
    }
    \keyw{Mathematics - Probability, 60G51, 60G44}
    \strng{namehash}{CD1}
    \strng{fullhash}{CD1}
    \field{labelnamesource}{author}
    \field{labeltitlesource}{title}
    \field{labelalpha}{{Cri}17}
    \field{sortinit}{{C}}
    \field{sortinithash}{{C}}
    \verb{eprint}
    \verb 1606.04993
    \endverb
    \field{title}{{Monotone and Convex Stochastic Orders for Processes with
  Independent Increments}}
    \field{journaltitle}{ArXiv e-prints}
    \field{eprinttype}{arXiv}
    \field{eprintclass}{math.PR}
    \field{year}{2017}
  \endentry
\endsortlist
\endinput
  \blx@bblend
  \endgroup
  \csnumgdef{blx@labelnumber@\the\c@refsection}{0}}
\makeatother

\begin{document}


\title{A generator approach to stochastic monotonicity and propagation of order}

\author{
\renewcommand{\thefootnote}{\arabic{footnote}}
Richard C. Kraaij  and Moritz Schauer
}

\footnotetext[1]{
Fakultät für Mathematik, Ruhr-University of Bochum, Postfach 102148, 
44721 Bochum, Germany, E-mail: \texttt{richard.kraaij@rub.de},
{\orcidicon\!\url{http://orcid.org/0000-0001-9152-9943}}.
}

\footnotetext[2]{
Mathematical Institute Leiden University, Niels Bohrweg 1, 2333 CA Leiden, The Netherlands, 
E-mail: \texttt{m.r.schauer@math.leidenuniv.nl},\\
{\orcidicon\url{http://orcid.org/0000-0003-3310-7915}}.
}

\maketitle

\begin{abstract}
%

We study stochastic monotonicity and propagation of order for Markov processes with respect to stochastic integral orders characterized by cones of functions satisfying $\Phi f \geq 0$ for some linear operator $\Phi$.

We introduce a new functional analytic technique based on the generator $A$ of a Markov process and its resolvent. We show that the existence of an operator $B$ with positive resolvent such that $\Phi A - B \Phi$ is a positive operator for a large enough class of functions implies stochastic monotonicity. This establishes a technique for proving stochastic monotonicity and propagation of order that can be applied in a wide range of settings including various orders for diffusion processes with or without boundary conditions and orders for discrete interacting particle systems.

\null 

\end{abstract}




\section{Introduction}

A class of functions $\cF_+$ defines an \textit{integral stochastic order}, i.e.~a partial order $\preccurlyeq$ on the space of random variables, by defining $\eta \preccurlyeq \xi$ if 
\[
\forall \, f \in \cF_{+} \colon \qquad  \bE f(\eta) \le \bE f(\xi).
\]
Such orders can characterize different aspects of random variables depending on the class of test functions $\cF_+$. An example of a stochastic order comparing a notion of ``size'' is the usual stochastic order generated by the cone of increasing functions. In contrast, an example for a stochastic order comparing a notion of ``risk'' is the supermodular stochastic order generated by the cone of supermodular functions.

In the theory of Markov processes we are interested in the \textit{propagation of order}: does $\zeta(0) \preccurlyeq \xi(0)$ imply that $\zeta(t) \preccurlyeq \xi(t)$ for later times $t$, in other words are such orderings are preserved in time?
\begin{sloppypar}
To answer this question, we first study the related property of \textit{stochastic monotonicity}; whether the semigroup $\{S(t)\}_{t \geq 0}$ of a single stochastic process ${S(t) f(x) } = {\bE[f(\eta(t)) \, | \, \eta(0) = x]}$ maps $\cF_+$ into itself. In particular, one can show that we have propagation of order for $\zeta$ and $\xi$ if one can find a stochastically monotone process $\eta$ that is in between $\zeta$ and $\xi$ in some appropriate sense. We give a discussion on the literature on the related notions of stochastic monotonicity and propagation of order below.
\end{sloppypar}

To study stochastic monotonicity, we introduce a new functional analytic technique based on an operator $\Phi$ that characterizes the order, the generator $A$ of $\{S(t)\}_{t \geq 0}$ and its resolvent $R(\lambda,A) = (\bONE - \lambda A)^{-1}$. We assume that the elements of the class $\cF_+$ can be characterized as the closure of a smaller class $\cF_{+,0}$ of elements which get mapped into the positive cone of a partially ordered space $F$ by an operator $\Phi \colon  \cD(\Phi) \subseteq C_b(\cX) \rightarrow F$.
The existence of an operator $B$ on $F$ with positive resolvents $R(\lambda,B)$ such that $\Phi A - B \Phi$ maps $\cF_{+,0}$ into the positive cone, leads by formal calculations to 
\begin{equation} \label{eqn:commutation_intro}
\Phi S(t) f \geq \e^{Bt} \Phi f,
\end{equation}
implying that $S(t) \cF_+ \subseteq \cF_+$. In general, \eqref{eqn:commutation_intro} is too much to ask for as $\Phi$ is possibly non-continuous, but the inclusion result is true and only depends on properties of the resolvents and not on those of the semigroups.

	Apart from that the method is based on manipulations based on the resolvent, an object which is generally better behaved than the semigroup, our approach introduces two novel ideas to the study of stochastic monotonicity and propagation of order via infinitesimal operators:
	\begin{enumerate}[(a)]
		\item A proper treatment of the domains of the operators $A$ and $B$,
		\item The use of the strict topology $\beta$ on the space of bounded continuous functions.
	\end{enumerate}
	The treatment of the fine details of the domains of $A$ and $B$ is a key step to widening the scope of applications of the generator based approach. This is illustrated for example by the class of one-dimensional diffusion processes with boundaries. Various processes come with the same type of generator $A$, only differing in the details of the boundary conditions: killing, reflecting, absorbing or combinations of these options. We will indeed see in Section \ref{section:one_d_diffusion} that for one-dimensional diffusion processes, the boundary conditions are of crucial importance when establishing stochastic monotonicity and propagation of order.
	
	\smallskip
	
	Regarding (b): when considering Markov processes on a Polish space $\cX$, it is well known that Markov transition semigroups are usually not strongly continuous with respect to the supremum norm on $C_b(\cX)$, an issue that already turns up for one-dimensional Brownian motion. Instead of working around this issue by various approximation procedures, we will adopt a more natural point of view. We will work with the \textit{strict} topology $\beta$ on $C_b(\cX)$, which we consider to be the natural generalization to the setting of non-compact Polish spaces $\cX$ of the norm topology from the case when $\cX$ is compact.
	
	We treat the strict topology in Appendix \ref{appendix:locally_solid_space}, where we show that a Markov process corresponding to a well-posed martingale problem has a \textit{strongly continuous and locally equi-continuous (SCLE)} transition semigroup, therefore solving the issues raised in \cite[p. 17]{RuScWo2015} for using the strict topology in this setting.
	
	This change allows us to give a global perspective on the generator $A$ which was not possible for semigroups on $C_0(\cX)$, when $\cX$ is only locally compact. For example, we can include a non-trivial class of increasing functions in the domain of a diffusion operator. This is of interest to avoid the situation of proving ${S(t) \cF_+ \subseteq \cF_+}$ while ${\cD(A) \cap \cF_+ = \emptyset}$. To conclude, the strict topology takes care of steps that classically would be treated via local approximation procedures. 

\smallskip

As a consequence of this new approach, our methods are applicable in a wide variety of settings, including those where boundary conditions play a  crucial role. The scope is illustrated by the wide range of Markov processes and orders for which we can prove stochastic monotonicity and propagation of order:  one-dimensional diffusion processes with boundary conditions, multi-dimensional diffusion processes, and discrete interacting particle systems. In particular in the first setting, the setting where boundary conditions play a crucial role, we are able to obtain entirely new results. Also in the other settings, we are able to improve upon known results.

\bigskip

A short review of related literature follows.
\smallskip

\cite{Da68b} introduced the concept of stochastic monotonicity for Markov chains on linearly ordered spaces and obtained first results, for example a comparison theorem. \cite{Da68b} also give an overview over previous (implicit) uses of the idea. In the same direction go \cite{KeKe77}, considering also continuous time Markov chains.

Stochastic monotonicity was generalized to countable partially ordered spaces and stochastic orders  other than the natural one implied by the partial order on the state space by \cite{Ma87}.

The general technique to obtain comparison results for Markov processes based on the stochastic monotonicity of an auxiliary process was recognized by \cite{Da68b}. Even earlier, \cite{Ka1962} obtained similar results. Relatedly \cite{Ma87} obtained comparison results from considerations on the infinitesimal generators.

Stochastic monotonicity results for specific processes related to this work are criteria for attractive particle systems in \cite{Gr79} and criteria for diffusion processes in \cite{HePi91}, both for the usual stochastic order. \cite{ChWa93}, building up on \cite{HePi91}, find sufficient and necessary conditions for diffusion processes to preserve orderedness with respect the usual stochastic order in time.
 
Regarding other stochastic orders, we mention the following results. A stochastic monotonicity result for supermodular functions  for interacting scalar diffusion processes was obtained by \cite{CoFlGr96}. 
\cite{BeRu2006} and \cite{BeRu2007} prove propagation of convexity properties for some semi-martingale models and preservation of various order cones for spatially homogeneous processes and give an extensive characterization of comparison results for semi-martingales which can be derived from stochastic monotonicity properties.
\cite{So15} characterizes monotonicity properties of deterministic ordinary differential equations with respect to different orders.

\cite{DaSz2006} consider dependence orders such as the concordance order and the supermodular order for stochastically monotone Markov processes with partially ordered Polish state spaces using the corresponding infinitesimal generators for continuous-time processes. Lévy processes in connection with the supermodular and concordance order were considered in \cite{BaBlMu2008}.
\cite{Cr16} have results for processes with independent increments based on explicit construction of couplings. 
 
Focusing on comparison results in the situation where stochastic monotonicity of one of the processes was already established, there is  \cite{RuWo2011}  who derive general comparison results for integral orders based on an integral representation for strongly continuous semigroups. In \cite{RuScWo2015} they expand their approach to time-inhomogeneous strongly continuous evolution groups in a similar setting. A corresponding general approach to stochastic monotonicity is lacking.

\bigskip

Our results can be contrasted with the work \cite{HePi91}. There the authors use a similar decomposition of $\grad A f = (B + D) \grad f$ specifically for the generators of multivariate diffusion processes to obtain results with respect to the usual stochastic order which is characterized by functions with positive gradient ($\operatorname{grad}$). By the use of Trotter's theorem \[
\operatorname{grad} S(t) f = \operatorname{strong\;limit}\limits_{n\to \infty} T(t/n) \e^{(t/n) D} \operatorname{grad} f,\]
where $\{T(t)\}_{t \ge 0}$ is the product of strongly continuous positive semigroups on spaces bounded uniformly continuous scalar functions with generator $B$ and $\e^{t D}$ is a positive semigroup generated by a bounded operator on the product space. From this they conclude that $\{S(t)\}$ preserves increasing functions. We work to some degree in reverse order to establish monotonicity directly and derive a statement about the evolution of the gradient of $S_t f$ as consequence (if that is defined), and our result is not specific to diffusion processes and the usual stochastic order.

We can weaken the smoothness assumptions on the diffusion processes and can handle different stochastic orders. We do not need explicit constructions of $T_t$ or $\e^{tD}$ and can weaken the boundedness condition on $\e^{tD}$. Compared to the approximation based arguments in \cite{BeRu2007} on integral orders for semi-martingales our conditions are local (there are for example stochastic monotone smooth diffusions not covered by their arguments). Thus our method is well suited for the case of diffusion processes and can be seen as a generalization of the result of \cite{HePi91}. 
On the other hand, our method is by no means restricted to such processes, which we illustrate by an application of our method to reprove a classical result for discrete interacting particle systems.

Besides being able to handle stochastic monotonicity, our result gives a generator based characterization of propagation of stochastic orderings in a different manner than  \cite{RuWo2011}.

\bigskip
The paper is organised as follows.
The following Section \ref{section:main_results} presents our contribution to the theory of stochastic orderings for Markov processes. First, we present our main result on stochastic monotonicity with respect to an integral order of general Markov processes which give rise to SCLE semigroups. Then we show that preservation of order for Markov processes is equivalent to an abstract monotonicity property of their product semigroups for which our main theorem gives workable sufficient conditions.
Subsequently, we show that also the informal inequality \eqref{eqn:commutation_intro} can be cast as inclusion result of the form $\tilde S(t) \cG_+ \subseteq \cG_+$ on an extended space within the framework of our main theorem. 
In section \ref{section:one_d_diffusion}, we apply our results to give criteria for stochastic monotonicity of real valued diffusion processes under various boundary conditions.
In section \ref{section:Diffusion_processes}, we derive stochastic monotonicity of diffusion processes for a family of integral order using first and second order partial differential operators. In section \ref{section:Interacting_particle_systems}, we show that  our results   imply  a classical stochastic monotonicity result for particle systems. 
The proofs of the results of section \ref{section:main_results} are given in section \ref{section:proofs_main_results}.

In Appendix \ref{appendix:locally_solid_space}, we gather results on locally solid spaces, in Appendix \ref{section:semigroups_on_lcs} we gather some results on SCLE semigroups on locally convex spaces and in Appendix \ref{section:martingale_problem} we define the strict topology and consider the martingale problem for SCLE semigroups for the strict topology.

\section{Main results}\label{section:main_results}

\subsection{Preliminaries}

By $\cX, \cY$ we denote Polish spaces, their elements will be denoted in general by $x \in \cX$ and $y \in \cY$. The Skorokhod space of right-continuous trajectories $\gamma\colon\bR^+ \rightarrow \cX$  with left limits is denoted by $D_\cX(\bR^+)$. Equipped with the Skorokhod topology, cf.~\cite{EK86}, the Skorokhod space is Polish. Locally convex spaces will be denoted by $E,F$, or by  $(E,\tau_E), (F,\tau_F)$
together their topology. Random variables on $\cX,\cY$ will be denoted by Greek letters, $\eta$, $\xi$, et cetera. $\bONE$ is the identity operator.

$C_b(\cX)$ and $C_c(\cX)$ are the sets of bounded functions and of compactly supported functions on $\cX$. The space of functions on a locally compact space $\cX$ vanishing at infinity is denoted by $C_0(\cX)$.  We denote by $M_b(\cX)$ the space of bounded Borel measurable functions on $\cX$.

Next, we consider function spaces on $\bR^d$. For multi-indices $\alpha \in \bN_0^d$, we write $\vn{\alpha}_1 = \sum_{i =1}^d \alpha_i$ and $\vn{\alpha}_\infty = \max_{1 \leq i \leq d} \alpha_i$. We use the multi-index notation for partial derivatives $\partial_{\alpha }=\partial_{1}^{\alpha _{1}}\partial_{2}^{\alpha_{2}}\ldots \partial_{n}^{\alpha _{n}}$, where $\partial_i^{\alpha_i} :=\partial^{\alpha_i} / \partial x_i^{\alpha_i}$. 
The empty product or zeroth power is understood as identity and the null element of $ \bN_0^d$ is denoted by $0$.

We denote by $|z|$ the Euclidian norm of $z \in \bR^d$. For $\delta \in (0,1]$ and $u \in C_b(\bR^d)$, we write
\begin{equation*}
\vn{u}_\delta := \sup_{\substack{x,y \in \bR \\ x \neq y}} \frac{|u(x) - u(y)|}{|x-y|^\delta}
\end{equation*}
for the H\"{o}lder norm of order $\delta$. For $k \in \bN$, and $k$ times continuously differentiable functions $u$, we write
\begin{equation*}
\vn{u}_k := \max_{\alpha \colon \vn{\alpha}_\infty \leq k} \vn{\partial_{\alpha} u}, \qquad \vn{u}_{k,\delta} \mathrel{:}= \vn{u}_k + \max_{\alpha \colon \vn{\alpha}_\infty \leq k} \vn{\partial_{\alpha} u}_\delta.
\end{equation*}
$C^k_b(\bR^d)$ and $C^\infty_b(\bR^d)$ are the spaces of functions for which derivatives up to order $k$, respectively all derivatives, are continuous and bounded. $C^{k,\delta}_b(\bR^d)$ is the space of those functions in $C^k_b(\bR^d)$ whose $k$-th derivatives have finite H\"{o}lder norm of order $\delta$. We write $C^\infty_c(\bR^d) :=  C_c(\bR^d) \cap C^\infty_b(\bR^d)$.

Our main example for a locally convex space will be $(E,\tau) = (C_b(\cX),\beta)$, where $\beta$ is the locally convex \textit{strict} topology, see Appendix~\ref{section:strict_topology} on page~\pageref{section:strict_topology}. This topology, when considering sequences, satisfies $f_n \stackrel{\beta}{\rightarrow} f$ if and only if $\sup_n \vn{f_n} < \infty$ and $f_n \rightarrow f$ uniformly on compact sets. As a topology, however, it satisfies in addition many desirable properties like Stone-Weierstrass, Arzel\`{a}-Ascoli, Dini's theorem, a Riesz-representation et cetera.
\smallskip

In our paper, we will connect two basic concepts: Feller processes and integral stochastic orders. We first introduce the latter.
\begin{definition}[Integral stochastic order]
Let $\cG_+ \subseteq C_b(\cX)$ and let $\eta,\xi$ be random variables on $\cX$. We say that $\eta \preccurlyeq \xi$ if for all $f \in \cG_+$ we have
\begin{equation*}
\bE f(\eta) \leq \bE f(\xi).
\end{equation*}
We call $\cG_+$ a generating cone for the integral order.
\end{definition}

Note that $\cG_+$ and the strict closure $\overline{\cG_+}$ induce the same integral order. Also the convex hull of a generating cone generates the same order. For more discussion on various cones generating the same order, see e.g.~Chapter 2 in \cite{MuSt02}.

\smallskip
For our discussion of orders in relation to Feller processes and their generators, we specifically consider orders that are generated by an operator. 
	
Thus, let $E,F$ be locally convex spaces. We will say that $T$ is an operator from $E$ to $F$ if $T \subseteq E\times F$. We write 
\begin{equation*}
\cD(T) := \left\{f \, \middle| \, \exists g\colon  (f,g) \in T  \right\}, \qquad \cR(T) := \left\{g \, \middle| \, \exists f\colon (f,g) \in T  \right\}
\end{equation*}
for the domain and range of $T$. As our operators $T$ are in general single valued, we write $Tf := g$ if $(f,g) \in T$. If we say that $T$ is linear, it is understood that the domain is a linear space on which $T$ acts linearly.

Next consider a locally convex space $(E,\tau_E)$ and a vector space $F$ that is equipped a partial order $\leq$ and positive cone $F_+ := \left\{g \in F \, | \, g \geq 0 \right\}$. Let $\Phi \subseteq E \times F$ be a linear operator.

\begin{definition} \label{definition:positive_cone_in_terms_of_Delta}
	Set $\cF_{+,0} := \left\{f \in E \cap \cD(\Phi) \, \middle| \, \Phi f \geq 0 \right\}$ and set $\cF_+ = \overline{\cF_{+,0}}$ (the $\tau_E$-closure). We say that $f$ is `positive' (non-negative) if $f \in \cF_+$. If $(E,\tau) = (C_b(\cX),\beta)$, we say that $\Phi$ generates the stochastic order corresponding to $\cF_+$.
\end{definition}

\begin{remark}
	If $E$ itself is also equipped with a natural order $\leq$, we will denote the cone generated by this order by $E_+$. It is then good to note that $E_+$ and $\cF_+$, the cone induced by $\Phi$ and $F_+$ are not necessarily the same. This happens for example if $E = C_b(\bR) = F$ and $f \leq g$ if $f(x) \leq g(x)$ for all $x \in \bR$. If $\Phi \colon C^1_b(\bR) \rightarrow C_b(\bR)$ is given by $\Phi f(x) = f'(x)$, then $E_+$ consists of the positive functions, whereas $\cF_+ \subseteq E$ consists of the non-decreasing functions.
\end{remark}

Next, we turn to the definition of Feller processes. To do so, we first introduce a suitable notion of an operator semigroup on a locally convex space.

\begin{definition}
	A family of continuous operators $\{S(t)\}_{t \geq 0}$ on a locally convex space $E$, i.e.\ $S(t)\colon E \rightarrow E$ is called a strongly continuous locally equi-continuous (SCLE) semigroup if
	\begin{enumerate}[(a)]
		\item $S(t)S(r) = S(t+r)$,
		\item the family is strongly continuous: the map $t \mapsto S(t)x$ is continuous for each $x \in E$,
		\item the family is locally equi-continuous: for each $T > 0$, the family $\{S(t)\}_{t \leq T}$ is equi-continuous.
	\end{enumerate}
\end{definition}

In Appendix~\ref{section:semigroups_on_lcs} on page~\pageref{section:semigroups_on_lcs} we also define dissipativity, the range condition and extensions of operators and introduce further related concepts.

We can now introduce our notion of a Feller process.

\begin{definition}
	A process $\eta$ on the Skorokhod space $D_\cX(\bR^+)$ is a Markov process with respect to its natural filtration $\cF_t := {\sigma(\eta(s) \, | \, s \leq t)}$ if
	\begin{equation*}
	\bE[\eta(t) \, | \, \cF_s] = \bE[\eta(t) \, | \, \eta(s)].
	\end{equation*}
	
	We say that $\eta$ is a Feller process if it is Markov and if the semigroup of conditional expectations
	\begin{equation*}
	S(t)f(x) := \bE[f(\eta(t)) \, | \, \eta(0) = x]
	\end{equation*}
	is a SCLE semigroup for the strict topology on $C_b(\cX)$.
\end{definition}

Feller processes are usually defined on locally compact spaces, where the defining property is that the transition semigroup is strongly continuous for the supremum norm on the space of functions that vanish at infinity. In Appendix \ref{section:semigroup_on_locally_compact_space}, we show that the both definitions coincide in this particular setting showing that the above definition is a reasonable extension to the setting of Polish spaces. In addition, in Section \ref{section:martingale_problem_SCLE}, we show that solutions to well-posed martingale problems yield Feller processes, further extending the credibility of the definition above.

\smallskip

To use functional analytic techniques, we need assumptions on our locally convex spaces $E$ and $F$  and on the orders defined on them.  We introduce terminology that connects orders and topology, and additionally introduce a class of locally convex spaces that share some of the properties of Banach spaces. 

The strict topology $\beta$ on $C_b(\cX)$ for a Polish space $\cX$ and the point-wise order $\leq$ satisfy all conditions that we will introduce below.

\begin{definition}
	We say that $(F,\leq)$ is a Riesz space if the suprema and infima of any finite collection of elements in $F$ exists.
	
	\smallskip

    A subset $A \subseteq F$ is \textit{solid} if $|f| \leq |g|$ and $g \in A$ implies that $f \in A$.  $(F,\tau_F,\leq)$ is a \emph{locally convex-solid space} if $(F,\leq)$ is a Riesz space and if $(F,\tau_F)$ is a topological vector space which has a basis of closed, convex and solid neighbourhoods of $0$. 
	
\end{definition}

See Appendix \ref{appendix:locally_solid_space} for a slightly extended discussion on Riesz spaces and locally  convex-solid spaces, or see Section 6 in \cite{AB78}. The next definition and its use for semigroup theory is discussed in Appendix \ref{section:semigroups_on_lcs}.

\begin{definition} \label{def:condB_condC_intro}
A locally convex space $(E,\tau_E)$ also equipped with a norm $\vn{\cdot}_E$, denoted by $(E,\tau_E,\vn{\cdot}_E)$ satisfies \textit{Condition C} if 
\begin{enumerate}[(a)]
\item $\tau_E$ is weaker than the norm topology.
\item Both topologies have the same bounded sets.
\item[(c)] $(E,\tau_E)$ is sequentially complete.
\item[(d)] $\cN$ is countably convex: for any sequence $p_n$ of $\tau_E$ continuous semi-norms satisfying $p_n(\cdot) \leq \vn{\cdot}_E$ and $\alpha_n \geq 0$ such that $\sum_n \alpha_n = 1$, we have that $p(\cdot) := \sum_n \alpha_n p_n(\cdot)$ is $\tau_E$ continuous.
\end{enumerate}

\end{definition}

It should be noted that any Banach space satisfies Condition $C$.

\smallskip

Our main results naturally fall into two classes, we start with the propagation of monotonicity, which will be followed by results on the propagation of order in Section~\ref{section:preservation_of_order}.

\subsection{Stochastic monotonicity} \label{section:propagation_of_monotonicity}

Our main result establishes conditions that suffice to show that a space of functions $\cF_+$ is preserved under the evolution of a semigroup $\{S(t)\}_{t \geq 0}$. We refer to this property as \emph{monotonicity} of $\{S(t)\}_{t \geq 0}$ (with respect to $\cF_+$) or \emph{stochastic monotonicity} when referring to a semigroup of expectation operators. It implies that the corresponding Markov process $X$ is a $\cF_+$-monotone process in the sense of \cite{Ma87}.

The setting is summarized in the following assumption. 

\begin{assumption} \label{assumption:main_assumption}
Let $(E,\tau_E)$ and $(F,\tau_F)$ be locally convex spaces, which are additionally equipped with norms $\vn{\cdot}_E, \vn{\cdot}_F$ so that $(E,\tau_E,\vn{\cdot}_E),(F,\tau_F,\vn{\cdot}_F),$ satisfy Condition C. Additionally, suppose that $F$ is equipped with a partial order $\leq$. The following is given:
\begin{enumerate}[(a)]
\item A locally convex-solid space $(F,\tau,\leq)$.
\item A linear semigroup $\{S(t)\}_{t \geq 0}$, SCLE for $\tau_E$ with linear generator $A  \subseteq E \times E$. The resolvents $R(\lambda,A)$ are well-defined as $R(\lambda,A) := (\bONE - \lambda A)^{-1}$ for all $\lambda >0$.
\item A linear operator $\Phi \subseteq E \times F$ and the cones $\cF_{+,0} \subseteq E$ and $\cF_+ \subseteq E$ generated by $\Phi$.
\item A linear operator $B \subseteq F \times F$ such that for each $\lambda > 0$ the resolvent $R(\lambda,B) := (\bONE- \lambda B)^{-1}$ is well-defined as a positive continuous linear operator. 
\end{enumerate}
\end{assumption}

\begin{theorem} \label{theorem:main_theorem}
Let Assumption \ref{assumption:main_assumption} be satisfied. In addition to $A$ and $B$, suppose there is a third positive and $\tau_F$ continuous linear operator $C  \colon F \rightarrow F$. Let $\cD$ be a set such that $\cD \subseteq \cD(B \Phi) \cap \cD(C \Phi) \cap \cD(\Phi A)$ and suppose 
\begin{enumerate}[(a)]
\item for all $f \in \cD$ we have $\Phi Af = B \Phi f + C \Phi f$,
\item $\cD \cap \cF_{+,0}$ is $\tau_E$ dense in $\cF_{+,0}$,
\item for sufficiently small $\lambda > 0$, we have $R(\lambda,A)\cD \subseteq \cD$.
\end{enumerate}
Then for all $t \geq 0$, we have: if $g \in \cF_+$, then $S(t)g \in \cF_+$.
\end{theorem}

\begin{remark} \label{remark:remarks_main_theorem}
Condition (a) is the main condition of interest, whereas (b) and (c) are of technical nature. In Section \ref{section:proofs_main_results} we will prove a slightly more general result that allows for non-continuous $C$ and variations in the conditions (b) and (c). It is sometimes convenient to rephrase condition (a) in terms of the graphs of the operators:
\begin{enumerate}
\item[(a')] For all $(f,g) \in A$ with $f \in \cD$, we have $(\Phi f, \Phi g) \in B + C$.
\end{enumerate}
\end{remark}

\begin{remark}
We will introduce a method that can be used to verify condition (b) in Section \ref{section:approximation_operators} below.
\end{remark}

\begin{remark}
We remark that condition (c) is satisfied in a range of well known examples. Consider for example the setting that $(E,\tau_E) = (C_b(\bR^d),\beta)$ and where the operator $A$ is a diffusion operator having a smooth drift function and smooth diffusion matrix. If the space $\cD$ contains all smooth bounded functions then it follows from general theory that $R(\lambda,A) \cD \subseteq \cD$.
\end{remark}

Before proceeding, we introduce one example that can be kept in mind and which illustrates difficulties that appear when dealing with the domains of $A,B,C$ and $\Phi$.

\begin{example} \label{example:one_d_convex_order}
Consider on $\cX = [0,1]$ the operator
\begin{equation*}
\tilde{A}f(x) = \frac{1}{2}f''(x) + b(x) f'(x),
\end{equation*}
where $b$ is a function that is smooth on a neighbourhood of $[0,1]$ in $\bR$ and $b(l) \geq 0$ and $b(r) \leq 0$. As both boundaries of $[0,1]$ are regular, we can impose various kinds of boundary conditions on $\tilde{A}$ to obtain an operator $A$ that generates a diffusion process.

Suppose we consider the convex order induced by $\Phi(f) = f''$. Calculating $\Phi \tilde{A} f$ for a four times continuously differentiable function $f$ yields
\begin{equation*}
\Phi \tilde{A} f(x) = \frac{1}{2}f^{(4)}(x) + b(x) f^{(3)}(x) + 2 b'(x) f''(x) + b''(x) f'(x).
\end{equation*}
In context of Theorem~\ref{theorem:main_theorem} this suggests to impose the condition $b'' = 0$ and for $B$ to use a subset of  
\begin{equation*}
\tilde{B}f(x)  = \frac{1}{2}f''(x) + b(x) f'(x) +  2b'(x) f(x), 
\end{equation*}
i.e.~a diffusion operator with killing, and $C = 0$.

\smallskip

The major issue for the application of Theorem \ref{theorem:main_theorem} here, in contrast to the diffusion setting on $\bR^d$, as in e.g.~\cite{HePi91}, is that it is not obvious that there exist $A$, $B$, $C$ such that there is a sufficiently rich set $\cD \subseteq \cD(\Phi A) \cap \cD(B \Phi) \cap \cD(C \Phi)$. In other words, we need to find $A$, $B$, $C$ such that $\Phi$ maps the domain of an operator $A$ into that of operators B,C.

\smallskip

Feller boundary conditions for $\tilde{A}$ include conditions on $f(0),f'(0),f''(0)$ and $f(1),f'(1),f''(1)$. As $\Phi f = f''$, we restrict ourselves to the setting where
\begin{equation*}
\cD(A) := \left\{f \in C^2[0,1] \, \middle| \, f''(0) = f''(1) = 0 \right\},
\end{equation*}
which, in more familiar terms can be expressed as
\begin{equation*}
\cD(A) := \left\{f \in C^2[0,1] \, \middle| \, Af(0) = b(0)f'(0), Af(1) = b(1)f'(1) \right\},
\end{equation*}
which is the generator of Brownian motion with sticky or absorbing boundary conditions because $b(l) \geq 0$ and $b(r) \leq 0$, see Remark \ref{remark:sign_of_drift_for_second_order} below.

We are then in the setting of our theorem with $B$ a subset of $\tilde{B}$ with Dirichlet boundary conditions:
\begin{equation*}
\cD(B) := \left\{f \in C^2[0,1] \, \middle| \, f(0) = f(1) = 0 \right\}.
\end{equation*}

It then follows that $\cD := \cD(A) \, \cap \, C^4[0,1] \subseteq \cD(\Phi A) \, \cap \, \cD(B \Phi) \, \cap \, \cD(C \Phi)$  on which $\Phi A = B \Phi + C \Phi$. In addition, this domain is sufficiently large to also prove (b) and (c) of Theorem \ref{theorem:main_theorem}. We consider variants of this example in Section \ref{section:one_d_diffusion}.
\end{example}

To obtain results also in the setting where the above method does not immediately apply, e.g.\ in the case when the drift function is not smooth, we also consider an approximation result, comparable with Proposition 5.1 in \cite{HePi91}. Because $\cF_{+}$ is closed, the property that the set $\cF_+$ gets mapped into itself is stable under the convergence of semigroups. 

\begin{lemma} \label{lemma:approximation}
Let Assumption \ref{assumption:main_assumption} be satisfied. Suppose that we have a sequence of SCLE semigroups $\{S_n(t)\}_{t \geq 0}$, $n \geq 1$ and an SCLE semigroup $\{S(t)\}_{t \geq 0}$ on $E$ such that
\begin{enumerate}[(a)]
\item $S_n(t)f \rightarrow S(t)f$ for $\tau_E$ for all $t \geq 0$ and $f \in E$,
\item $S_n(t) \cF_+ \subseteq \cF_+$.
\end{enumerate}
Then we have that $S(t) \cF_+ \subseteq \cF_+$ for all $t \geq 0$.
\end{lemma}

\begin{proof}
Pick $f \in \cF_+$ and $t \geq 0$. For all  $n \geq 1$, we have by (b) that $S_n(t)f \in \cF_+$. Because $\cF_+$ is closed for $\tau_E$ topology, it follows by (a) that $S(t)f \in \cF_+$.
\end{proof}

\begin{remark}
	In the setting $(E,\tau_E) = (C_b(\cX),\beta)$ for a Polish space $\cX$ convergence of semigroups in the strict topology can be established via the martingale problem method, see e.g.~\cite{SV79}, and its connection to the strict topology, c.f.~Appendix \ref{section:martingale_problem}. Analogously, one can proceed via Trotter-Kato approximation type theorems, see for example Theorem 1.6.9 in \cite{EK86}.
\end{remark}

The lemma and our main theorem give the following immediate corollary. It is particularly useful in a setting where the verification of (c) of Theorem \ref{theorem:main_theorem} poses difficulties. 

\begin{corollary} \label{corollary:main}
Suppose we have semigroups $\{S_n(t)\}_{t \geq 0}$ with generators $A_n \subseteq E \times E$ for $n \geq 1$ such that there are operators $B_n \subseteq F \times F, \, C_n \subseteq F \times F$ that satisfy the conditions of Theorem \ref{theorem:main_theorem}. Suppose additionally that $S_n(t)f \rightarrow S(t)f$ for $\tau_E$ for all $t \geq 0$ and $f \in E$.

\smallskip

Then for $g \in \cF_+$ also $S(t)g \in \cF_+$ for all $t \geq 0$.
\end{corollary}

\begin{remark}
We do not focus on necessary conditions in this paper, but note that an argument used by \cite{HePi91} to obtain necessary conditions for monotonicity for diffusion semigroups generalizes to some extend.
Assume that $F$ is of the form 
 $F = C_b(\cY, \bR^n)$.
Their argument applies if 
$\Phi$ is a spatial derivative operator which commutes with taking the time derivative in $\frac{\partial }{\partial t} S(t) f = A f$. If for $y \in \cY$, $f \in \cD$
\[
(\Phi A f)(y) = \lim_{t \to 0} \frac{(\Phi S_t f - \Phi f)(y)}{t}
\]
holds, then monotonicity, i.e.\ $\{S(t)\}_{t\ge0}$ preserves $\cF_{+}$, implies
\[
(\Phi f)(y) = 0 \quad \Longrightarrow \quad (\Phi A f)(y) \ge 0
\]
for $f \in \cD \cap \cF_{+,0}$. At least in the context of \cite{HePi91} this equation applied to skilfully chosen test functions together with our main theorem completely characterises monotone diffusions with smooth coefficients.
\end{remark}

\subsection{A lower bound for monotone semigroups} \label{section:lower_bound}

Next, we turn our attention back to the point raised in the introduction, namely when $\Phi A - B \Phi \geq 0$ implies the stronger claim $\Phi S(t)f \geq \e^{Bt} \Phi f$.
With minimal additional assumptions, we can derive such a strengthening of Theorem \ref{theorem:main_theorem} almost as a corollary.

Let $\Upsilon\colon E \times F \rightarrow F^2$ with domain $\cD(\Upsilon) := \cD(\Phi) \times F$ be given in matrix form by
\begin{equation} \label{eqn:def_Delta_lower_bound}
\Upsilon = \begin{pmatrix} \Phi &-\bONE\\ 0 & \bONE \end{pmatrix}.
\end{equation}
The cone $\cG_{+,0}$ generated by $\Upsilon$ is given by pairs of functions $(f,h) \in \cD(\Upsilon)$ that satisfy $\Phi f \geq h \geq 0$ in $F$.

In addition to the assumption that we have a SCLE semigroup  $\{S(t)\}_{t \geq 0}$ on the space $E$ with generator $A$, we now also assume that there is a positive SCLE semigroup $\{T(t)\}_{t \geq 0}$ on $F$ with generator $B$. It follows that the product semigroup $\cS(t) = S(t) \times T(t)$ is SCLE on the cartesian product $E \times F$. 

If we assume that $\Phi A f = (B+C)\Phi f$, then we obtain
\begin{equation*}
\Upsilon
\underbrace{
\begin{pmatrix} A &  0 \\ 0 & B  \end{pmatrix}
}_{\displaystyle  =:\; \cA} \begin{pmatrix} f\\h
\end{pmatrix}
= 
\underbrace{\begin{pmatrix}
B & 0 \\ 0 & B
\end{pmatrix}}_{\displaystyle =:\;\cB} \Upsilon \begin{pmatrix} f \\h
\end{pmatrix} 
+ 
\underbrace{
\begin{pmatrix}
C & C \\ 0 & 0 
\end{pmatrix}}_{\displaystyle =:\; \cC} \Upsilon \begin{pmatrix} f \\ h
\end{pmatrix},
\end{equation*}
which puts us effectively in the situation of Theorem \ref{theorem:main_theorem} as the operator $\cB$ is resolvent positive and $\cC$ is positive.

\begin{remark}
Note that if $F$ is locally solid and satisfies Condition C of Definition \ref{def:condB_condC_intro} in Appendix \ref{section:semigroups_on_lcs} then the statements that $B$ is resolvent positive and that $\{T(t)\}_{t \geq 0}$ is positive are equivalent.
\end{remark}

	Imposing analogues of Assumptions (b) and (c) of Theorem \ref{theorem:main_theorem} on the domain of the operator $B$ suffices to establish a commutation result for the semigroups, effectively bootstrapping Theorem \ref{theorem:main_theorem} to strengthen itself.
	
\begin{theorem}\label{theorem:lower_bound}
	Consider a positive SCLE semigroup $\{S(t)\}_{t \geq 0}$ on the space $(E,\tau_E,\vn{\cdot}_E)$ with generator $A$ and let Assumption \ref{assumption:main_assumption} be satisfied.
	
	Additionally, suppose that $B$ is the generator of a positive SCLE semigroup $\{T(t)\}_{t \geq 0}$ on $F$. Let  $\cD^* \subseteq F$ be such that $\Phi \cD \subseteq \cD^* \subseteq \cD(B) \cap \cD(C)$.
	Suppose that
	\begin{enumerate}[(a)]
		\item the conditions of Theorem \ref{theorem:main_theorem} are satisfied,		\item $(\cD \times \cD^*) \cap \cG_{+,0}$ is $\tau_E \times \tau_F$ dense in $\cG_{+,0}$,
		\item for sufficiently small $\lambda > 0$, we have $R(\lambda,B)\cD^* \subseteq \cD^*$.
	\end{enumerate}
	Then we have for all $t \geq 0$ that $(S(t)\times T(t)) \cG_+ \subseteq \cG_+$. In particular, if $f \in \cF_{+,0}$ and $S(t)f \in \cF_{+,0}$, then we have $\Phi S(t)f \geq T(t) \Phi f \geq 0$.
\end{theorem}

\subsection{Preservation of order and the comparison of two semigroups} \label{section:preservation_of_order}

We proceed with our results on the preservation of orderedness. Consider two processes $\eta^{(1)}(t),\eta^{(2)}(t)$ on some space $\cX$ with generators $A^{(1)},A^{(2)}$ on $(E,\tau_E) := (C_b(\cX),\beta)$ and an integral stochastic order $\preccurlyeq$ generated by some operator $\Phi$ with cone $\cF_+$. 

Our goal is to obtain conditions under which $\eta^{(1)}(0) \leq \eta^{(2)}(0)$ implies $\eta^{(1)}(t) \preccurlyeq \eta^{(2)}(t)$ for all $t \geq 0$. To motivate the main theorem of this section, we first start by a well-known method, using a third auxiliary process $\eta(t)$ with generator $A$, of which the semigroup $\{S(t)\}_{t \geq 0}$ preserves monotonicity. Suppose that 
\begin{enumerate}[(a)]
	\item $S(t) \cF_{+} \subseteq \cF_+$ for all $t \geq 0$ ,
	\item and $(A^{(1)} - A)f \leq 0$ and $(A^{(2)} - A)f \geq 0$ for all $f \in \cF_+ \cap \cD(A^{(1)}) \cap \cD(A^{(2)}) \cap \cD(A)$.
\end{enumerate}
We then obtain for  $f \in \cF_{+} \cap \cD(A) \cap \cD(A^{(1)}) \cap \cD(A^{(2)})$ , assuming perhaps too optimistically that $S(r)f \in  \cF_{+} \cap \cD(A) \cap \cD(A^{(1)}) \cap \cD(A^{(2)})$  for $r \geq 0$, that 
\begin{align*}
S^{(2)}(t)f - S^{(1)}(t)f & = \left(S^{(2)}(t)f - S(t)f\right) + \left(S(t)f - S^{(1)}(t)f\right) \\
& = \int_0^t \left(S^{(2)}(t-r)\left(A^{(2)} - A \right) S(r)f \right.\\&\qquad\quad \left.- S^{(1)}(t-r)\left(A^{(1)} - A \right) S(r)f \right) \dd r \\
& \geq 0.
\end{align*}

This establishes the preservation of order for $\eta^{(1)}$ and $\eta^{(2)}$ started jointly at a fixed point $x = \eta^{(1)}(0) = \eta^{(2)}$. Some more work would allow to also vary the starting point, cf.~Corollary \ref{corollary:preservation_of_order} below.

	Note that if our optimistic assumption is satisfied, this argument can be made rigorous by interpreting this integral as a Riemann integral for the strict topology. We give an alternative method to obtain the positivity of $S^{(2)}(t)f-S^{(1)}(t)f$  by exploiting a deeper connection to Theorem \ref{theorem:main_theorem}, which additionally also removes the issues regarding the assumptions on $S(t) f$.
	
	\begin{remark}
		As the result below obtains its strength from the resolvent formalism of Section \ref{section:propagation_of_monotonicity}, it is restricted to time-homogeneous Markov processes, whereas the sketched integral methods can be extended to time-inhomogeneous Markov processes, see \cite{RuScWo2015} for the context of $C_0$-evolution systems in Banach spaces.
	\end{remark}

Similar to the integral method above, our last theorem will consider a semigroup in the middle. The direct comparison, however, is replaced by the use of Theorem \ref{theorem:main_theorem}.

Let $i \in \{1,2\}$. Let $\{S^{(i)}(t)\}_{t \geq 0}$ and $\{S(t)\}_{t \geq 0}$ be SCLE semigroups on $(C_b(\cX),\beta)$ with generators $A^{(i)}$ and $A$. We will assume that 
\begin{enumerate}[(a)]
	\item the operator $A$ satisfies $\Phi A =B\Phi + C\Phi$ as in Theorem \ref{theorem:main_theorem},
	\item there is a positive continuous operator $C^{(i)} \subseteq F \times E$ and a constant $i \in \{1,2\}$ such that $(A^{(i)} - A)f = (-1)^i C^{(i)} \Phi f$. 
\end{enumerate}

Arguing in the spirit of Section \ref{section:lower_bound}, we define $\Upsilon \subseteq E^2 \times \left(F \times E\right)$, $\cD(\Upsilon) = \cD(\Phi)^2$ by $\Upsilon(f,h) = (\Phi f, (-1)^i(h-f))$. We see that
\begin{equation*}
\Upsilon
\underbrace{
	\begin{pmatrix} A & 0 \\ 0 & A^{(i)}
	\end{pmatrix}
}_{\displaystyle  =:\;\cA} \begin{pmatrix} f \\ h
\end{pmatrix}
= 
\underbrace{\begin{pmatrix}
	B & 0 \\ 0 & A^{(i)}
	\end{pmatrix}}_{\displaystyle  =:\;\cB} \Upsilon \begin{pmatrix} f \\ h
\end{pmatrix} 
+ 
\underbrace{
	\begin{pmatrix}
	C & 0 \\  C^{(i)} &0
	\end{pmatrix}}_{\displaystyle =:\; \cC} \Upsilon \begin{pmatrix} f \\ h
\end{pmatrix}
\end{equation*}
where $\cB$ is resolvent positive and $\cC$ is positive, which puts us back in the setting of $\Phi A = B \Phi + C \Phi$ but in a higher dimensional space. We can similarly extend our semigroup and generating cones. In particular, we see that $\Upsilon$ generates the cone 
\begin{equation*}
\cH_{+,0}^{(i)} := \left\{(f,h) \in E^2 \, \middle| \, f \in \cF_{+,0}, \, (-1)^i f \leq (-1)^i h \right\}.
\end{equation*}
Let $\cH_+^{(i)}$ be the closure of $\cH_{+,0}^{(i)}$ and set $\{\cS(t)\}_{t \geq 0} = \{(S(t), S^{(i)}(t))\}_{t \geq 0}$. Note that $(f,f) \in \cH_+^{(i)}$ if $f \in \cF_+$. If the conditions for Theorem \ref{theorem:main_theorem} are satisfied for the semigroup $\{S(t)\}_{t \geq 0}$ and additionally, the semigroup $\{S^{(i)}(t)\}_{t \geq 0}$ is positive for $F_+$, then the only difficulty in the application of Theorem \ref{theorem:main_theorem} for the coupled semigroup is the verification of Condition (b). This is the result of the following theorem.

\begin{theorem}\label{theorem:comparison}
	Fix $i \in \{1,2\}$. Consider two SCLE positive semigroups $\{S(t)\}_{t \geq 0}$ and $\{S^{(i)}(t)\}_{t \geq 0}$ on the space $(E,\tau_E) = (C_b(\cX),\beta)$ with generators $A$ and $A^{(i)}$. Let Assumption \ref{assumption:main_assumption}  be satisfied for $A$. In addition to $A,  A^{(i)}$ and $B$, suppose there are operators $C \subseteq F \times F$, $C^{(i)} \subseteq F \times E$ that are positive and continuous.  
	
	Let $\cD \subseteq \cD(B \Phi) \cap \cD(C \Phi) \cap \cD(\Phi A)$ and  $\cD^* \subseteq E$  a linear space $\cD  \subseteq  \cD^* \subseteq  \cD(A^{(i)}) \cap \cD(\Phi)$. Suppose that for all $f \in \cD$ we have 
	\begin{enumerate}
	\item[(a)]
	\begin{equation} \label{eqn:theorem_comparison_assumptions}
	\Phi A f = (B+C) \Phi f, \quad \, (A^{(i)} - A)f = (-1)^i C^{(i)} \Phi f.
	\end{equation}
	\end{enumerate}
	Suppose finally that
	\begin{enumerate}
	\item[(b)]
	 $(\cD \times \cD^*) \cap \cH_{+,0}^{(i)}$ is strictly dense in $\cH_{+,0}^{(i)}$,
	 \item[(c)]  $R(\lambda,\cA)\, \left( \cD\times \cD^*\right) \subseteq \cD\times \cD^*$. 
	\end{enumerate}
	\smallskip
	
	Then, for all $t \geq 0$ and $(f,h) \in \cH_+^{(i)}$, we have $(S(t)f,S^{(i)}(t)h) \in \cH_+^{(i)}$. In particular, if $f \in \cF_+$ and $t \geq 0$, then $(-1)^{(i)} S(t) f \leq (-1)^{(i)} S^{(i)}(t)f$. 
\end{theorem}

\begin{remark}
	Note that one of the conditions of Theorem \ref{theorem:comparison} is that there is a sufficiently large intersection of the domains of $\cD(A)$ and $\cD(A^{(i)})$. This implies that we need to assume the same boundary conditions for both operators. We are unsure if this condition is essential.
\end{remark}

The main application of this result is the following important corollary that establishes the preservation of stochastic order under the evolution of two Markov processes.

\begin{corollary} \label{corollary:preservation_of_order}
	Let $\{\eta^{(1)}(t)\}_{t \geq 0}$ and $\{\eta^{(2)}(t)\}_{t \geq 0}$ be Feller processes on $E$ with semigroups $\{S^{(1)}(t)\}_{t \geq 0}$ and $\{S^{(2)}(t)\}_{t \geq 0}$. Suppose we can find a third semigroup $\{S(t)\}_{t \geq 0}$ so that all conditions for Theorem \ref{theorem:comparison} are satisfied for the two couples of semigroups
	\begin{enumerate}[(a)]
		\item $S^{(1)}$ and $S$ with $i = 1$, i.e.~$(A^{(1)} - A)f = - C^{(1)} \Phi f$ for a positive continuous operator $C^{(1)} \subseteq E \times F$,
		\item $S^{(2)}$ and $S$ with $i = 2$, i.e.~$(A^{(2)} - A)f =  C^{(2)} \Phi f$ for a positive continuous operator $C^{(2)} \subseteq E \times F$.
	\end{enumerate}
	If $\eta^{(1)}(0) \preccurlyeq \eta^{(2)}(0)$, then $\eta^{(1)}(t) \preccurlyeq \eta^{(2)}(t)$ for all $t \geq 0$.
\end{corollary}

\begin{remark} \label{remark:restriction_of_Upsilon}
Note that the domain of $\Upsilon$ was taken to be $\cD(\Phi)^2$, whereas $D(\Phi) \times E$ would have sufficed. We restrict ourselves to this setting as it implies the result that we would like to obtain (i.e.~for $f \in \cF_+$, we have $(f,f) \in \cH_+^{(i)}$), while making the verification of condition (b) of Theorem \ref{theorem:comparison} via Proposition \ref{proposition:core_double_approximation_operators} below slightly easier. 

By straightforward approximation arguments the result of the theorem can be extended for an extended domain of $\Upsilon$.
\end{remark}

\subsection{Approximation operators} \label{section:approximation_operators}

For condition (b) of Theorems \ref{theorem:main_theorem}, Theorem \ref{theorem:lower_bound} and Theorem \ref{theorem:comparison}, we need to find sequences or nets of functions in the domain $\cD$ intersected with $\cF_{+,0}$ that approximate any function in $\cF_{+,0}$. 
Consider the following setting.
\begin{assumption} \label{assumption:approx_operators}
Let $(E,\tau_E)$ and $(F,\tau_F)$ be locally convex spaces, which are additionally equipped with norms $\vn{\cdot}_E, \vn{\cdot}_F$ so that $(E,\tau_E,\vn{\cdot}_E),(F,\tau_F,\vn{\cdot}_F),$ satisfy Condition C. Additionally, suppose that $F$ is equipped with a partial order $\leq$.

Let $\Phi \subseteq E \times F$ and set $\cF_{+,0} := \{f \in E\cap\cD(\Phi) \, | \, \Phi f \geq 0\}$. Finally, let $\cD$ be a subspace of $E$.
\end{assumption}

\begin{definition} \label{definition:approximation_operators}
Let Assumption \ref{assumption:approx_operators} be satisfied. A family of linear operators $T_n \colon E \rightarrow E$, $n \geq 1$ is a family of \textit{approximation operators} for $(\Phi, \cF_{+,0},\cD)$ if
\begin{enumerate}[(a)]
\item $T_n  \cF_{+,0} \subseteq \cF_{+,0}$ for all $n$,
\item $T_n  \cD(\Phi) \subseteq \cD$ for all $n$,
\item the family $\{T_n\}_{n \geq 1}$ is (strictly) strongly continuous: for all $f \in E$, we have that $T_n f \rightarrow f$ for the strict topology.
\end{enumerate}
We say that the family $\{T_n\}_{n \geq 1}$ is continuous if
\begin{enumerate}[(a)]
\item[(d)] for each $n$ the operator $T_n$ is strictly continuous.
\end{enumerate}

Finally, we say that the family $T_n$ is positive if 
\begin{enumerate}[(a)]
\item[(e)] $f \leq g$ implies that $T_n f \leq T_n g$ for all $n$.
\end{enumerate}
\end{definition}

The next two results can be used to establish condition (b) of the three Theorems \ref{theorem:main_theorem}, \ref{theorem:lower_bound} and \ref{theorem:comparison}.

\begin{proposition} \label{proposition:core_approximation_operators}
Let Assumption \ref{assumption:approx_operators} be satisfied. Let $\{T_n\}$ be approximation operators for $(\Phi,\cF_{+,0},\cD)$, then $\cD \cap \cF_{+,0}$ is strictly dense in $\cF_{+,0}$.
\end{proposition}

\begin{corollary}\label{corollary:core_approximation_operators}
Let  $\cD^* \subseteq F$ be such that $\Phi \cD \subseteq \cD^*$ and $\Upsilon$ and $\cG_{+,0}$ as in Theorem~\ref{theorem:lower_bound}.
Let  $\{T_n\}$ be approximation operators for $(\Phi,\cF_{+,0},\cD)$ and 
 $\{T'_n\}$ be approximation operators for $(\bONE,F_+,\cD^\star)$ such that for $(f,h) \in \cD(\Upsilon)$
 \begin{equation}\label{Phi_positive}
 \text{if $\Phi f \ge h \ge 0$ then $\Phi T_n f \ge T'_n h \ge 0$.}
 \end{equation}
 Then $\{\tilde T_n\}$ with $\tilde T_n = (T_n, T'_n)$ are approximation operators for $(\Upsilon,\cG_{+,0},\cD\times\cD^*)$.
\end{corollary}

\begin{proposition} \label{proposition:core_double_approximation_operators}
Let Assumption \ref{assumption:approx_operators} be satisfied. 


Let $\{T_n\}$ be a positive continuous family of approximation operators for $(\Phi,\cF_{+,0},\cD)$. Let $i \in \{1,2\}$. Consider the cone
\begin{equation*}
\cH_{+,0}^{(i)} := \left\{(f,h) \in E^2 \, \middle| \, f, h \in \cD(\Phi), \, \Phi f \ge 0, \,  (-1)^i f \leq (-1)^i h \right\}.
\end{equation*}
Then $\cD^2 \cap \cH_{+,0}^{(i)}$ is strictly dense in $\cH_{+,0}^{(i)}$.
\end{proposition}

The proofs of the results in this section can be found in Section \ref{section:proofs_main_results}.

\section{One dimensional diffusion processes with Feller boundary conditions} \label{section:one_d_diffusion}

A main new application of our results extend on Example \ref{example:one_d_convex_order}. We will consider diffusion processes on subintervals of $\bR$ that have at least one boundary point. Before studying stochastic order properties of these diffusion processes, we start with a short discussion on boundary conditions in Section \ref{section:1d_diffusion_boundary} and convenient regularity conditions in Section \ref{section:1d_diffusion_regularity}.

Afterwards, in Sections \ref{section:1d_diffusion_increasing_order}, \ref{section:1d_diffusion_convex_order} and \ref{section:1d_diffusion_increasing_convex_order_interval}, we will study the increasing, convex, and the increasing convex order, induced by the operators $f \mapsto f'$, $f \mapsto f''$ and $f \mapsto (f',f'')$. In the proofs of the main results in these sections we use several classes of approximation operators. We post-pone their analysis to Section \ref{section:1d_diffusion_approximation_operators}.

\subsection{Boundary conditions} \label{section:1d_diffusion_boundary}

For one-dimensional diffusion processes we consider operators that are appropriate subsets of 
\begin{equation} \label{eqn:1d_diffusion_operator}
\tilde{L}f(x) = \frac{1}{2} a(x) f''(x) + b(x)f'(x) - c(x) f(x),
\end{equation}
where $a,b,c$ are continuous on an open interval $I \subseteq \bR$ and $c \geq 0$. In particular, if considering intervals with a boundary, we need to specify what happens if a process is at the boundary. These boundary conditions need to be consistent with whether or not the process can enter or exit at the boundary point. These issues are best addressed in terms of the speed and scale and killing measure.

\begin{definition}
	Fix some element $z \in I$, which serves as a base point for integration but otherwise serves no significant role. First define 
	\begin{equation*}
	\Lambda(x) = \int_z^x \frac{2b(y)}{a(y)} \dd y,
	\end{equation*}
	and define the scale measure $s$, speed measure $m$ and killing measure $k$ on $I$ by
	\begin{equation*}
	s(x) = \int_z^x \e^{-\Lambda(y)} \dd y, \quad m(x) = \int_z^x \frac{2 \e^{\Lambda(y)}}{a(y)} \dd y, \quad k(x) = \int_z^x \frac{2 c(y) \e^{\Lambda(y)}}{a(y)} \dd y.
	\end{equation*}
\end{definition}

Note that $s$, $m$, $k$ defined here are functions. The associated measures are given on intervals by
\begin{equation*}
s((a,b)) = s(b) - s(a), \quad m((a,b)) = m(b) - m(a), \quad k((a,b)) = k(b) - k(a).
\end{equation*}
Note that in our setting   the measures $s,m,k$ are absolutely continuous on $I$ with respect to the Lebesgue measure  with densities
\begin{equation*}
\frac{\dd s}{\dd x}(x) = \e^{-\Lambda(x)}, \qquad \frac{\dd m}{\dd x}(x) = \frac{2 \e^{\Lambda(x)}}{a(x)}, \qquad  \frac{\dd k}{\dd x}(x) = \frac{2 c(x) \e^{\Lambda(x)}}{a(x)}.
\end{equation*} 
The behaviour of a diffusion process can now be studied in terms of $s$ and $m$. In particular, if $X(t)$ is the diffusion process we are considering, $t \mapsto s(X(t))$ is a local martingale. The density $\frac{\dd}{\dd x} m(x)$ determines how much time the process spends at a location $x$, whereas $\frac{\dd}{\dd x}k(x)$ determines the killing rate.

If $c = 0$ the diffusion is conservative on $I$ and we find that
\begin{equation*}
\tilde{L}f(x) = \frac{1}{2} a(x) f''(x) + b(x)f'(x) = \frac{\dd}{\dd m(x)} \frac{\dd}{\dd s(x)} f(x).
\end{equation*}

 In addition, $s$, $m$ and $k$ determine whether the process can reach, or leave, a boundary. To be specific, define the auxiliary functions
\begin{align*}
u(x) = \int\limits_{(z,x)} \left(m((z,y)) + k((z,y)) \right) s(\dd y), \quad v(x) = \int\limits_{(z,x)} s((z,y)) (m(\dd y) + k(\dd y)),
\end{align*}
where the convention $(x,y) = \{z \colon \min(x,y) < z < \max(x,y)\}$ allows a compact treatment of right and left endpoints.
\begin{definition} 
	Let $I$  be an open interval with boundary points $\partial I = \{l,r\}$, $l < r$ in $\bR \cup \{- \infty\} \cup \{\infty\}$. A boundary point $e \in  \{l, r\}$ is called
	\begin{enumerate}[(a)]
		\item \textit{exit} if $u(e) < \infty$,
		\item \textit{entrance} if $v(e) < \infty$,
		\item \textit{regular} or \textit{non-singular} if it is both exit and entrance,
		\item \textit{natural} if it is neither exit nor entrance.
	\end{enumerate}
\end{definition}

	The space on which we define our diffusion process depends on the boundary behaviour of the diffusion process. 
	
	\begin{assumption} \label{assumption:one_d_diffusion_boundaries}
		Consider $\tilde{L}$ as in \eqref{eqn:1d_diffusion_operator} defined on the open interval $I$ with boundary $\partial I \subseteq \bR \cup \{-\infty\} \cup \{\infty\}$. Then our space $\cX$ satisfies $I \subseteq \cX \subseteq \overline{I}$ and for any boundary point $e \in \partial I$ we have 
		\begin{enumerate}[(a)]
			\item either $e$ is regular, and then $e \in \cX$,
			\item or $e$ is natural, and then $e \notin \cX$.
		\end{enumerate}
	\end{assumption}

Note that the measures $m$ and $k$ are only specified on $I$. We extend them by allowing mass at the boundary points $\partial I$.

We recall the following partial result from \cite[II.1.7.]{BoSa2002} combining several results taken from a number of sources, using the notation 
\begin{equation*}
\frac{\dd f}{\dd s}(l +) = \lim_{x \downarrow l} \frac{\dd f}{\dd s}(x), \qquad \frac{\dd f}{\dd s}(r-) = \lim_{x \uparrow r} \frac{\dd f}{\dd s}(x).
\end{equation*}

\begin{theorem}[] \label{theorem:domain_diffusion_operator}
	1. Consider the operator $\widetilde{L}$ on an interval $I$ with extension $\cX$ satisfying Assumption \ref{assumption:one_d_diffusion_boundaries}.
	Let $L \subseteq \widetilde{L}$ be the graph of functions $(f,g) \in C_{b}(\cX) \times C_{b}(\cX)$ such that $\frac{\dd f}{\dd s}$ exists and
	\begin{equation}\label{domain_diffusion_operator}
		\int_a^b g(x) m(\dd x) = \frac{\dd f}{\dd s}(b) - \frac{\dd f}{\dd s}(a) - \int_a^b f(x) k(\dd x)
	\end{equation}
	for all $a < b$ in $I$.
	
	2. In addition suppose for the  boundaries $e \in \{l, r\}$	
	\begin{enumerate}[(a)]
		\item If $e$ is regular, then we have one of the four options  (a1), (a2) or (a3) or (a4) with $\gamma_e \in (0,\infty)$:
			\begin{flalign*}
				\text{(a1)} && g(l) m(\{l\}) & = \frac{\dd f}{\dd s}(l+) - f(l)k(\{l\}), & \text{if } m(\{l\}) < \infty, k(\{l\}) < \infty,  \\
				&& \text{ respective } \\
				&& g(r) m(\{r\}) & = - \frac{\dd f}{\dd s}(r-) - f(r)k(\{r\}), & \text{if } m(\{r\}) < \infty, k(\{r\}) < \infty,  \\
				\text{(a2)} && g(e) & = 0, & \text{if } m(\{e\}) = \infty, k(\{e\}) < \infty,  \\
				\text{(a3)} && f(e)  & = 0, & \text{if } m(\{e\}) < \infty, k(\{e\}) = \infty, \\
				\text{(a4)} && g(e) & =  - \gamma_e f(e), & \text{if } m(\{e\}) = \infty, k(\{e\}) = \infty.
			\end{flalign*}
		\item If $e$ is natural: no additional conditions.
	\end{enumerate}
	Then $L$ is the strict generator of a Feller process with the stated boundary conditions.	
	\end{theorem}
	Condition (a1) is referred to as sticky if $m(\{e\}) > 0$ and $k\{e\} = 0$, reflecting if $m(\{e\}) = 0$ and $k(\{e\}) = 0$ and elastic if $m(\{e\}) = 0$ and $k(\{e\}) > 0$. Condition (a2) is referred to as absorbing, condition (a3) is referred to as killing and condition (a4) is called a trap.

\subsection{A regular class of operators}\label{section:1d_diffusion_regularity}

We will treat one set of examples exhaustively.

\begin{condition} \label{condition:proper_diffusion_operator}
	Consider the operator
	\begin{equation*}
	\widetilde{L}f(x) = \frac{1}{2}a(x) f''(x) + b(x) f'(x) - c(x) f(x), \quad x \in \cX.
	\end{equation*}
	with $\cD(\tilde L) = \{f \in C^2_b(\cX), \tilde L f \in C_b(\cX)\}$.
	Suppose that $a,b,c$ satisfy the following three conditions:
	\begin{enumerate}[(a)]
	\item Suppose that $a,b,c$ are twice continuously differentiable.
	\item There is a constant $\overline{a} > 0$ such that for all $x \in \cX$: $0 < \overline{a}^{-1} \leq a(x) \leq \overline{a} < \infty$.
	\item There are constants $\overline{b},\overline{c}$ such that for all $x \in \cX$: $|b(x)| \leq \overline{b}(1 + |x|)$ (bounded growth), $|c(x)| \leq \overline{c}$.
	\end{enumerate}
\end{condition}

	\begin{remark}
		In Section \ref{section:1d_diffusion_boundary} above, we discussed operators where $c \geq 0$. This choice naturally corresponds to diffusion processes with killing. If $c \leq 0$ and bounded, we can naturally construct the corresponding semigroup via the one with drift $c + \vn{c}$ and then removing the $\vn{c}$ part via perturbation argument or a Feynman-Kac construction. See for example Theorem III.1.3 in \cite{EN00} for the perturbation argument.
	\end{remark}

\begin{lemma} \label{lemma:boundary_behaviour_for_standard_ab}
	Let $\widetilde{L}$ satisfy Condition \ref{condition:proper_diffusion_operator} with $c \ge 0$. Then 
	\begin{enumerate}[(a)]
		\item the function $\frac{\dd s}{\dd x} (x)= \e^{-\Lambda(x)}$ is bounded and bounded away from $0$ near finite boundaries;
		\item any finite boundary is a regular boundary;
		\item any infinite boundary is natural (and one can set $\cX \subseteq \bR$).
	\end{enumerate}
\end{lemma}

\begin{proof}
	Let $e$ be a finite boundary. Then on a neighbourhood of $e$ the functions $b,c, a, a^{-1}$ are bounded. It follows that the functions $\Lambda, \frac{\dd s}{\dd x},(\frac{\dd s}{\dd x})^{-1}\frac{\dd m}{\dd x},\frac{\dd k}{\dd x}$ are bounded as well. The integrals $u,v$ remain bounded on finite intervals too. This establishes the first two claims.
	
	\smallskip

	Let $e$ be an infinite boundary. Assume without loss of generality that $e = r$ and $z > 0$. 
	Using the upper bound on $b$ and the lower bound on $a$, we find some constant $M > 0$ such that $
	- |\Lambda(w) - \Lambda(y)| \geq - M |y^2 - w^2|.$
	
	As $c \geq 0$, we have $k \geq 0$. Thus, using the upper bound on $a$, we find for $x/2 > z$,
	\begin{align*}
	u(x) 
		& \ge  2 \overline{a}^{-1}   \int_z^x  \int_z^y e^{-M|y^2 - w^2|} \dd w\, \dd y \\
		& \geq 2 \overline{a}^{-1}   \int_{x/2}^x  \int_{z}^{x/2} e^{M(w^2 - y^2)} \dd w \dd y \\
		& \geq 2 \overline{a}^{-1}   \int_{x/2}^x e^{M w^2} \dd w \int_{z}^{x/2}  e^{-M y^2}  \dd y,
	\end{align*}
	which diverges to $+\infty$ for $x \to \infty$.
	In the same fashion,
		\begin{align*}
		v(x) & \geq 2\overline{a}^{-1} \int_z^x \int_z^y \e^{-\Lambda(w) + \Lambda(y)} \dd w \, \dd y\\
		& \ge  2 \overline{a}^{-1}   \int_z^x  \int_z^y e^{-M|y^2 - w^2|} \dd w\, \dd y,
		\end{align*}
	at which point the same argument applies, together showing that $r$ is a natural boundary.
\end{proof}

Clearly, if $f'(e) = 0$ (respective $f(e) = 0$) for $f \in \cD(L)$ for a finite boundary $e$ and $L \subseteq \widetilde{L}$ for an operator $\widetilde{L}$ that satisfies Condition \ref{condition:proper_diffusion_operator}, then a diffusion generated by $L$ has a reflecting boundary (respective killing boundary) at $e$ irrespective of the (finite) value of $b(e)$. For other boundary types, the values of $c(e)$ and $b(e)$ influence which boundary conditions correspond to certain constrains on $f(e), f'(e), f''(e)$. The next lemma sheds some light on the situation when there is no killing at the boundary. We argue for an interval $(-\infty,r]$, $r < \infty$, but other cases can be treated similarly.

	\begin{lemma} \label{lemma:allowed_boundary_conditions}
		Let $\cX = (-\infty, r] \subseteq \bR$. Suppose that $\widetilde{L}$ with $c \ge 0$, $c(r) = 0$  satisfies Condition \ref{condition:proper_diffusion_operator}. The operator $L \subseteq \widetilde{L}$ with 
			\[\cD(L) =  \{ f \in \cD(\widetilde L)\colon  \text{\small $f$ fulfils right boundary condition (a1) or (a2) with $k(\{r\}) = 0$} \}
			\]
			is the strict generator of a diffusion with boundary condition
			\begin{equation*}
			\cD(L)  = \left\{f \in \cD(\widetilde L) \, \middle| \, \gamma_r f'(r) =  f''(r) \right\}, 
			\end{equation*}
			where $\gamma_r \in \left[-\infty, - \tfrac{2b(r)}{a(r)}\right]$. Here $- \infty f'(r) = f''(r)$ should be interpreted as $f'(r) = 0$. 			
			We have that the diffusion has 
			\begin{itemize}
				\item reflecting boundary if $\gamma_r = - \infty$,
				\item sticky boundary if $\gamma_r \in \left(-\infty, - \tfrac{2b(r)}{a(r)}\right)$,
				\item absorbing boundary if $\gamma_r = - \tfrac{2b(r)}{a(r)}$.
			\end{itemize}
		Analogous statements hold for left boundaries with the condition $\gamma_r \in \left[-\infty, - \tfrac{2b(r)}{a(r)}\right]$ replaced by $\gamma_l \in \left[ - \tfrac{2b(l)}{a(l)},\infty\right]$.
	\end{lemma}

\begin{remark} \label{remark:sign_of_drift_for_second_order}
	In the context of Lemma \ref{lemma:allowed_boundary_conditions}, note that the right-side boundary condition
	\begin{equation*}
	\cD(L)  = \left\{f \in \cD(\widetilde L) \, \middle| \,   f''(r) = 0 \right\}, 
	\end{equation*}
	is possible only if $b(r) \leq 0$, with analogous statement for a left-side boundary if $b(l) \geq 0$.
\end{remark}

\begin{proof}[Proof of Lemma \ref{lemma:allowed_boundary_conditions}]
	The reflecting case is immediate, with $\gamma_r = -\infty$, so we can assume $m(\{r\}) > 0$.
	
	 Pick $f \in \cD(L)$. Without loss of generality, we may assume $\Lambda(r) = 0$ and write the boundary condition as 
	\begin{equation}\label{eqn:allowed_boundary}
	\left(\frac12 a(r)f''(r) + b(r)f'(r)\right)  m(\{r\})  = -f'(r) .
	\end{equation}
	Reordering yields
	\begin{equation*}
	f''(r) = - \frac{\tfrac{1}{m(\{r\})} + b(r)}{\tfrac{1}{2} a(r)} f'(r),
	\end{equation*}
	which identifies the constant $\gamma_r \in \left.\left(-\infty, - \tfrac{2b(r)}{a(r)}\right.\right]$. The absorbing case $m(\{r\}) = \infty$ corresponds to $\gamma_r = - \tfrac{2b(r)}{a(r)}$.
\end{proof}

\begin{lemma} \label{lemma:gaining_regularity}
	Consider an operator $L \subseteq \widetilde{L}$, where 
	\begin{equation*}
	\widetilde{L}f(x) = \frac{1}{2}a(x) f''(x) + b(x) f'(x) - c(x) f(x)
	\end{equation*} 
	and where $\tilde{L}$ satisfies Condition \ref{condition:proper_diffusion_operator} and where $L$ satisfies an appropriate set of boundary conditions as in Theorem \ref{theorem:domain_diffusion_operator}.
	
	Let $a,b, c \in C^z(\cX)$ for some $z \geq 2$. Let $n \in \bN$ and $h \in C^n(\cX)$. Let $f$ be a solution to $f - \lambda L f = h$. Then $f \in C^{(n \wedge z)+2}$. 
\end{lemma}
\begin{proof}
	We start by establishing an auxiliary result: if $g \in C^n(\cX)$ and $\frac{\dd}{\dd m} \frac{\dd}{\dd s} f = g$, then $f \in C^{(n \wedge z)+2}$.
	 
	\smallskip

	As $\frac{\dd}{\dd x} = \frac{\dd s}{\dd x} \frac{\dd}{\dd s}$ and $\frac{\dd}{\dd x} = \frac{\dd m}{\dd x} \frac{\dd}{\dd m}$  continuous differentiability of some function $f$ in $x$ follows from  continuous differentiability in $s$ or $m$ (note that $\frac{\dd s}{\dd x}, \frac{\dd m}{\dd x}$ are bounded and continuous by Lemma \ref{lemma:boundary_behaviour_for_standard_ab}.) 
	For $f \in \cD(L)$, we know that $\frac{\dd f}{\dd s}$ exists and, as $k = 0$, is given as the $m$ integral over $g$. We conclude $\frac{\dd f}{\dd s}$ is continuous. It follows that
	\begin{equation*}
	\frac{\dd f}{\dd x} = \frac{\dd s}{\dd x} \frac{\dd f}{\dd s}.
	\end{equation*}
	By the fundamental theorem of calculus we have $\frac{\dd}{\dd m} \frac{\dd}{\dd s} f = g$. In particular, we find by the argument above that $\frac{\dd f}{\dd s}$ is continuously differentiable in $x$. By definition $\frac{\dd s}{\dd x}$ is also continuously differentiable in $x$. Thus by the product rule
	\begin{equation*}
	\frac{\dd^2 f}{\dd x^2} = \frac{\dd^2s}{\dd x^2} \cdot\frac{\dd f}{\dd s}  + \frac{\dd s}{\dd x} \cdot g.
	\end{equation*}
	In addition, note that $\frac{\dd}{\dd x} \frac{\dd f}{\dd s} = \frac{\dd m}{\dd x} g$. As $\frac{\dd s}{\dd x}$ is $z+1$ times continuously differentiable, we thus see that the right hand side is $n \wedge z$ times continuously differentiable. In other words, $f$ is $(n \wedge z) + 2$ times continuously differentiable. 
	
	\smallskip
	
	Thus, we have established our auxiliary result: if $g \in C^n(\cX)$ and $\frac{\dd}{\dd m} \frac{\dd}{\dd s} f = g$, then $f \in C^{(n \wedge z)+2}$. Next, let $h \in C^n(\cX)$ and let $f$ be a solution to $f - \lambda L f = h$, or equivalently $\frac{\dd}{\dd m} \frac{\dd}{\dd s} f = \lambda^{-1}(f-h) + c f$. Now the result follows by induction from the first part of our proof.

\end{proof}

\subsection{The increasing order} \label{section:1d_diffusion_increasing_order}

We consider appropriate subsets of the operator in \eqref{eqn:1d_diffusion_operator}, the latter now denoted by $\tilde{A}$, and study the propagation of the increasing order which is generated by $\Phi \subseteq C_b(\cX) \times C_b(\cX)$ defined by $\Phi f = f'$.

First, note that for $f \in C^3(\cX)$, we have
\begin{equation} \label{eqn:Phi_increasing_on_1d_diffusion}
\Phi \widetilde{A} f(x) = \frac{1}{2} a(x) f^{(3)} (x) +\left(\frac{1}{2} a'(x) + b(x)\right) f''(x) + b'(x) f'(x),
\end{equation}
suggesting to use a subset $B$ of 
\begin{equation} \label{eqn:B_increasing_on_1d_diffusion}
\tilde{B} f(x) = \frac{1}{2} a(x) f''(x) + \left(\frac{1}{2} a'(x) + b(x)\right) f'(x) +  b'(x) f(x)
\end{equation}
The following result is immediate.

\begin{lemma}\label{lemma:tildeb_proper}
	Suppose that $\widetilde{A}$ satisfies Condition \ref{condition:proper_diffusion_operator} with $a \in C^3(\cX)$ with $a'$ of bounded growth and $b' \in C^2_b(\cX)$. Then $\widetilde{B}$ satisfies Condition \ref{condition:proper_diffusion_operator}.
\end{lemma}

In the setting of $\widetilde{A}$ satisfying Condition \ref{condition:proper_diffusion_operator}, we give conditions for the preservation of $\cF_+$ under the semigroup corresponding to $A \subseteq \widetilde{A}$ satisfying boundary conditions as in Theorem \ref{theorem:domain_diffusion_operator}. 

\begin{theorem} \label{theorem:1d_diffusion_increasing_functions}
	Let $\widetilde{A}$ satisfy Condition \ref{condition:proper_diffusion_operator} and $a \in C^3(\cX)$ with derivative of bounded growth and $b' \in C^2_b(\cX)$. Let $A \subseteq \widetilde{A}, B \subseteq \widetilde{B}$ be two operators satisfying for each $e \in \{l,r\} \cap \bR$ boundary conditions as in Theorem \ref{theorem:domain_diffusion_operator}  such that with the same conventions for $\gamma_e$ as in Lemma~\ref{lemma:allowed_boundary_conditions}
	\begin{flalign*}
	\text{(A)} && \cD(A) & = \left\{f \in \cD(\tilde A) \, \middle| \, \gamma_e f'(e) = f''(e), e \in \{l,r\} \cap\bR\right\}, & \\
	\text{(B)} && \cD(B) & = \left\{f \in \cD(\tilde B) \, \middle| \, \gamma_e f(e) = f'(e), e \in \{l, r\} \cap\bR \right\}. &
	\end{flalign*}

	Let $\{S(t)\}_{t \geq 0}$ be the semigroup generated by $A$ and let $\{T(t)\}_{t \geq 0}$ be the semigroup generated by $B$.
	
	Then the semigroup $\{S(t)\}_{t \geq 0}$ generated by $A$ maps increasing functions to increasing functions. In addition,  $(S(t) f) ' \geq T(t)(f')$ for $f \in C^1_b(\cX)$.
\end{theorem}

\begin{remark}
	Note that the range of the constants $\gamma_l$ and $\gamma_r$ for (A) is restricted, see Lemma \ref{lemma:allowed_boundary_conditions}. Similar restrictions hold for the results that follow in the sections on one-dimensional diffusion processes below.
	Especially, a condition with $\gamma_e = 0$ is possible only under conditions on the sign of the drift at the boundary, see Remark \ref{remark:sign_of_drift_for_second_order}. 
\end{remark}

\begin{proof}[Proof of Theorem \ref{theorem:1d_diffusion_increasing_functions}]
	We first proof the theorem in the case where $\gamma_e$ are finite, in which case $F=C_b(\cX)$.	
	We verify the conditions for Theorem \ref{theorem:main_theorem}. By the computations above $\Phi (\cD(A) \cap C^3_b(\cX)) \subseteq \cD(B)$ and $\Phi A = (B + C) \Phi$.

	We therefore choose $\cD = \cD(A) \cap C^3_b(\cX)$.
	We are left to verify (b) and (c) of Theorem \ref{theorem:main_theorem}. 
	(b) is verified using the approximation operators $\{T'_n\}$ in Proposition \ref{proposition:approximation_operators_1d_diffusion}. 
	For (c) note that it suffices to show that $R(\lambda,A) \cD(A) \subseteq \cD(A)$ and $R(\lambda,A) C^3(\cX) \subseteq C^3(\cX)$. 
	As $A R(\lambda,A) = \lambda^{-1}(R(\lambda,A) - \bONE)$, we find that $R(\lambda,A) C(\cX) \subseteq \cD(A)$. 
	By Lemma \ref{lemma:gaining_regularity} we have $R(\lambda,A) C^3(\cX) \subseteq C^3(\cX)$.
	
	\smallskip
	
	We proceed with the verification of the conditions for Theorem \ref{theorem:lower_bound}.
	Corresponding with the choice of $\cD$, we work with $\cD^* = \cD(B) \cap C^2_b(\cX)$ so that $\Phi \cD \subseteq \cD^*$. 
	Condition (c) follows as above using Lemma \ref{lemma:gaining_regularity} which can be applied by Lemma \ref{lemma:boundary_behaviour_for_standard_ab}. 
	The last property that needs to be checked is (b). It, however, follows from Corollary \ref{corollary:core_approximation_operators}  using the set of approximation operators $\{(\hat T'_n, \hat T''_n)\}$  defined in Proposition \ref{proposition:refined_approximation_operators_1d_diffusion}. 
	
	We can thus conclude by Theorem \ref{theorem:lower_bound} that if $f \in \cF_{+,0}$ and $S(t) f \in \cF_{+,0}$, then $\Phi S(t) f \geq T(t) \Phi f$.
	
	If one of $\gamma_e$ is infinite, then the space $F$ has to be chosen accordingly and the proof has to be adapted. For example if
	if all finite boundaries have $\gamma_e = \inf$, then $F = \{f\in C_b(\cX)\colon f(e) = 0 \text{ for } e \in \partial \cX \cap \cX\}$. In the second part of the argument we 	
	use the set of approximation operators $\{(T'_n, T''_n)\}$  defined in Proposition \ref{proposition:refined_approximation_operators_1d_diffusion} instead. Note that $\{T''_n\}$ is strongly continuous on $F$ in this case. The case of mixed boundary conditions can be treated similarly.

\end{proof}

Our next result considers that of preservation of order. For this we will use Theorem \ref{theorem:comparison} and Corollary \ref{corollary:preservation_of_order}. For operators $A^{(1)},A^{(2)}$ we find an operator $A$ and positive continuous operators $C^{(1)},C^{(2)}$ such that
\begin{align*}
(A^{i} - A) h = (-1)^i C^{(i)} \Phi h.
\end{align*}
We will write $\mu \lesssim_\st \nu$ if $\mu$ is smaller than $\nu$ with respect to the integral order generated by all increasing functions. 

\begin{theorem}  \label{theorem:1d_diffusion_comparison_increasing}
	Let $\tilde A^{(1)}$ and $\tilde A^{(2)}$ be operators of the form
	\begin{align*}
	\tilde A^{(1)} f(x) = \frac{1}{2} a(x) f''(x) + b^{(1)}(x) f'(x), \quad \tilde A^{(2)} f(x) = \frac{1}{2} a(x) f''(x) + b^{(2)}(x) f'(x),
	\end{align*} satisfying Condition \ref{condition:proper_diffusion_operator} with $a \in C^3(\cX)$ with derivative of bounded growth and $b^{(1)\prime} \in C^2_b(\cX)$,
	where 
	\[
	  b^{(1)}(x) \le b^{(2)}(x), \qquad x \in \cX
	\]
	with difference $b^{(2)}-b^{(1)}$ bounded.
	
	If $\eta^{(1)},\eta^{(2)}$ are diffusion processes generated by 
	operators  $A^{(i)}$ being subsets of $\tilde A^{(i)}$, $i = 1, 2$ with
	\begin{flalign*}
	\text{(A)} && \cD(A^{(i)}) & = \left\{f \in \cD(\tilde A^{(i)}) \, \middle| \,  \gamma_e f'(e) = f''(e), e \in \{l,r\} \cap \bR\right\}, & 
	\end{flalign*}
	 with  choices for $\gamma_e$ that hold for both $A^{(1)}$ and $A^{(2)}$,
	 	
	then $\eta^{(1)}(0) \lesssim_\st \eta^{(2)}(0)$ implies $\eta^{(1)}(t) \lesssim_\st \eta^{(2)}(t)$ for all $t \geq 0$.
\end{theorem}

\begin{remark}
	
	Note that $A^{(1)}$ and $A^{(2)}$ having the same domain does not imply $b^{(1)}(e) = b^{(2)}(e)$, as  $m^{(1)}(\{e\})$ and $m^{(2)}(\{e\})$, the stickiness of the boundary $e$ for $A^{(1)}$ and $A^{(2)}$,  may differ. See Lemma~\ref{lemma:allowed_boundary_conditions}.
\end{remark}

\begin{proof}[Proof of Theorem \ref{theorem:1d_diffusion_comparison_increasing}]
	As $A^{(1)}$ is stochastically monotone by Theorem \ref{theorem:1d_diffusion_increasing_functions} we can take $A = A^{(1)}$  and use Theorem \ref{theorem:comparison} with  $i=2$. The conditions for $A$ of that theorem are satisfied with $\cD$ as in Theorem \ref{theorem:1d_diffusion_increasing_functions} above.
	
	It follows for $f \in C^2_b(\cX)$ that $(A^{(2)} - A^{(1)})f(x) = \left(b^{(2)}(x) - b^{(1)}(x)\right) f'(x)$, so that we can choose $C^{(1)} f(x) = (b^{(2)}(x) - b^{(1)}(x)) f(x)$. For $\cD^*$ we take $C^2_{b}(\cX) \cap \cD(A)$. Condition (c) of Theorem \ref{theorem:comparison} follows as before, and Condition (b) follows from Proposition \ref{proposition:core_double_approximation_operators} using the approximation operators $\{T'_n\}$ introduced in Proposition \ref{proposition:approximation_operators_1d_diffusion} below.
\end{proof}

\subsection{The convex order} \label{section:1d_diffusion_convex_order}

We proceed with the study of the convex order, generated by $\Phi \subseteq C_b(\cX) \times C_b(\cX)$ defined by $\Phi f = f''$. A first remark is that $C_b(\bR)$ contains only a trivial set of convex functions. For a half-line, the situation is slightly more difficult, but we obtain that the only convex functions in $C_b[0, \infty)$ are decreasing, which puts us in a setting similar to that of next section. We will therefore restrict our analysis to bounded intervals $\cX = [l, r]$.

\smallskip

Considering
\begin{equation*}
\tilde{A} f(x) = \frac{1}{2}a(x)f''(x) + b(x) f'(x),
\end{equation*}
we find for a four times continuously differentiable function that 
\begin{multline} \label{eqn:Phi_convex_on_1d_diffusion}
\Phi \tilde{A} f(x) = \frac{1}{2}a(x)f^{(4)}(x) + \left(a'(x) + b(x)\right) f^{(3)}(x) \\
+ \left(\frac{1}{2}a''(x) + 2 b'(x)\right) f''(x)+ b''(x) f'(x).
\end{multline}
This suggest the restriction to affine drift ($b'' \equiv 0$) and taking an operator $B \subseteq \widetilde{B}$, where
\begin{equation}\label{eqn:B_convex_on_1d_diffusion}
\tilde{B}f(x) = \frac{1}{2}a(x)f''(x) + \left(a'(x) + b(x)\right) f'(x) + \left(\frac{1}{2}a''(x) + 2 b'(x)\right)f(x)
\end{equation}
and $C = 0$. An immediate result is that our regular class of operators is preserved.

\begin{lemma}
	Suppose that $\widetilde{A}$ satisfies Condition \ref{condition:proper_diffusion_operator} with $a,b \in C^4(\cX)$. Then $\widetilde{B}$ satisfies Condition \ref{condition:proper_diffusion_operator}.
\end{lemma}

\begin{theorem} \label{theorem:1d_diffusion_convex_functions}
Let $\cX = [l,r] \subseteq \bR$. Let $\widetilde{A}$ satisfy Condition \ref{condition:proper_diffusion_operator}. In addition, $b(l) \geq 0  \geq b(r)$ and $b'' = 0$ and that $a$ is $C^4$. Let $A \subseteq \widetilde{A}$, $B \subseteq \widetilde{B}$ have domains
\begin{align*}
\cD(A) & = \left\{f \in \cD(\tilde A) \, \middle| \, f''(l) = 0 = f''(r) \right\},\\
\cD(B) & = \left\{f \in \cD(\tilde B) \, \middle| \, f(l) = 0 = f(r) \right\}.
\end{align*}
Let $\{S(t)\}_{t \geq 0}$ be the semigroup generated by $A$ and let  $\{T(t)\}_{t \geq 0}$ be the semigroup generated by $B$.

Then the semigroup $\{S(t)\}_{t \geq 0}$ generated by $A$ maps convex functions to convex functions. In addition, if $f \in C^2(\cX)$, then $(S(t) f)'' \geq T(t)(f'')$.
\end{theorem}

\begin{proof}
	 Taking into account the boundary conditions of $B$, we work with $F = \{f\in C_b(\cX)\colon f(l) = f(r) = 0\}$.
	 First note that  $\{S(t)\}_{t \geq 0}$ is well defined by Lemma~\ref{lemma:allowed_boundary_conditions} as the semigroup corresponding to a diffusion process in Theorem \ref{theorem:domain_diffusion_operator}.
	The proof is similar to that of Theorem \ref{theorem:1d_diffusion_increasing_functions}, using $C = 0$, $\cD = \cD(A) \cap C^4(\cX)$,  $\cD^* = \cD(B) \cap C^2(\cX)$ 
	using Theorem~\ref{theorem:main_theorem} with approximation operators $T_n$ respective 	
		Theorem~\ref{theorem:lower_bound} where Condition (c) follows from Corollary \ref{corollary:core_approximation_operators} by using the set of approximation operators 
	$\{ (T_n,  {T}''_n)\}$ using the approximation operators introduced in Propositions \ref{proposition:approximation_operators_1d_diffusion} and \ref{proposition:approximation_operators_1d_diffusion_comparison} below.
	Again $\{T''_n\}$ is strongly continuous on $F$ in this case.
\end{proof}

We will write $\mu \lesssim_{cx} \nu$ if $\mu$ is smaller than $\nu$ with respect to the integral order generated by all convex functions.

\begin{theorem} \label{theorem:1d_diffusion_comparison_convex}
	Let $\tilde A^{(1)}$ and $\tilde A^{(2)}$ be operators of the form
	\begin{align*}
	\tilde A^{(1)} f(x) = \frac{1}{2} a^{(1)}(x) f''(x) + b(x) f'(x), \qquad \tilde A^{(2)} f(x) = \frac{1}{2} a^{(2)}(x) f''(x) + b(x)f'(x)
	\end{align*}
	with affine $b$, satisfying Condition \ref{condition:proper_diffusion_operator} and  $a \in C^4(\cX)$,
	with \[
	a^{(1)} \leq a \leq a^{(2)}
	\] 
	
	If $\eta^{(1)},\eta^{(2)}$ are diffusion processes generated by 
	operators  $A^{(i)}$ being subsets of $\tilde A^{(i)}$, $i = 1, 2$ with
	\begin{equation*}
	\cD(A^{(i)}) := \left\{f \in \cD(\tilde A^{(i)}) \, \middle| \, f''(0) = f''(1) = 0 \right\},
	\end{equation*}

	then $\eta^{(1)}(0) \lesssim_{cx} \eta^{(2)}(0)$ implies $\eta^{(1)}(t) \lesssim_{cx} \eta^{(2)}(t)$ for all $t \geq 0$
\end{theorem}

\begin{proof}
	The result follows by an application of Corollary \ref{corollary:preservation_of_order}. The conditions for the operator $A$ with $A f = \frac12 a f''  + b f' $  of Theorem \ref{theorem:comparison} are satisfied with $\cD$ as in Theorem \ref{theorem:1d_diffusion_convex_functions} above.
	
	By definition, it follows for $f \in C^2[l,r]$ that $(A^{(i)} - A)f(x) = \frac12 \left(a^{(i)}(x) - a(x)\right) f''(x)$, so that we can choose $C^{(1)} f(x) = \frac12(a(x) - a^{(1)}(x)) f(x)$ and $C^{(2)} f(x) = \frac12(a^{(2)}(x) -a(x)) f(x)$.  Condition (c) of Theorem \ref{theorem:comparison} follows as before, and Condition (b) follows from Proposition \ref{proposition:core_double_approximation_operators} using the approximation operators $\{ T_n\}$ introduced in Proposition \ref{proposition:approximation_operators_1d_diffusion} below.
\end{proof}

\subsection{The increasing convex order on an interval} \label{section:1d_diffusion_increasing_convex_order_interval}

Finally, we consider the increasing convex order, generated by $\Phi \subseteq C_b(\cX) \times (C_b(\cX) \times C_b(\cX))$ defined by $\Phi f = (f',f'')$ on an interval $\cX = [l,r] \subseteq \bR$. If as usual
\begin{equation*}
\tilde{A} f(x) = \frac{1}{2}a(x)f''(x) + b(x) f'(x),
\end{equation*}
we find for a four times continuously differentiable function, that $\Phi \widetilde{A}$ can be expressed as a vector composed of the expressions in \eqref{eqn:Phi_increasing_on_1d_diffusion} and \eqref{eqn:Phi_convex_on_1d_diffusion}.

We will then an appropriate subset $B \subseteq \widetilde{B}$, where
\begin{equation}\label{eqn:increasing_convex_order_B12}
\widetilde{B}\begin{pmatrix}
f_1 \\ f_2
\end{pmatrix} = \begin{pmatrix}
\widetilde{B}_1 f_1 \\ \widetilde{B}_2 f_2
\end{pmatrix},
\end{equation}
with
\begin{align*}
\tilde B_1 f(x) & = \frac{1}{2} a(x) f'' (x) +\left(\frac{1}{2} a'(x) + b(x)\right) f'(x) + b'(x) f(x) \\
\tilde B_2 f(x) & = \frac{1}{2}a(x)f''(x) + \left(a'(x) + b(x)\right) f'(x) + \left(\frac{1}{2}a''(x) + 2 b'(x)\right)f(x)
\end{align*}
and a positive bounded operator
\begin{equation} \label{eqn:increasing_convex_order_C}
C\begin{pmatrix}
f_1 \\ f_2
\end{pmatrix}  = \begin{pmatrix}
0 & 0 \\ b'' & 0
\end{pmatrix}\begin{pmatrix}
f_1 \\ f_2
\end{pmatrix}.
\end{equation}

\begin{theorem} \label{theorem:1d_diffusion_convex_increasing_functions}
	Let $\widetilde{A}$ satisfy Condition \ref{condition:proper_diffusion_operator} with $a, b \in C^4(\cX)$. In addition, suppose that $b(l) \geq 0$, $b(r) \leq 0$ and $b'' \geq 0$. Let $A \subseteq \widetilde{A}$ and $B_1 \subseteq \widetilde{B}_1$, $B_2 \subseteq \widetilde{B}_2$  have domains
	\begin{align*}
	\cD(A) & = \left\{f \in \cD(\tilde A) \, \middle| \, f''(l) = 0 = f''(r) \right\}, \\
	\cD(B_1) & = \left\{f \in \cD(\tilde B_1) \middle| f'(l) = 0 = f'(r) \right\}, \\
	\cD(B_2) & = \left\{f \in \cD(\tilde B_2) \middle| f(l) = 0 = f(r) \right\}.
	\end{align*}
	Let $\{S(t)\}_{t \geq 0}$ be the semigroup generated by $A$ and let $\{T_1(t)\}_{t \geq 0}$, $\{T_2(t)\}_{t \geq 0}$ be the semigroups generated by $B_1,B_2$.
	
	Then the semigroup $\{S(t)\}_{t \geq 0}$ generated by $A$ maps increasing convex functions to increasing convex functions. In addition, if $f \in C^2(\cX)$, then $(S(t) f)' \geq T_1(t)(f')$ and $(S(t)f)'' \geq T_2(t) (f'')$. 
\end{theorem}

\begin{proof}
	A proof of this result can be carried out following that of Theorem \ref{theorem:1d_diffusion_increasing_functions} with $C$ as above.
		 Taking into account the boundary conditions of $B = B_1\otimes B_2$, we work with $F = C_b(\cX)\times\{f\in C_b(\cX)\colon f(l) = f(r) = 0\}$.
	We set
	 $\cD = \cD(A) \cap C^4(\cX)$,  $\cD^* = (\cD(B_1) \cap C^3(\cX) ) \times (\cD(B_2) \cap C^2(\cX))$ using Theorem~\ref{theorem:main_theorem} with approximation operators $\{T_n\}$ respective 
	Theorem~\ref{theorem:lower_bound} where Condition (c) follows from Corollary \ref{corollary:core_approximation_operators} by using the set of approximation operators 
	$\{ (T_n, {T}_n', {T}''_n)\}$ using the approximation operators introduced in Propositions \ref{proposition:approximation_operators_1d_diffusion} and \ref{proposition:approximation_operators_1d_diffusion_comparison} below.
\end{proof}

We will write $\mu \lesssim_{icx} \nu$ if $\mu$ is smaller than $\nu$ with respect to the integral order generated by all increasing convex functions.

\begin{theorem}\label{theorem:1d_diffusion_comparison_increasingconvex}
	Let $\tilde A^{(1)}$, $\tilde A^{(2)}$ be operators of the form
	\begin{align*}
	\tilde A^{(i)} f(x) = \frac{1}{2} a^{(i)}(x) f''(x) + b^{(i)}(x) f'(x)
	\end{align*}
	satisfying Condition \ref{condition:proper_diffusion_operator}.
	
		Assume there are four times continuously differentiable functions $a,b$ such that $b(l) \geq 0$, $b(r) \leq 0$ and $b'' \geq 0$ and
	\begin{equation*}
	(-1)^i a^{(i)} \geq (-1)^i a, \qquad  (-1)^i b^{(i)} \geq (-1)^i b.
	\end{equation*}	
	
	If $\eta^{(1)},\eta^{(2)}$ are diffusion processes generated by 
	operators  $A^{(i)}$ being subsets of $\tilde A^{(i)}$, $i = 1, 2$ with
	\begin{equation*}
	\cD(A^{(i)}) := \left\{f \in \cD(\tilde A^{(i)}) \, \middle| \, f''(0) = f''(1) = 0 \right\},
	\end{equation*}

	then $  \eta^{(1)}(0) \lesssim_{icx}  \eta^{(2)}(0)$ implies $\eta^{(1)}(t) \lesssim_{icx} \eta^{(2)}(t)$ for all $t \geq 0$.
\end{theorem}

\begin{proof} 
	The proof is similar to that of Theorem \ref{theorem:1d_diffusion_comparison_convex}. The main observation is that
	\begin{equation*}
	A^{(i)} f - A f = \underbrace{\begin{bmatrix}b^{(i)}-b & 0 \\ 0 & a^{(i)}-a \end{bmatrix}}_{\displaystyle =: (-1)^i \cC^{(i)}} \Phi f
	\end{equation*}
	where $\cC^{(i)}$ are positive continuous operators. 
\end{proof}

\subsection{Approximation operators} \label{section:1d_diffusion_approximation_operators}

In the three sections above, we proved our results up to the construction of approximation operators. In this section we will construct these approximation operators. Their  definitions depend on the boundary behaviour that they need to achieve.

The first ingredient for the construction of the approximation operators is a collection of mollifiers -- functions $k_\delta \ge 0$ which are smooth and symmetric around $0$ with
$\int k_\delta(x) \, \dd x = 1$ and $\text{support } k_\delta \subseteq B_\delta(0)$. 
			
Denote by $\cM_n$ the operator $\cM_n f := k_{1/n} \ast f$, 
where ${k_\delta \ast f(x) = \int  f(\xi) k_{\delta}(x- \xi) \dd \xi}$. 
The collection of operators $\{\cM_n\}$ inherits the following properties from the mollifiers, 
cf.~Theorem C7 (iii) in \cite{Ev10}:
\begin{enumerate}[(a)]
	\item $\cM_n$ is positive as in Definition~\ref{definition:approximation_operators} and preserves the sets of non-negative, increasing, convex and increasing convex functions,
	\item $\cM_n M_b(\bR) \subseteq C^\infty_b(\bR)$,
	\item for $f \in \cM_b(\bR)$, we have $\sup_n \vn{\cM_n f} \leq \vn{f}$ and $\cM_n f \rightarrow f $ uniformly on compact sets, implying that we also have strict convergence by Theorem \ref{theorem:propertiesCbstrict} (d). \label{enum_item:strong_continuity_Mn} 
\end{enumerate}
			
Mollification alone, however, does not achieve the correct behaviour at the boundary. 
The remedy is to interpolate the function at the boundary.
The interpolation is carefully chosen to preserve positive, increasing or convex functions and related order properties, as detailed in the following three propositions.

\begin{proposition}\label{proposition:approximation_operators_1d_diffusion}
Let $\cX = [l,r] \cap \bR$, $l, r \in \bR \cup \{-\infty\} \cup \{\infty\}$. There are sets of positive approximation operators $\{T_n\}, \{T'_n\}$ and $\{T''_n\}$, with graph in $C_b(\cX) \times C_b(\cX)$ such that 
	\begin{enumerate}[(a)]
		\item $T_n$ preserves the set of increasing functions and the set of convex functions,
		\item $T'_n$  preserves the sets of increasing and decreasing functions, \label{increasing-convex-approximation}

		\item With $\tilde T_n$ denoting $T_n$, $T'_n$ or $T''_n$
		\[\tilde T_n C_b(\cX) \subseteq C^\infty_b(\cX),\]
		\item For each $n$, $T_n$, $T'_n$ and $T''_n$ are strictly continuous, and the sequences $\{T_n\}$ and $\{T'_n\}$ are strongly continuous on $C_b(\cX)$, whereas the sequence $\{T''_n\}$ is strongly continuous on the space $\{f \in C_b(\cX) \, | \, \forall e \in \partial X \cap X\colon \, f(e) = 0 \}$.
		\item If $e \in \partial \cX \cap \cX$, we have for each $n$ and $f \in C_b(\cX)$ that
			\begin{align*}
				(T_n f)''(e)  = 0,\\
				(T'_n f)''(e) = (T'_n f)'(e) = 0,\\
				(T''_n f)''(e) = (T''_n f)'(e) = (T''_n f)(e) = 0.\\
			\end{align*}
	\end{enumerate}
\end{proposition}

\begin{proposition}\label{proposition:approximation_operators_1d_diffusion_comparison}
 In the context of Proposition \ref{proposition:approximation_operators_1d_diffusion}, we have in addition that
	\begin{enumerate}[(a)]
		\item if $f \in C_b^1(\cX)$ and $f' \ge h\ge0$, then $( T_n f)'  \ge  {T}'_n h\ge0$ and  $(T'_n f)'  \ge  T''_n h\ge0$,
		\item if $f \in C_b^2(\cX)$ and $f''\ge h\ge0$, then $({T}_n f)'' \ge {T}_n'' h\ge0$.
	\end{enumerate}

\end{proposition}

\begin{proof}[Proof of both results.]
	We prove the proposition for $\cX = [0,\infty)$, the other cases being similar. We start by constructing the operators $T_n$.

	Introduce operators $\cB_n \colon C_b[0,\infty) \rightarrow C_b(-1,\infty)$ satisfying
	\[
		\cB_n f(x) = \begin{cases}
			(1-x n)f(0) + x n f(\tfrac1n) & \text{if } x \le \tfrac1n\\
			f(x) & \text{if } x \ge \frac1n
		\end{cases},
	\]
	such that $\cB_n f$ is linear on $[-\frac1n, \frac1n]$ and coincides with $f$ at $0$ and $1/n$.
	
	It is straightforward to verify that the map $ \cB_n$ is strictly continuous. It preserves the sets of increasing functions and convex functions, noting that piecewise linear interpolations of increasing functions are increasing and piecewise linear interpolations of convex functions are convex. In addition $\cB_n f \rightarrow f$ for the strict topology. 
			
	\smallskip
	
	Define  $ T_n  = \cM_n \cB_n$.
	As  we act with $\cM_n$ last, we find $ T_n C_b(\cX) \subseteq C_b^\infty(\cX)$.
	By construction, $T_n$ is strictly continuous and preserves the sets of increasing or convex functions. This establishes (a).
	
	\smallskip

	Now define
	\[
		\cB'_n f(x) = \begin{cases} 
			n\int_0^{\frac1n} f (\xi) \dd \xi &  \text{if }  x < 1/n\\
			f(x) &  \text{if }  x \ge 1/n
		\end{cases}.
	\]
	Again, it is straightforward to show that $\cB_n'$ is positive and preserves the set of increasing functions. In contrast to $\cB_n$, $\cB_n'$ is not mapping continuous functions to continuous functions. It is, however, continuous for the strict topology on the space of bounded measurable functions.
		
		Define $T'_n\colon C_b[0,\infty) \rightarrow C^\infty_b[0, \infty)$ where  $T'_n  = \cM_n \cB'_n$. By construction, $T'_n$ is  also positive, strictly continuous and preserves the sets of increasing functions. This establishes (b).

	\smallskip

	Finally define
	\[
		\cB''_n f(x) = \begin{cases} 
			0 &  \text{if }  x < 1/n\\
			f(x) &  \text{if }  x \ge 1/n
		\end{cases} 
	\]
and set
$T''_n\colon C_b[0,\infty) \rightarrow C^\infty_b[0, \infty)$ where  $T''_n  = \cM_n \cB''_n$.  As above, we find that $T''_n$ is a positive, strictly continuous operator.

\bigskip

	By Proposition \ref{proposition:equi_continuity_via_strong_continuity}, the operators $\{\cB_n\}$ are equi-continuous.  This in combination with the strong continuity of the operators $M_n$, see property \eqref{enum_item:strong_continuity_Mn} on page \pageref{enum_item:strong_continuity_Mn}, implies that the family $\{T_n\}$ is strongly continuous. 
	
	As the operators $\{\cB_n'\}$ and $\{\cB_n''\}$ are not mapping into the set of continuous functions, we give a direct proof. Fix $f \in C_b(\cX)$. By definition, it follows that $\vn{\cB_n' f - f} \rightarrow 0$. Using property (c) of the family $\{\cM_n\}$ and the decomposition 
	\begin{equation*}
	T_n' f -f = \cM_n(\cB_n' f - f) + \cM_n f - f,
	\end{equation*}
	we find that the family $\{T_n'\}$ is strongly continuous for the strict topology. On the subspace 
	\begin{equation*}
	\left\{f \in C_b(\cX) \, \middle| \, f(0) = 0 \right\},
	\end{equation*}
	the same proof can be carried out for $T_n''$.

	The final property to verify is the correct boundary behaviour, i.e.~in this setting, property (d). This is immediate after noting that
	$\cB_n f$ is linear on $(-n^{-1},n^{-1})$,
	 $\cB'_n f$ is constant on $(-n^{-1},n^{-1})$ and $\cB''_n f$  vanishes on $(-n^{-1},n^{-1})$.

	\bigskip

	In order to prove the second proposition, note that for $f' = h$, 
	$\cB'_n h$ is the weak derivative of $\cB_n f$ and
	$\cB''_n h$ is the weak derivative of $\cB'_n f$, so if $f' \ge h$
	\[
		( T_n f)' \ge T_n' h \quad\text{and}\quad ( T_n' f)' \ge T_n'' h .
	\]
	Furthermore, 
	if $f'' \ge h$
	\[
		(  T_n f)'' \ge T_n'' h.
	\]
\end{proof}

\begin{proposition}\label{proposition:refined_approximation_operators_1d_diffusion}
Consider the setting of Proposition \ref{proposition:approximation_operators_1d_diffusion} and suppose for each element $e \in \partial \cX \cap \cX$ we have a constant $\gamma_e \in \bR$.

There are operators $\hat T'_n$ with properties (b), (c) and (d) of $T'_n$ in Proposition \ref{proposition:approximation_operators_1d_diffusion} and a strongly continuous family (on $C_b(\cX)$) of operators $\{\hat T''_n\}$ with $\hat T_n'' C_b(\cX) \subseteq C_b^\infty(\cX)$ such that for each $e \in \partial \cX \cap \cX$:
\begin{equation*}
\gamma_e(\hat T'_n f)'(e) = (\hat T'_n f)''(e) \text{ and } \gamma_e(\hat T''_n f)(e) = (\hat T''_n f)'(e).
\end{equation*} 
In addition, we have that if $f \in C_b^1(\cX)$ and $f'  \ge h \ge 0$, then $(\hat{T}'_n f)'  \ge  \hat T''_n h \ge 0$.

\end{proposition}

\begin{proof}
	Again we only consider the setting of $\cX = [0,\infty)$. Define
	\[
		\hat \cB'_n f(x) = \begin{cases} 
			f(0) + \frac{\frac12\gamma_e x^2 + x} {\frac12\gamma_e/n  + 1}
			n (f(\tfrac1n) - f(0))  &  \text{if }  x \le 1/n\\
			f(x) &  \text{if }  x > 1/n
		\end{cases}
	\]
	and set $ \hat T'_n = \cM_n \hat\cB'_n$, and also define
	\[
		\hat \cB''_n f(x) = \begin{cases} 
			 \frac{\gamma_e x + 1} {\frac12\gamma_e/n  + 1}
			n \int_0^{\frac1n} f (\xi) \dd \xi   &  \text{if }  x \le 1/n\\
			f(x) &  \text{if }  x > 1/n
		\end{cases}
	\]
	and set $\hat T''_n = \cM_n \hat\cB''_n$. Again $\hat\cB''_n f'$ is the weak derivative of $\hat\cB_n f$, which establishes the final claim of the proposition.	
	
	We verify the properties of $\hat T_n'$. As we act with $\cM_n$ last, $\hat T'_n C_b(\cX) \subseteq C_b^\infty(\cX)$. Next, note that $\hat \cB'_n f(1/n) = f(1/n)$, implying that $\hat \cB'_n f$ is continuous. A elementary computation shows that $\hat{\cB}_n'$ is strictly continuous on $C_b(\cX)$. Using Proposition~ \ref{proposition:equi_continuity_via_strong_continuity}, we conclude that $\{\hat{T}_n'\}$ is a strongly continuous family on $C_b(\cX)$. 
	
	Finally, we verify the that the first and second derivative have the appropriate quotient at $0$. For $x \leq \tfrac{1}{n}$:
	\begin{align*}
	\left(\hat \cB'_n f\right)'(x) & = \frac{\gamma_e x + 1} {\frac12\gamma_e/n  + 1}	n (f(\tfrac1n) - f(0)), \\
	\left(\hat \cB'_n f\right)''(x) & = \frac{\gamma_e} {\frac12\gamma_e/n  + 1}	n (f(\tfrac1n) - f(0)),
	\end{align*}
	so that a convolution with the symmetric function $k_{n^{-1}}$ yields
	\begin{align*}
	\left(\hat T'_n f\right)'(0) & = \frac{1} {\frac12\gamma_e/n  + 1}	n (f(\tfrac1n) - f(0)), \\
	\left(\hat T'_n f\right)''(0) & = \frac{\gamma_e} {\frac12\gamma_e/n  + 1}	n (f(\tfrac1n) - f(0)),
	\end{align*}
	implying $\gamma_e(\hat T'_n f)'(0) = (\hat T'_n f)''(0)$.
	
	\smallskip
	
	The verification of the properties of $\hat T_n''$ follows by using similar steps.

\end{proof}

\subsection{Potential applications for degenerate operators: The Wright-Fisher diffusion}

	To conclude our section on one-dimensional diffusion operators, we give some hint at further applications of our main results, by sketching a possible application for the Wright-Fisher diffusion. The discussion in the present section will be non-rigorous as we are lacking the possibility of allowing absorbing boundaries for exit-non entrance boundaries in Theorem \ref{theorem:domain_diffusion_operator} and are missing a regularity result of the type of Lemma \ref{lemma:gaining_regularity}. %
	
	The Wright-Fisher diffusion on $[0,1]$ has a generator $A$ that is a subset of
	\begin{equation*}
	\tilde{A}f(x) := \frac{1}{2} x(1-x) f''(x).
	\end{equation*}
	We aim to study the the preservation of the set of increasing or convex functions, i.e. $\Phi_1 f = f'$ or $\Phi_2 f = f''$. A computation of $\Phi_1 \tilde{A}$ and $\Phi_2 \tilde{A}$ leads to the definition of $(\tilde{B}_1,C_1)$ and $(\tilde{B}_2,C_2)$:
	\begin{align*}
	\tilde{B}_1 f(x) & = \frac{1}{2} x(1-x) f''(x) + \left(\frac{1}{2} - x\right) f'(x), & C_1 f(x) & = 0, \\
	\tilde{B}_2 f(x) &= \frac{1}{2} x(1-x) f''(x) + \left(1 - 2x\right) f'(x) - f(x), & C_2 f(x) & = 0.
	\end{align*} 
	Straightforward computation, taking base-point $z = \tfrac{1}{2}$, yields that the speed measure $s$ for $\tilde{A}$ and the speed measures $s_1,s_2$ for $\tilde{B}_1, \tilde{B}_2$ satisfy
	\begin{equation*}
	s(x) = 1, \qquad \left(\frac{\dd}{\dd x} s_1(x)\right)^{-1} = 4x(1-x),  \qquad \left(\frac{\dd}{\dd x} s_2(x)\right)^{-1} = 16x^2(1-x)^2.
	\end{equation*}
	Furthermore, one can show that the boundaries for $\tilde{A}$ are exit-non-entrance, whereas for $\tilde{B}_1$ and $\tilde{B}_2$ they are entrance-non-exit.

		 Theorem \ref{theorem:domain_diffusion_operator} allows for only one set of possible boundary conditions for $\tilde{A},\tilde{B}_1$ and $\tilde{B}_2$: that is, a killing boundaries for $\tilde{A}$ and reflecting boundaries for $\tilde{B}_1$ and $\tilde{B}_2$. We believe, however, that the operator $\tilde{A}$ can be equipped with absorbing boundaries as well:
		 \begin{equation*}
		 \cD(A) := \left\{ f \in C_b[0,1] \cap C^2(0,1) \, \middle| \, \lim_{x \downarrow 0} x f^{(2)}(x) = 0, \lim_{x \uparrow 1} (1-x) f^{(2)}(x) = 0 \right\},
		 \end{equation*}
		 so that $Af(x) = x(1-x)f''(x)$ for $x \in (0,1)$ and $Af(x) = 0$ for $x \in \{0,1\}$. This choice would correspond to the natural situation that a process that hits the boundary stays at the boundary instead of moving to a graveyard state.
		
		\smallskip

		The map $\Phi_1 \subseteq C_b[0,1] \times C_b[0,1]$ maps the domain of $A$ into the domain of $\tilde{B}_1$ equipped with its only choice of boundaries. Therefore, this leads to
		\begin{multline*}
		\cD_1 = \left\{ f \in C_b[0,1] \cap C^3(0,1) \, \middle| \,  \lim_{x \downarrow 0} x f^{(2)}(x) = 0, \lim_{x \uparrow 1} (1-x) f^{(2)}(x) = 0, \right. \\
		\left. \tilde{B}_1 \Phi_{1} f \in C_b[0,1] \right\},
		\end{multline*}
		It is not clear to us that the resolvents of $A$ map $\cD_1$ into itself, this is in particular due to the condition $\tilde{B}_1 \Phi_{1} f \in C_b[0,1]$. 
		
		\smallskip
		
		The map $\Phi_2$ maps the domain of $A$ to a set of functions for which 
		\begin{equation*}
		\lim_{x \downarrow 0} xh(x) =  0, \qquad \lim_{x \uparrow 1} (1-x) h(x) = 0,
		\end{equation*}
		which for functions $h \in C_b[0,1]$ is always satisfied. Thus, the map $\Phi_2 \subseteq C_b[0,1] \times C_b[0,1]$ does not map into the natural domain of $\cD(B_2)$ as an operator on $C_b[0,1]$. Convex functions, however, are also characterized by having non-negative second derivative on the interior of their domain. Thus, we can consider $\Phi_2$ as a subset of $C_b[0,1] \times C_b(0,1)$ and consider $B_2$ as an operator on $C_b(0,1)$ on which it needs no boundary conditions. Thus, we can work with 
		\begin{multline*}
		\cD_1 = \left\{ f \in C_b[0,1] \cap C^4(0,1) \, \middle| \,  \lim_{x \downarrow 0} x f^{(2)}(x) = 0, \lim_{x \uparrow 1} (1-x) f^{(2)}(x) = 0, \right. \\
		\left. \tilde{B}_2 \Phi_{2} f \in C_b(0,1) \right\}.
		\end{multline*}
		Again, it is not clear that the resolvents of $A$ map $\cD_2$ into itself.
		
		In both cases, however, we do have so that $\cD_i \subseteq \cD(\Phi_i A) \cap \cD(B_i \Phi_i) \cap \cD(C_i \Phi_i)$ and $\Phi_i A = B_i \Phi_i + C_i \Phi_i$.

\section{ Diffusion processes on \texorpdfstring{$\bR^d$}{Rd}} \label{section:Diffusion_processes}

\subsection{Integral orders for diffusion processes}

In this section, we characterise order preservation properties for diffusion processes on $\cX = \bR^d$. The orders that we consider are induced by operators $\Phi \subseteq C_b(\bR^d) \times (C_b(\bR^d))^{ |\cI|}$ that are given by $\cD(\Phi) = C_b^2(\bR^d)$ and $\Phi f = (\partial_\alpha f, \alpha \in \cI)$, for a set of multi-indices ${\mathcal I} \subseteq \{0,1\}^d$ that satisfy the following assumption.
 \begin{assumption}\label{assumption:index_set_multi_d_diffusion}
	The index set $\cI \subseteq \{0,1\}^d$ is such that every element $\alpha \in \cI$ has size $1$ or $2$, i.e.~$\|\alpha\|_1 \in \{1,2\}$.
\end{assumption}

Elements of $\alpha \in \cI$ are suggestively denoted $i$ for $\vn{\alpha}_1 = 1$ and $\{ij\}$ for $\vn{\alpha}_1= 2$.

The two main examples that we will consider below are the component-wise increasing functions and supermodular functions that are generated by the sets $\cI := \{1,\dots,d\}$ and $\cI :=\left\{\{kl\} \, \middle| \, 1 \leq k < l \leq d \right\}$.

\smallskip

The processes that we consider are non-degenerate diffusion processes. This means that we consider processes with generators $(A,\cD(A))$, see Proposition \ref{proposition:diffusion_semigroup_is_SCLE} below, with $C_b^2(\bR^d) \subseteq \cD(A)$ so that if $f \in C_b^2(\bR^d)$:
\begin{equation}\label{eqn:elliptic_operator}
\begin{aligned}
Af(x) & =  \frac{1}{2} \sum_{i,j=1}^d a_{ij}(x) \partial_i \partial_j f(x) + \sum_{i = 1}^d b_i(x) \partial_i f(x), \\
& = \sum_{\vn{\gamma}_1 \le 2} g_\gamma \partial_\gamma f,
\end{aligned}
\end{equation}
where for each $x \in \bR^d$ the matrix $\{a_{ij}(x)\}_{1 \le i,j \le d}$ is symmetric non-negative definite and where the second line is written using multi-indices, summing over indices $\gamma \in \{0,1,2\}^d$ with $g_0\equiv 0$. Even though existence theory for Markov processes and semigroups allows for more general conditions, we assume the following:
 	
 	\begin{assumption}\label{assumption:diffusion_uniform_ellipticity}
 		Consider an operator $(A,\cD(A))$ as in \eqref{eqn:elliptic_operator}. There is $\delta > 0$ such that $a_{ij}, b_i \in C^{2,\delta}_b(\bR^d)$ for all $i,j \in \{1,\dots,d\}$. In addition suppose that $a$ is uniformly elliptic: there is a $\kappa > 0$ such that for all $\xi \in \bR^d$: 
 		\begin{equation*}
 		\inf_{x \in \bR^d}  \sum_{i,j = 1}^d a_{ij}(x) \xi_i \xi_j \geq \kappa|\xi|^2.
 		\end{equation*}
 	\end{assumption}

 	From the uniform ellipticity condition on $A$ a direct analogue of Lemma \ref{lemma:gaining_regularity} can be obtained with \cite{Kr96}. 
	 		
 	\smallskip
 	
 	Next we show that operators of this type generate Feller processes.
 	
 	\begin{proposition} \label{proposition:diffusion_semigroup_is_SCLE}
 		Let $(A,\cD(A))$ satisfy Assumption \ref{assumption:diffusion_uniform_ellipticity}. Then $A$ generates a SCLE semigroup on $(C_b(\bR^d),\beta)$. The semigroup corresponds to the transition semigroup of a Feller process.
 	\end{proposition}
 
 \begin{proof}
 	By Theorem 8.1.6 in \cite{EK86} the restriction of $A$ to $C_0(\bR^d)$ generates a strongly continuous semigroup $\{\tilde{S}(t)\}_{t \geq 0}$ on $(C_0(\bR^d),\vn{\cdot})$ that corresponds to a Markov process on $D_{\bR^d}(\bR^+)$. In the terminology of \cite{EK86}, this semigroup is Feller, meaning that it satisfies the conditions of the second part of Theorem \ref{theorem:transition_semigroup_is_strongly_continuous}. It follows that this semigroup $\{\tilde{S}(t)\}_{t \geq 0}$ has an extension to a SCLE semigroup $\{S(t)\}_{t \geq 0}$ on $C_b(\bR^d)$. In addition, we find that the domain of the generator contains the strict closure of $C_c^2(\bR^d)$, which contains the set $C_b^2(\bR^d)$.
 \end{proof}

\subsection{Stochastic monotonicity}

We proceed with our main result regarding the stochastic monotonicity of diffusion semigroups.

\begin{theorem}\label{theorem:diffusion_monotonicity} 
Let $\cI$ be an index set satisfying Assumption \ref{assumption:index_set_multi_d_diffusion} and let $\cF_{+,0}$ and $\cF_+$ be generated by the partial derivatives with indices in $\cI$. Let $A$ satisfy Assumption \ref{assumption:diffusion_uniform_ellipticity}. If for all $\alpha \in  {\mathcal I}$ and for all  ${\beta< \alpha}$ and $ \gamma \in  \{0,1, 2\}^d$ with $\vn{\gamma}_1 \in \{1, 2\}$ the conditions
\begin{equation}\label{gammabed}
\not\exists \zeta \geq 0 \colon  \beta+\gamma = \zeta + \alpha
\quad\Longrightarrow\quad\begin{cases}
\partial_{\alpha-\beta} g_\gamma \ge 0 &\text{if }\beta+\gamma \in  \cI,\\
\partial_{\alpha-\beta} g_\gamma = 0 & \text{else}
\end{cases}
\end{equation}
hold, then $\{S(t)\}_{t\ge0}$ propagates $\cF_+$.

\smallskip

For each $\alpha \in \cI$ define $\widetilde{B}^{(\alpha)}$ as the second order differential operator
\begin{equation*}
	\tilde B^{(\alpha)} f = \sum_{\|\zeta\|_1 \le 2}\tilde{g}^{(\alpha)}_\zeta \partial_\zeta f,
	\end{equation*}
	where the functions $\tilde{g}^{(\alpha)}_\zeta$
are expressed in terms of the coefficients in the multi-index notation for $A$ as in \eqref{eqn:elliptic_operator}:
\begin{equation*}
\tilde g_\zeta^{(\alpha)} = \sum_{\|\gamma\|_1 \le 2}   \sum_{\beta\le \alpha\atop \beta + \gamma = \zeta + \alpha}\binom{\alpha}{\beta} \partial_{\alpha-\beta} g_\gamma.
\end{equation*}
Let $\widetilde{T}^{(\alpha)}(t)$ be the semigroup generated by $B^{(\alpha)}$. Then we have that for each $t \geq 0$ and $f \in \cF_{+,0}$ that $\partial_\alpha S(t) f \geq \widetilde{T}^{(\alpha)}(t) (\partial_{\alpha} f)$.
\end{theorem}

Before giving the proof of this result, we give two main examples to make clear how to verify \eqref{gammabed}.
	
\begin{example}
	If $\cI = \{1, \dots, d\}$ or $\Phi f = \operatorname{grad}f$, then $\Phi f \ge 0$ characterizes component-wise increasing functions. If $a$ and $b$ fulfill Assumption \ref{assumption:diffusion_uniform_ellipticity} and
	\begin{subequations}\label{diffmon}
		\begin{gather}
		\text{$a_{ij}$ depends only on $x_i$ and $x_j$},\\
		\text{$b_{i}$ is increasing in $x_j$ for all $j \ne i$},
		\end{gather}
	\end{subequations}
	then the conclusions of Theorem \ref{theorem:diffusion_monotonicity} hold. In particular, we obtain that $\{S(t)\}_{t\geq 0}$ is stochastically monotone.
\end{example}

\begin{proof}
	For all $\alpha \in \cI$, we have $\vn{\alpha}_1 = 1$, so $\beta < \alpha$ implies $\beta = 0$. First consider all $\gamma$ with $\vn{\gamma}_1 = 2$, i.e.~$\gamma = \{i j\}$ (where possibly $i = j$). If $\alpha \notin \{i,j\}$, then there is no $\zeta \geq 0$ such that $\gamma = \zeta + \alpha$, which implies we must assume $\partial_\alpha g_{\{i j\}} = \partial_\alpha a_{ij} = 0$ as $\beta + \gamma = \{ij\} \notin \cI$.
	
	Next consider all $\gamma$ with $\vn{\gamma}_1 = 1$, i.e.~$\gamma = i$. If $\alpha \neq i$, then there is no $\zeta \geq 0$ such that $\gamma = \zeta + \alpha$, which implies we must assume $\partial_\alpha g_{i} = \partial_\alpha b_i \geq 0$ as $\beta + \gamma = \{i\} \in \cI$.
\end{proof}

\begin{example}
	$\cI = \left\{\{kl\} \, \middle| \, 1 \le k < l \le d\right\}$ or $\Phi f = (\partial_k \partial_l f)_{1 \le k < l \le d}$, then $\Phi f \ge 0$  characterizes smooth supermodular functions.
	If $a$ and $b$ fulfill Assumption \ref{assumption:diffusion_uniform_ellipticity} and additionally
	\begin{subequations}
		\begin{gather}
		\text{$a_{ij}$ depends only on $x_i$ and $x_j$},\\
		\text{$b_{i}$ depends only on $x_i$},
		\end{gather}
	\end{subequations}
	then the conclusions of Theorem \ref{theorem:diffusion_monotonicity} hold. In particular, $\{S(t)\}_{t\ge0}$ propagates the supermodular functions.
\end{example}

\begin{proof}
	 Write $\alpha = \{k l \}$ with $k \neq l$. We first consider $\|\beta\|_1 = 1$ and set without loss of generality $\beta = k$. The following two choices of $\gamma$ give our conditions:
	\begin{description}
		\item[$\gamma = \{i j\}:$] If $l \in \{i,j\}$, then we can find $\zeta \geq 0$ such that $\beta + \gamma = \zeta + \alpha$. If $l \notin \{i,j\}$, this is not possible. As $\vn{\gamma + \beta}_1 = 3$, we find that $\partial_{{\alpha - \beta}} g_\gamma = \partial_l a_{ij} = 0$.
		\item[$\gamma = k:$] There is no $\zeta$ such that $\beta + \gamma = \zeta + \alpha$. As $\{k k\} \notin \cI$, we find that $\partial_{\alpha - \beta} g_\gamma = \partial_l b_k = 0$.
	\end{description}


Any other combination of $\beta$ and $\gamma$ gives conditions that are implied by the two above.

\end{proof}

\begin{proof}[Proof of Theorem \ref{theorem:diffusion_monotonicity}]
To establish that $S(t)$ propagates $\cF_+$, we apply Theorem \ref{theorem:main_theorem}. We choose $\cD := C_b^4(\bR^d)$. We start by a calculation of $\Phi A$ for $f \in \cD$ to identify appropriate operators $B$ and $C$. Using the multi-index notation of the Leibniz rule
$
 \partial_\alpha (g f) = \sum_{\beta\le \alpha} \binom{\alpha}{\beta} \partial_{\alpha-\beta} g\; \partial_\beta f,
$
\begin{align*}
\partial_\alpha A f 
= \sum_{{\|\gamma\|_1 \le 2} }
\sum_{\beta\le \alpha} \binom{\alpha}{\beta} \partial_{\alpha-\beta} g_\gamma\; \partial_{\beta +\gamma } f.
\end{align*}
We try to write as many terms in terms of  $\{\partial_{\alpha} f\}_{\alpha \in \cI}$. This can be achieved by finding all $\beta \leq \alpha$ such that $\beta + \gamma \geq \alpha$. This motivates the definition of $\widetilde{B}^{(\alpha)}$. All terms in $\partial_\alpha Af$ that are left are written in terms of $C^{(\alpha)} \Phi$, where $C^{(\alpha)}$ is given on any multi-indexed vector $u = (u_{i})_{i \in \cI}$ by
\begin{equation*}
C^{(\alpha)} u := \sum_{{\|\gamma\|_1 \le 2} } \sum_{{\scriptstyle\beta\le \alpha\atop {\scriptstyle
 			\beta + \gamma \ne \zeta + \alpha  \atop \scriptstyle (\text{for any }\zeta \geq 0)}}}\binom{\alpha}{\beta} \partial_{\alpha-\beta} g_\gamma\;  u_{\gamma + \beta}.
\end{equation*}
This defines two operators $B, C$ by 
\begin{equation*}
B u  = \left(\tilde{B}^{(\alpha)} u_{\alpha}\right)_{\alpha \in \cI}, \qquad C u  = \left(C^{(\alpha)} u\right)_{\alpha \in \cI}.
\end{equation*} 
Note that $B$ operates on the diagonal. In addition, if $\beta + \gamma = \zeta +  \alpha $ for $\vn{\zeta}_1 = 2$ then it holds $\tilde g_\zeta^{{(\alpha)}} = g_\zeta$   and the operators $\tilde B^{(\alpha)}$  differ from $A$ only in the drift and the killing terms. These terms are bounded by assumption. Thus each component generates a positive SCLE diffusion semigroup, so that also $B$ generates a positive SCLE semigroup. This implies that $B$ is resolvent positive.

The operator $C$ has a more complicated matrix structure, but each component is continuous. The main condition in the theorem together with $\gamma_0 = 0$ implies that $C$ is positive.

We conclude that on the set $\cD$ we have
\begin{equation*}
\Phi A f = B \Phi f + C \Phi f
\end{equation*} 
for a resolvent positive operator $B$ and a positive continuous operator $C$.

It remains to verify conditions (b) and (c) of Theorem \ref{theorem:main_theorem}. Condition (b) follows from Proposition \ref{proposition:core_approximation_operators} and Lemma \ref{lemma:multi_d_diffusion_operators} below. Condition (c)  follows from Assumption \ref{assumption:diffusion_uniform_ellipticity}, i.e.~$a,b \in C^{2+\delta}_b(\bR^d)$, and Theorem 4.3.2 in \cite{Kr96}. 

\smallskip

The second part of the result follows by Theorem \ref{theorem:lower_bound} whose conditions can be checked as in the proof of Theorem \ref{theorem:1d_diffusion_increasing_functions}, making the appropriate changes, e.g.~changing the use of Lemma \ref{lemma:gaining_regularity} to Theorem 4.3.2 in \cite{Kr96}.
Thus, Condition (c) of Theorem \ref{theorem:lower_bound} follows from Corollary \ref{corollary:core_approximation_operators} by using the set of approximation operators $\{(T_n,T_n')\}$ defined below.

 A final property that needs to be checked is that $f \in \cF_{+,0}$ implies $S(t)f \in \cF_{+,0}$. This, however, follows from Theorem 8.2.1 (ii) in \cite{Kr96} as $\cF_{+,0}\subseteq C^2_b(\bR^d)$.

\end{proof}

Using convolutions as in the proof of Proposition \ref{proposition:approximation_operators_1d_diffusion} the following result is immediate.

\begin{lemma}\label{lemma:multi_d_diffusion_operators}
Let $\cI$ be an index set satisfying Assumption \ref{assumption:index_set_multi_d_diffusion} and let $\cF_{+,0}$ and $\cF_+$ be generated by the partial derivatives in $\cI$.

There exits positive continuous approximation operators $\{T_n\}_{n \geq 1}$ for $(\Phi, \cF_{+,0},\cD)$. In addition, there are operators ${T'}_{n,\alpha}$, $\alpha \in \cI$ such that for $i \in \{0,1\}$ we have if $(-1)^i \partial f \geq (-1)^i h$, then $(-1)^i \partial_\alpha (T_n f) \geq (-1)^i {T'}_{n,\alpha} h$.
\end{lemma}

\subsection{A comparison result for diffusion processes}

	We proceed with preservation of order. The theorem complements the results  \cite{BeRu2007} obtained using martingale methods. 
	
	\begin{theorem}
		\label{theorem:diffusion_comparison}  
		
		Let $\cI$ be an index set satisfying Assumption \ref{assumption:index_set_multi_d_diffusion} and let $\cF_{+,0}$ and $\cF_+$ be generated by the partial derivatives with indices in $\cI$. Denote by $\lesssim$ the stochastic order induced by $\cF_+$.
		
		Let $\eta^{(1)},\eta^{(2)}$ be two diffusion processes  with generators $A^{(z)}$, $z \in \{1,2\}$ and let $A$ be a third generator, all satisfying Assumption \ref{assumption:diffusion_uniform_ellipticity} having coefficients $a_{ij}^{(z)}, a_{ij}, b_i^{(z)},b_i$.
		
		Suppose that the coefficients $a_{ij},b_i$ satisfy \eqref{gammabed} in Theorem \ref{theorem:diffusion_monotonicity} and that
		\begin{equation} \label{eqn:thm_multi_d_diffusion_comparison}
		\begin{aligned}
		b_i^{(1)} \leq b_i \leq b_i^{(2)} \text{ for all } i \in \cI & \text{ and }  b_i^{(1)} = b_i = b_i^{(2)} \text{ otherwise,}\\
		a_{ij}^{(1)} \leq a_{ij} \leq a_{ij}^{(2)} \text{ for all } \{ij\} \in \cI & \text{ and }  a_{ij}^{(1)} = a_{ij} = a_{ij}^{(2)} \text{ otherwise,}
		\end{aligned}
		\end{equation}
		then  $S^{(1)}(t) f \leq S(t) f \leq S^{(2)}(t) f$ for $f \in \cF_+$. In addition we have $\eta^{(1)}(0) \lesssim \eta^{(2)}(0)$ then $\eta^{(1)}(t) \lesssim \eta^{(2)}(t)$ for all $t \geq 0$.
	\end{theorem}

\begin{proof}
	We apply Theorem \ref{theorem:comparison} and Corollary \ref{corollary:preservation_of_order}. Note that the operator $A$ is as in Theorem \ref{theorem:diffusion_monotonicity} and there are $B$ and $C$ such that the conditions for Theorem \ref{theorem:main_theorem} are satisfied. 
	Recall that $\cD = C_b^4(\bR^d)$. 
	
	Here we choose in addition $\cD^* := C_b^2(\bR^{|\cI|})$. Note that condition (a) of Theorem \ref{theorem:comparison} is satisfied by our choice of $A$. Condition (b) follows by Proposition \ref{proposition:core_double_approximation_operators} and the approximation operators $T_n$ of Lemma \ref{lemma:multi_d_diffusion_operators}. Condition (c) follows by Theorem 4.3.2 in \cite{Kr96} as the coefficients of the operator $B$ are $C_b^{0,\delta}$ by Assumption \ref{assumption:diffusion_uniform_ellipticity} on $A$. 
\end{proof}

\section{Interacting particle systems} \label{section:Interacting_particle_systems}

A   third class of examples that we will consider are discrete interacting particle systems on $\cX := \{0,1\}^S$  or $\cX := \{-1,1\}$, where $S$ is some countable set or graph. Elements in $\cX$ are denoted with Greek letters $\sigma,\zeta$, whereas we use $i,j,x,y \in S$ to denote a location in $S$. We write $\sigma_i$ for the value of $\sigma\in \cX$ at site $i \in S$.

On $\cX$, we consider Feller processes $\left\{\sigma(t) \right\}_{t \geq 0}$. Three important examples of such systems are as follows.
\begin{itemize}
\item \textit{The contact process}: each site $i \in S$ represents an individual that can either be healthy, $\sigma_i = 0$, or infected, $\sigma_i = 1$. The dynamics are as follows: infected sites infect healthy neighbouring sites and infected sites heal at some rate independent of the states of the neighbours.
\item \textit{Voter model:} each site $i \in S$ represents an individual that has an opinion $\sigma_i \in \{0,1\}$, that changes at an exponential rate proportional to the number of neighbours with a different opinion.
\item \textit{Glauber dynamics for the Ising model:}  Usually this process is defined on $\{-1,1\}^S$. Each site represents a magnetic spin that changes under the influence of Glauber dynamics.
\end{itemize} 

We will reprove and slightly improve a classical result  on stochastic monotonicity by Liggett, Theorem III.1.5 in \cite{Li85}, there proven via coupling methods.

We start with a small introduction of the order generated by increasing functions and that of generators.

\subsection{Preliminaries: orders and generators}
	
	We equip $\cX$ with the partial order $\le$ defined for $\sigma, \xi \in \cX$ by $\sigma\leq \xi$ if $\sigma_i \leq \xi_i$ for all $i \in S$.
	We say that $f \in C_b(\cX)$ is monotone if $\sigma \leq \xi$ implies $f(\sigma) \leq f(\xi)$ and denote the set of monotone functions by $\cM$. Our goal is to prove the preservation of monotonicity. For this we need to  characterize the monotone functions $\cM$ in terms of an operator $\Phi$. Without loss of generality, we assume that $\cX = \{0,1\}^S$. For $x \in S$ and $r \in \{0,1\}$, we write
	\begin{equation*}
	\sigma[x,r]_j = \begin{cases}
	\sigma_j & \text{if } j \neq x, \\
	r & \text{if } j = x,
	\end{cases}
	\end{equation*}
	for the configuration obtained from $\sigma$ by changing the spin at site $x \in S$ to $r$ and define
	\begin{equation*}
	\Phi \colon C(\cX) \rightarrow C_0(\cX \times S), \qquad \Phi f(\sigma,x) := f(\sigma[x,1]) - f(\sigma[x,0]).
	\end{equation*}
	As always, we write $\cF_{+,0} := \left\{f \in C_b(\cX) \, | \, \Phi f \geq 0 \right\}$ and $\cF_+ := \overline{\cF_{+,0}}$. Our first result is that indeed $\Phi$ characterizes the monotone functions.
	
	\begin{proposition} \label{proposition:IPS_monotone_functions}
		We have $\cM =\cF_+$.
	\end{proposition}
	
	\begin{proof}
		If $f \in D\cap\cM$, it is clear that $f \in \cF_{+,0}$. Suppose now that $f \in D \cap \cF_{+,0}$, we prove that it is in $\cM$. Pick $\sigma,\xi$ such that $\sigma \geq \xi$. Let $\phi \colon \{1,2,\dots\} \rightarrow S$ be a bijection and denote by $(\sigma\xi)^{j}$ the configuration 
		\begin{equation*}
		(\sigma\xi)^{j}_{\phi(i)} := \begin{cases}
		\sigma_{\phi(i)} & \text{for } i > j, \\
		\xi_{\phi(i)} & \text{for } i \leq j.
		\end{cases}
		\end{equation*}
		Then by the fact that $f \in D$, we have
		\begin{equation*}
		f(\sigma) - f(\xi) = \sum_{j = 0}^\infty f((\sigma\xi)^{j}) - f((\sigma\xi)^{j+1}). 
		\end{equation*}
		Because $f \in \cF_{+,0}$, we conclude that $f(\sigma) - f(\xi) \geq 0$ and $f \in \cM$.
		
		\smallskip
		
		The proof that $\overline{D \cap \cM} = \cM$ is straightforward, which implies that $\cM = \cF_+$.
	\end{proof}

Next, we introduce a class of operators on $\cX$ that generate interacting particle systems. These systems will make jumps on the set $\cX$ by changing the state at one site $i \in S$ at a time. The processes, however, will in general have a total jump-rate that equals infinity. If we assume that changes in the configuration `far away in $S$' do not influence the jump-rates of sites for `small' elements in $S$ the Markov process is well-defined, see Theorem \ref{theorem:LiggettI.3.9} proven by Liggett below.

\smallskip

We give the appropriate notation and definitions. For $i \in S$, let $\sigma^i \in \cX$ be the configuration obtained from $\sigma \in \cX$ by changing the $i$-th coordinate:
	\begin{equation*}
	\sigma^i_j = \begin{cases}
	\sigma_{j} & \text{if } j \neq i, \\
	1 - \sigma_{j} & \text{if } j = i.
	\end{cases}
	\end{equation*}
	The rate at which the system makes a transformation from $\sigma$ to $\sigma^i$ is given by a continuous function $\sigma \mapsto c(i,\sigma)$.
	
	For a function $f \in C_b(\cX)$ that depends on a finite number of coordinates in $S$ only, the generator of our interacting particle system is given by
	\begin{equation} \label{eqn:IPS_def_A}
	Af(\sigma) = \sum_i c(i,\sigma)\left[f(\sigma^i) - f(\sigma)\right].
	\end{equation}
	Note that in general $\sum_i \vn{c(i,\cdot)} = \infty$. To make sure that we can associate a Feller process to an extension of $A$ in \eqref{eqn:IPS_def_A} we introduce some additional definitions. 
	
	For $f \in C(\cX)$, define
	\begin{equation*}
	\delta_f(x) = \sup \left\{|f(\sigma) - f(\zeta)| \, \middle| \, \sigma_y = \zeta_y  \text{ for } y \neq x \right\},
	\end{equation*}
	the variation of $f$ at $x \in S$, and the space of test functions (of bounded variation) by
	\begin{equation}
	D = \left\{f \in C_b(\cX) \, \middle| \, \tn{f} := \sum_{x \in S} \delta_f(x) < \infty \right\}. \label{eqn:defD}
	\end{equation}

	Our first condition on $A$, i.e. \eqref{eqn:IPS_summability} below, will make sure that $D$ is contained in the domain of the closure of $A$.  This condition is expressed in terms of the total rate at which the value of the configuration at $i \in S$ changes. To this end, denote $c(i) := \sup\{c(i,\sigma) \, | \, \sigma \in \cX\}$, the maximal rate of $\sigma \mapsto c(i,\sigma)$. 
	
	\smallskip
	
	Our second condition on $A$, \eqref{eqn:IPS_spread_of_correlation} below, makes sure that correlations between individual coordinates in $S$ do not grow to fast. For $i, u \in S$, we define the dependence of $\sigma \mapsto c(i,\sigma)$ on the $u$ variable:
	\begin{equation*}
	c_u(i) = \sup \left\{ \left|c(i,\sigma) - c(i,\hat{\sigma})\right| \, \middle| \, \sigma_y = \hat{\sigma}_y \text{ for } y \neq u \right\}.
	\end{equation*}

	We state the main generation theorem for interacting particle systems.
	
	\begin{theorem}[\cite{Li85}, Theorem I.3.9] \label{theorem:LiggettI.3.9}
		Assume that
		\begin{gather}
		\sup_i c(i) < \infty  \label{eqn:IPS_summability}, \\
		M :=  \sup_{x \in S} \sum_{u \neq x} c_u(x) < \infty.  \label{eqn:IPS_spread_of_correlation}
		\end{gather}
		Then the closure of $A$, which we will also denote by $A$, includes $D$ in the domain and generates a strongly continuous positive contraction semigroup $S(t)$ on $(C(\cX),\vn{\cdot})$ that corresponds to a Feller process $\{\eta(t)\}_{t \geq 0}$ and the semigroup satisfies $\tn{S(t)f} \leq \e^{tM}\tn{f}$ for all $t \geq 0$. 
	\end{theorem}

\subsection{Preservation of order for spin-systems}

	To obtain conditions for the propagation of monotonicity, we aim to apply our main result, Theorem \ref{theorem:main_theorem}. To do so, we need an expression for $\Phi A f$ in terms of two operators $B$ and $C$ on the space $C_b(\cX \times S)$. We introduce these operators first. 
	
	\smallskip
	
	Denote 
	\begin{equation*}
	s(i,\sigma) = \begin{cases}
	1 & \text{if } \sigma_i = 0, \\
	- 1 & \text{if } \sigma_i = 1.
	\end{cases}
	\end{equation*}
	Let $A$ be as in \eqref{eqn:IPS_def_A}. We will show that $\Phi A = B \Phi + C \Phi$ for appropriate versions of
	\begin{align*}
	B g(\sigma,x) & := \sum_i \frac{c(i,\sigma[x,1]) + c(i,\sigma[x,0])}{2}\left[g(\sigma^i,x) - g(\sigma,x)\right], \\
	C g(\sigma,x) & :=\sum_{i \neq x} \frac{c(i,\sigma[x,1])-c(i,\sigma[x,0])}{2}s(i,\sigma)\left[g(\sigma[x,1],i) + g(\sigma[x,0],i)\right]. 
	\end{align*}
	To give appropriate domains for $B$ and $C$, we introduce a space of test functions, analogously to $D$ on $\cX$, on the space $\cX \times S$. 
	
		Denote by 
		\begin{equation*}
		D^{\otimes S} := \left\{f \in C_0(\cX \times S) \, \middle| \, \forall \, x \in S: f(\cdot,x) \in D, \text{ and } \tn{f}_{\infty} := \sup_{x \in S} \tn{f(\cdot,x)} < \infty  \right\}.
		\end{equation*}


	\begin{remark}\label{remark:D_into_D}
		A straightforward calculation yields
		\begin{equation} \label{eqn:D_into_D}
		\delta_i(\Phi f)(\cdot,x) \leq 2\delta_i f, \qquad \forall \, x,i \in S.
		\end{equation}
		This implies that $\Phi$ maps $D$ into $D^{\otimes S}$.
 	\end{remark}

In Lemma \ref{lemma:IPS_trace_generator} below, we show that $B$, is well defined on $D^{\otimes S}$ and generates a positive strongly continuous contraction semigroup $\{T(t)\}_{t \geq 0}$  on $C_0(\cX\times S)$. We have the following result that slightly improves Theorem III.1.5 in \cite{Li85} by giving a lower bound in addition to the preservation of monotonicity.

\begin{theorem} \label{theorem:IPS_stochastic_preservation}
	Let $A$ be as in \eqref{eqn:IPS_def_A} satisfying the conditions of Theorem \ref{theorem:LiggettI.3.9}. 
	
	Suppose in addition that for all $x,i \in S$, $x \neq i$, we have
	\begin{equation*}
	c(i,\sigma[x,1]) \geq c(i,\sigma[x,0]) 
	\end{equation*}
	if $\sigma_i = 0$ and
	\begin{equation*}
	c(i,\sigma[x,1]) \leq c(i,\sigma[x,0]) 
	\end{equation*}
	if $\sigma_i = 1$. 
	
	Then we have that $S(t)\cF_+ \subseteq \cF_+$ for all $t \geq 0$. Additionally, we have for all $f \in \cF_{+,0}$, $h \in C_0(\cX \times S)$ with the property that $\Phi f \geq h \geq 0$ and $t \geq 0$ that $\Phi S(t)f \geq T(t)h$.
\end{theorem}

Before we prove the theorem, we show that for each fixed $x$, the restriction of the operator $B$ to the $x \in S$ coordinate is the generator of an interacting particle system.

\begin{lemma} \label{lemma:IPS_trace_generator}
For each $x \in S$, the operator $B_x: D \rightarrow C(\cX)$ defined by
\begin{equation*}
B_x f(\sigma) := \sum_i \frac{c(i,\sigma[x,1]) + c(i,\sigma[x,0])}{2}\left[f(\sigma^i) - f(\sigma)\right]
\end{equation*}
satisfies the conditions in Theorem \ref{theorem:LiggettI.3.9} and generates a semigroup $\{T_x(t)\}_{t \geq 0}$ on $C_b(\cX)$. It follows that $B$ is well defined on $D^{\otimes S}$ and generates a Markov semigroup $\{T(t)\}_{t \geq 0}$ with the property that $T(t)g(\sigma,x) = (T_x(t) g(\cdot,x))(\sigma)$. 
Finally, we have that for $h \in D^{\otimes,S}$ that $\tn{T(t)h}_\infty \leq e^{Mt} \tn{h}_\infty$.

\end{lemma}

\begin{proof}
For notational convenience, denote 
\begin{equation*}
c^x(i,\sigma) = \frac{c(i,\sigma[x,1]) + c(i,\sigma[x,0])}{2}.
\end{equation*}
It is immediate that $\sigma \mapsto c^x(i,\sigma)$ is continuous for all $i \in S$. Additionally, \eqref{eqn:IPS_summability} is satisfied as $c^x(i) \leq c(i)$. Finally, \eqref{eqn:IPS_spread_of_correlation} is implied by 
\begin{equation} \label{eqn:IPS_bound_on_correlation_coefficients}
M_x := \sup_{y \in S} \sum_{u \neq y} c_u^x(y) \leq \sup_{y \in S} \sum_{u \neq y} c_u(y) \leq M,
\end{equation}
because $c_u^x(y) \leq c_u(y)$ for all $u, y \in S$. 

A straightforward calculation shows that the semigroup $\{T(t)\}_{t \geq 0}$ is strongly continuous for the norm and that $D^{\otimes S}$ is in the domain of the generator and that on this set, the generator equals $B$. 
The final claim follows immediately from the uniform bound in \eqref{eqn:IPS_bound_on_correlation_coefficients}.
\end{proof}

\begin{proof}[Proof of Theorem \ref{theorem:IPS_stochastic_preservation}]
We first check the conditions for Theorem \ref{theorem:main_theorem}. We pick $(E,\tau_E) = (C(\cX),\vn{\cdot})$, $F = (C_0(\cX \times S),\vn{\cdot})$ and $\cD = D$. This space indeed satisfies $D \subseteq \cD(\Phi A) \cap \cD(B \Phi) \cap \cD(C \Phi)$ by the observation made in \eqref{eqn:D_into_D}. Note that $B$ generates a positive semigroup by Lemma \ref{lemma:IPS_trace_generator}, which implies that $B$ is resolvent positive. The positivity of $C$ is clear by the assumption of the theorem. 

\smallskip

We proceed by proving that $\Phi Af = B \Phi f + C \Phi f$ for $f \in D$. First note that for all $x, i \in S$ and $f \in D$
\begin{equation*}
\Phi f(\sigma^i,x) - \Phi f(\sigma,x) = f(\sigma[x,1]^i) - f(\sigma[x,1]) - \left[f(\sigma[x,0]^i) - f(\sigma[x,0])\right].
\end{equation*}
$\Phi Af(\sigma,x)$ can be expressed as a sum over $i \in S$. For each $i$, a straightforward calculation yields that the summand equals
\begin{align*}
& c(i, \sigma[x,1]) \left[f(\sigma[x,1]^i) - f(\sigma[x,1])\right] - c(i,\sigma[x,0]) \left[f(\sigma[x,0]^i) - f(\sigma[x,0])\right] \\
& \qquad = \frac{c(i,\sigma[x,1]) + c(i,\sigma[x,0])}{2}\left[\Phi f(\sigma^i,x) - \Phi f(\sigma,x)\right] \\
& \qquad \qquad + \frac{c(i,\sigma[x,1])-c(i,\sigma[x,0])}{2}\left[f(\sigma[x,1]^i) - f(\sigma[x,1]) \right] \\
& \qquad \qquad + \frac{c(i,\sigma[x,1])-c(i,\sigma[x,0])}{2}\left[f(\sigma[x,0]^i) - f(\sigma[x,0]) \right].
\end{align*}
Note that if $i = x$, then the second and third term on the right hand side cancel each other. If $i \neq x$, we can express the two differences in $f$ in terms of $\Phi f$ and $s(i,\sigma)$:
\begin{align*}
& c(i, \sigma[x,1]) \left[f(\sigma[x,1]^i) - f(\sigma[x,1])\right] - c(i, \sigma[x,0]) \left[f(\sigma[x,0]^i) - f(\sigma[x,0])\right] \\
& \qquad = \frac{c(i,\sigma[x,1]) + c(i,\sigma[x,0])}{2}\left[\Phi f(\sigma^i,x) - \Phi f(\sigma,x)\right] \\
& \qquad \qquad + \delta_{\{i \neq x\}}\frac{c(i,\sigma[x,1])-c(i,\sigma[x,0])}{2}s(i,\sigma)\left[\Phi f(\sigma[x,1],i) + \Phi f(\sigma[x,0],i)\right] .
\end{align*}

We conclude that $\Phi Af = (B+C)\Phi f$. 

\smallskip

To apply Theorem \ref{theorem:main_theorem}, we are left to check conditions (b) and (c). For (c) note that because $S(t)$ is continuous for the $\tn{\cdot}$ topology by Theorem \ref{theorem:LiggettI.3.9}, we find that $R(\lambda,A)D \subseteq D$ for all $\lambda > 0$ via the integral representation formula for the resolvent.

For (b), we proceed by checking the conditions for Proposition \ref{proposition:core_approximation_operators}. Define the operators
\begin{equation*}
T_n f(\sigma) := f(\sigma[n])
\end{equation*}
where $\sigma[n]$ is the configuration defined by $\sigma[n]_{\phi(i)} = \sigma_{\phi(i)}$ if $i \leq n$ and $\sigma[n]_{\phi(i)} = 0$ otherwise. Here we use the map $\phi$ introduced in the proof of Proposition \ref{proposition:IPS_monotone_functions}. Note that by construction $T_n \colon \cD(\Phi) \rightarrow D$. Furthermore, note that $f \in \cF_{+,0}$ implies that $\Phi(T_n f) \in \cF_{+,0}$. The final property $\lim_n \vn{T_n f  -f} = 0$ follows from the uniform continuity of $f$. Thus the operators $\{T_n\}_{n \geq 1}$ are approximation operators for $(\Phi,\cF_{+,0},D)$. We conclude that $\cF_{+,0} \cap D$ is dense in $\cF_{+,0}$ by Proposition \ref{proposition:core_approximation_operators}.

\smallskip

The final claim of the theorem is established by an application of Theorem \ref{theorem:lower_bound}. We take $\cD^* = D^{\otimes S}$ and note that in this setting $F_+$ is the cone of non-negative functions in $C_0(\cX \times S)$. As we have control on the growth of the norm $\tn{T(t)h}_\infty$ for functions $h \in D^{\otimes S}$ by Lemma \ref{lemma:IPS_trace_generator}, it follows as above that also the resolvent $R(\lambda, B)$ maps $D^{\otimes S}$ into itself.

To conclude, we verify condition (c) for Theorem \ref{theorem:lower_bound} via Corollary \ref{corollary:core_approximation_operators}. 
	
	Define $T_n'\colon C_0(\cX \times S) \rightarrow C_0(\cX \times S)$ by  
	\begin{equation*}
	T_n' h(\sigma,x) := \begin{cases}
	h(\sigma[n],x) & \text{if } x = \phi(i) \text{ for some } i \leq n, \\
	0 & \text{if } x = \phi(i) \text{ for some } i > n,
	\end{cases}
	\end{equation*}
	It is straightforward to verify that $\{T_n'\}$ are approximation operators for $(\bONE,F_+,D^{\otimes S})$. Finally, we verify that if $\Phi f \geq h \geq 0$, then $\Phi T_n f \geq T_n' h$.

	\textit{Case 1:} If $i > n$, then $\Phi(T_n f)(\sigma,\phi(i)) = 0 = T_n'h(\sigma,\phi(i))$. 
	\textit{Case 2:} If $x \in S$ is such that $x = \phi(i)$ for $i \leq n$, then $\Phi(T_n f)(\sigma,x) = f(\sigma[x,1][n]) - f(\sigma[x,0][n])$, where $\sigma[x,1][n]$ means that we first replace $\sigma$ by $\sigma[x,1]$ and then $\sigma[x,1]$ by replacing all coordinates with index larger than $n$ by $0$. As $x = \phi(i)$ for $i \leq n$, this result is the same as when we interchange these two replacements: that means $\sigma[x,1][n] = \sigma[n][x,1]$. This implies
	\begin{multline*}
	\Phi(T_n f)(\sigma,x) = f(\sigma[x,1][n]) - f(\sigma[x,0][n]) = f(\sigma[n][x,1]) - f(\sigma[n][x,0]) \\
	= \Phi f(\sigma[n],x) \geq h(\sigma[n],x) = T_n' h(\sigma,x).
	\end{multline*}
	Thus, Condition (c) of Theorem \ref{theorem:lower_bound} follows from Corollary \ref{corollary:core_approximation_operators} by using the set of approximation operators $\{(T_n,T_n')\}$.
	
	We conclude that for $f \in \cF_{+,0}$ and $h \in C_0(\cX \times S)$ with $\Phi f \geq h \geq 0$, we have $\Phi S(t)f \geq T(t) h$.
\end{proof}

\section{Proofs of main results} \label{section:proofs_main_results}

We start by stating a generalisation of Theorem \ref{theorem:main_theorem}.

\begin{theorem} \label{theorem:main_theorem2}
 Let Assumption \ref{assumption:main_assumption} be satisfied. In addition to $A$ and $B$, suppose there is a third  linear operator $C \colon \cD(C) \subseteq F \rightarrow F$. Let $\cD$ be a set such that $\cD \subseteq \cD(B \Phi) \cap \cD(C \Phi) \cap \cD(\Phi A)$ and suppose that for all $f \in \cD$ we have
\begin{equation} \label{eqn:commutatation_relation_main_thm2}
\Phi A f = (B+C) \Phi f.
\end{equation}
Fix some $\lambda_0 > 0$ and suppose that for all $\lambda \in (0,\lambda_0)$ 
\begin{enumerate}[(a)]
\item the set $\cF_{\lambda} :=  ((\bONE - \lambda A)\cD) \cap \cF_{+,0}$ is dense in $\cF_{+,0}$ with respect to $\tau_E$,
\item  the operator $\bONE - \lambda C R(\lambda,B)$ is invertible with full domain and has inverse which we denote by $U_\lambda$ which satisfies $U_\lambda \geq 0$. 
\end{enumerate}
Then, for any $\lambda$ such that $0 < \lambda < \lambda_0$ and for all $t \geq 0$, we have: if $g \in \cF_+$, then $R(\lambda,A)g \in \cF_+$ and $S(t)g \in \cF_+$.
\end{theorem}

The main tool for the proof of this result implies that if $\Phi A - B \Phi$ is `positive', then also $\Phi R(\lambda,A) - R(\lambda,B) \Phi$ is `positive'. The result has two versions, one that we will need that uses $U_\lambda \geq 0$ and a result that assumes more, but has a stronger consequence. The conditions for the latter result hold for example in the typical case in which $C$ is continuous and non-negative, cf.~Lemma \ref{lemma:existence_Ulambda}.

	\begin{lemma} \label{lemma:fundamental_bound}
		Let the assumptions of Theorem \ref{theorem:main_theorem2} be satisfied. Then the following two statements hold.
		\begin{enumerate}[(i)]
			\item if  $U_\lambda \geq 0$, then it holds for all $g \in \cF_\lambda$ that $\Phi R(\lambda,A) g \geq 0$,
			\item if  $U_\lambda \geq \bONE$, then it holds for all $g \in \cF_\lambda$ that $\Phi R(\lambda,A) g \geq R(\lambda,B) \Phi g \geq 0$.
		\end{enumerate}
	\end{lemma}
	
	The proof of the lemma, and also that of Lemma \ref{lemma:existence_Ulambda} below, are inspired on the manipulations that were made in the proof of Theorem 3.1 in \cite{Ar87}.
	
	\begin{proof}
		Fix $\lambda \in (0,\lambda_0)$,  $g \in \cF_\lambda$ and set $f = R(\lambda,A) g$.  Because $\cF_\lambda \subseteq (\bONE - \lambda A) \cD$, we find $f \in \cD$.  Using the linearity of $\Phi$  and \eqref{eqn:commutatation_relation_main_thm2},  we obtain
		\[
		\Phi g = \Phi f - \Phi \lambda A f =  \Phi f - \lambda (B+C) \Phi f
		= (\bONE - \lambda (B+C) ) \Phi f.
		\]
		 As $f \in \cD \subseteq \cD(B \Phi)$, we can rewrite the right-hand side as
		\begin{equation*}
		\Phi g = \left(\bONE  - \lambda C R(\lambda,B) \right)  (\bONE  - \lambda B) \Phi f.
		\end{equation*}	
		Applying first $U_\lambda$ and then $R(\lambda,B)$ on both sides of this equation, using that $f = R(\lambda,A)g$, we find
		\begin{equation*}
		R(\lambda,B)  U_\lambda  \Phi g =\Phi R(\lambda,A) g. 
		\end{equation*}
		Next, we use that $g \in \cF_{\lambda} \subseteq \cF_{+,0}$. (i) follows by using that 
	$R(\lambda,B) \geq 0$ and $U_\lambda \geq 0$, whereas (ii) follows because $R(\lambda,B) (U_\lambda - \bONE)\Phi g \ge 0$ which implies
	\[
	\Phi R(\lambda,A) g \geq  R(\lambda,B) \Phi g.
	\]

	\end{proof}

\begin{proof}[Proof of Theorem \ref{theorem:main_theorem2}]
Fix  $\lambda \in (0,\lambda_0)$. We start by proving that $R(\lambda,A) \cF_\lambda \subseteq \cF_{+,0}$. Thus, choose $g \in \cF_\lambda$. By definition of $\cF_\lambda$, we have $R(\lambda,A)g \in \cD \subseteq \cD(\Phi)$. By Assumption (b) and Lemma \ref{lemma:fundamental_bound}, we find $R(\lambda,A)g \in \cF_{+,0}$.

	
\smallskip
	
Now let $g \in \cF_+$. As $\cF_+$ is the $\tau_E$ closure of $\cF_\lambda$ by Assumption (a), we can find a net $g_\alpha \in \cF_\lambda$ such that $g_\alpha \rightarrow g$. As $R(\lambda,A)$ is a continuous map, and $R(\lambda,A)g_\alpha \in \cF_{+,0}$ for every $\alpha$ by the argument above, we conclude that $R(\lambda,A)g \in \cF_+$.
	
\smallskip
	
We proceed with the second statement: $S(t) \cF_+ \subseteq \cF_+$ for all $t \geq 0$. Fix $t \geq 0$ and $g \in \cF_+$. By Yosida-approximation, cf.~\eqref{eqn:Yosida_approximation}, we have
\begin{equation*}
S(t)g := \lim_n \e^{t A_n} g = \lim_n \sum_{k \geq 0} \frac{(tnR(n^{-1},A))^k}{k!}\e^{-tn},
\end{equation*}
for $\tau_E$. As $\cF_+$ is $\tau_E$ closed and convex, it suffices to prove for sufficiently large $n$ that $R\left(\frac{t}{n},A\right)^k g \in \cF_+$ for all $k \geq 1$. Choosing $n$ large enough such that $t/n < \lambda_0$, this follows from by iterating the result for the resolvent $k$ times.
\end{proof}

The next two lemmas establish that Theorem \ref{theorem:main_theorem} is a consequence of Theorem \ref{theorem:main_theorem2}.

\begin{lemma} \label{lemma:existence_Ulambda}
Suppose we are in the setting of Theorem \ref{theorem:main_theorem2}. Suppose that $C$ is continuous and positive. Then there is a $\lambda_0 > 0$ such that for $\lambda \in (0,\lambda_0)$ the operator
	\begin{equation*}
	\bONE - \lambda C R(\lambda,B) 
	\end{equation*}
	is invertible with inverse $U_\lambda$ that satisfies $U_\lambda \geq \bONE$. 

\end{lemma}

\begin{proof}
Our goal is to define the operator $U_\lambda$ as the sum
\begin{equation} \label{eqn:def_U_lambda}
U_\lambda f := \sum_{k=0}^\infty \left(\lambda C R(\lambda,B)\right)^k f.
\end{equation}
We have to show this sum is well-defined, and additionally, that the operator that we obtain in this way  dominates $\bONE$.

By assumption $R(\lambda,B)$ and $C$ are $\tau_F$ continuous. This implies that they are $\vn{\cdot}_F$ bounded. In particular, if we put $\lambda_0 = \vn{C}^{-1}_F\vn{R(\lambda,B)}_{F}^{-1}$. Then for $\lambda < \lambda_0$, the sum in \eqref{eqn:def_U_lambda} converges in the Banach space $(F,\vn{\cdot}_F)$. Elementary arguments yield that indeed $U_\lambda = \left(\bONE - \lambda C R(\lambda,B)  \right)^{-1}$.

Next, we prove positivity of $U_\lambda - \bONE \geq 0$. Note that
\begin{equation*}
U_\lambda f - f = \lim_n \sum_{k=1}^n \left(\lambda C R(\lambda,B)\right)^k f
\end{equation*}
for $\vn{\cdot}_F$ and hence, also for $\tau_F$. Thus, if $f \in \cF_+$, then the partial sums on the right-hand side are positive, due to the positivity of $C$ and $R(\lambda,B)$. Because $(F,\tau,\leq)$ is locally solid, Lemma \ref{lemma:positive_cone_is_closed} in Appendix \ref{appendix:locally_solid_space} implies that the limit is also in the positive cone. We conclude that $U_\lambda \geq \bONE$.
\end{proof}

\begin{remark}
We note that most of the argument in the proof above uses the norm $\vn{\cdot}_F$ topology only. It is only in the final step that we use that $\tau_F$ is locally convex-solid, to establish that $U_\lambda \geq \bONE$. If one were be able to prove that for any locally  convex-solid topology $\tau_F$, also the strong topology $\vn{\cdot}_F$ is locally   convex-solid, then all arguments for the space $F$ can be carried out with the norm topology.
\end{remark}

The next lemma will be our main tool to establish the density of $\cF_\lambda$ in $\cF_{+,0}$.

\begin{lemma} \label{lemma:core}
Consider the setting of Theorem \ref{theorem:main_theorem2}  and fix some $\lambda > 0$. If
\begin{enumerate}[(a)]
\item $\cD \cap \cF_{+,0}$ is $\tau_E$ dense in $\cF_{+,0}$,
\item we have $R(\lambda,A)\cD \subseteq \cD$,
\end{enumerate}
then $\cF_\lambda$ is $\tau_E$ dense in $\cF_{+,0}$.
\end{lemma}

\begin{proof}
	Fix $\lambda > 0$. Recall that $\cF_\lambda := (\bONE - \lambda A) \cD \cap \cF_{+,0}$. It follows by (a) that it suffices to prove that $\cD \subseteq (\bONE - \lambda A) \cD$. This statement follows by applying $(\bONE - \lambda A)$ to both sides of the inclusion of (b). Note that this is possible due to the inclusion $\cD \subseteq \cD(A)$. 	
\end{proof}

Theorem \ref{theorem:lower_bound} can be proven as a corollary of Theorem \ref{theorem:main_theorem} by considering product spaces and product semigroups instead of the original objects.

\begin{proof}[Proof of Theorem \ref{theorem:lower_bound}]
The result follows from an application of Theorem \ref{theorem:main_theorem} for the semigroup $\cS(t) = (S(t), T(t))$, the operator $\Upsilon$ and $\cD \times \cD^*$ playing the role of $\cD$. We check (a), (b) and (c) of Theorem \ref{theorem:main_theorem}. (a) was verified in the discussion preceding the theorem. (b) is implied by assumption (b) of this theorem and (c) follows by an application of (c) of this theorem and the product structure of the the resolvent.

We conclude that $(S(t),T(t)) \cG_+ = \cS(t) \cG_+ \subseteq \cG_+$. The final claim is immediate.
\end{proof}

\begin{proof}[Proof of Theorem \ref{theorem:comparison}]
	The result follows from an application of Theorem \ref{theorem:main_theorem} with $\cH_{+,0}^{(i)}$ corresponding to $\cF_{+,0}$ and	$\cD \times \cD^*$  corresponding to $\cD$. By the latter choice condition (c) is immediate consequence of $R(\lambda, A) \cD \subseteq \cD$. 
\end{proof}

\begin{proof}[Proof of Corollary \ref{corollary:preservation_of_order}]
Fix $t \geq 0$ and suppose that $\eta^{(1)}(0) \preccurlyeq \eta^{(2)}(0)$. We prove that $\bE[f(\eta^{(1)}(t))] \leq \bE[f(\eta^{(2)}(t))]$ for all $f \in \cF_+$. 

Denote by $\nu$ be the law of $\eta^{(1)}(0)$ and by $\mu$ the law of $\eta^{(2)}(0)$. By Theorem \ref{theorem:comparison}, we have $S^{(1)}(t)f \leq S(t)f \leq S^{(2)}(t)f$ and $S(t)f \in \cF_+$. This implies
\begin{multline*}
\bE[f(\eta^{(1)}(t))] = \ip{S^{(1)}(t)f}{\nu} \leq \ip{S(t)f}{\nu} \\
\leq \ip{S(t)f}{\mu} \leq \ip{S^{(2)}(t)f}{\mu} = \bE[f(\eta^{(2)}(t))].
\end{multline*}
Note that the third inequality is a consequence of $\eta^{(1)}(0) \preccurlyeq \eta^{(2)}(0)$. We conclude that $\eta^{(1)}(t)  \preccurlyeq \eta^{(2)}(t)$. 
\end{proof}

We proceed with the proofs of Propositions \ref{proposition:core_approximation_operators} and \ref{proposition:core_double_approximation_operators}.

\begin{proof}[Proof of Proposition \ref{proposition:core_approximation_operators}]
Pick $f \in \cF_{+,0}$. Note that $\cF_{+,0} \subseteq \cD(\Phi)$. This implies by properties (a) and (b) of approximation operators that $T_n f \in \cF_{+,0} \cap \cD$. By property (c), we find $T_n f \rightarrow f$ for the strict topology. We conclude that $\cD \cap \cF_{+,0}$ is strictly dense in $\cF_{+,0}$.
\end{proof}

%

\begin{proof}[Proof of Corollary \ref{corollary:core_approximation_operators}]
	By construction, $\tilde T_n$ is a strongly continuous family of strictly continuous operators.
	Let $(f,h) \in \cG_{+,0}$. Then $\Phi f \ge h$ and therefore $\Upsilon \tilde T_n (f,h) = ( \Phi T_n f - T_n' h, T_n' h) \ge 0$ by \eqref{Phi_positive} .
	Therefore $\tilde T_n \cG_{+,0} \subseteq \cG_{+,0}$ and also property (a) of the definition of approximation operators holds. Finally, by property (b) of $T_n$ and $T'_n$, also property (b) for $\tilde T_n$ holds.
\end{proof}

\begin{proof}[Proof of Proposition \ref{proposition:core_double_approximation_operators}]
We prove the result for $i = 2$, the other case being similar. Let $(f,h) \in \cH_{+,0}^{(2)}$ 
and consider the sequence $\{T_n f, T_n h\}$. As $T_n$ is positive and $f \leq h$, we find $T_n f \leq T_n h$ for all $n$. As in the proof of Proposition \ref{proposition:core_approximation_operators}, we find $T_n f, T_n h \in \cD$ and hence $(T_n f, T_n h) \in \cH_{+,0}^{(2)} \cap \cD^2$. 

From the convergence of the sequences
$T_n f \rightarrow f$ and $T_n h \rightarrow h$, we conclude
$(T_n f, T_n h) \rightarrow (f,h)$. Therefore  $\cD^2 \cap \cH_{+,0}^{(2)}$ is dense in $\cH_{+,0}^{(2)}$.
\end{proof}

\begin{remark}
Regarding Remark \ref{remark:restriction_of_Upsilon}, note that we obtain $T_n h \in \cD$ by our assumption that $h \in \cD(\Phi)$.
\end{remark}

\appendix

\section{Locally convex-solid spaces} \label{appendix:locally_solid_space}

The following definitions are taken from \cite{AB78}. Consider an ordered vector space $(F,\leq)$. We say that a set $A \subseteq F$ has a supremum $g$ if $g$ has the properties
\begin{enumerate}[(i)]
\item that $f \leq g$ for all $f \in A$,
\item if $f \leq h$ for all $f \in A$ and some other $h \in E$, then $g \leq h$.
\end{enumerate}
Similarly we define, if it exists, infimum of a set. We say that $(F,\leq)$ is a \textit{Riesz space} if the infimum and supremum of every finite set exists. We denote $f \vee g$ for the supremum of the set $\{f,g\}$ and we write $f \wedge g$ for the infimum of the set $\{f,g\}$. Furthermore, we write $f^+ = f \vee 0, f^- = f \wedge 0$ and $|f| = f^+ - f^-$. We write $F_+ := \left\{ f \in F \, \middle| \, f \geq 0 \right\}$ for the positive cone of $F$.

We say that a subset $A \subseteq F$ is \textit{solid} if $|f| \leq |g|$ and $g \in A$ implies that $f \in A$.

\begin{definition}
	We say that $(F,\tau_F,\leq)$ is a \emph{locally convex-solid space} if $(F,\tau_F)$ is a topological vector space which has a basis of closed, convex and solid(for $\leq$) neighbourhoods of $0$. 
\end{definition}

\begin{lemma}[Theorem 5.6 in \cite{AB78}] \label{lemma:positive_cone_is_closed}
	The positive cone $F_+$ in a locally convex-solid space $(F,\tau_F,\leq)$ is $\tau_F$ closed. 
\end{lemma}

\section{Semigroups on locally convex spaces} \label{section:semigroups_on_lcs}

The definitions in this appendix follow \cite{Kr16}.

Let $(E,\tau_E)$ be a locally convex space. We call the family of operators $\{T(t)\}_{t \geq 0}$ a \textit{semigroup} if $T(0) = \bONE$ and $T(t)T(s) = T(t+s)$ for $s,t \geq 0$. A family of $(E,\tau_E)$ continuous operators $\{T(t)\}_{t \geq 0}$ is called a \textit{strongly continuous semigroup} if $t \mapsto T(t)f$ is continuous and \textit{weakly continuous} if $t \mapsto \ip{T(t)f}{\mu}$ is continuous for every $f \in E$ and $\mu \in E'$. 
\begin{sloppypar} 
We call $\{T(t)\}_{t \geq 0}$ a \textit{locally equi-continuous} family if for every $t \geq 0$ and continuous semi-norm $p$, there exists a continuous semi-norm $q$ such that ${\sup_{s \leq t} p(T(s)f) \leq q(x)}$ for every $f \in E$. Finally, we abbreviate strongly continuous and locally equi-continuous semigroup to \textit{SCLE} semigroup.
\end{sloppypar}

We say that $A \subseteq E\times E$ is the \textit{generator} of a SCLE semigroup $\{T(t)\}_{t \geq 0}$ if 
\begin{equation*}
g = \lim_{t \downarrow 0} \frac{T(t)f -f}{t} \quad \Leftrightarrow \quad (f,g) \in \cD(A).
\end{equation*}
Throughout the paper, we make heavy use of the resolvent of $A$. First define the resolvent set $\rho(A) \subseteq \bC$
\begin{equation*}
\rho(A) := \{\lambda \in \bC \, | \, (\bONE - \lambda A) \text{ is invertible and the inverse is $\tau_E$ continuous}\}.
\end{equation*}
For $\lambda \in \rho(A)$, we denote $R(\lambda,A) = (\bONE - \lambda A)^{-1}$.

\begin{definition}
Let $(E,\tau_E)$ be locally convex and let $T \subseteq E \times E$.
\begin{enumerate}[(i)]
\item We write $\overline{T}$ for the $\tau_E$ \textit{closure} of $T$ in $E \times E$. We say that $T$ is \textit{closed} if $T = \overline{T}$.
\item We say that $\cT \subseteq E \times E$ is an \textit{extension} of $T$ if $T \subseteq \cT$.
\end{enumerate}
\end{definition}

	The study of strongly continuous semigroups on Banach spaces is greatly aided by the fact that the norm allows for uniform control in various proofs. On locally convex spaces such control is not to be expected. The uniform control, however, can be replaced by `probabilistic control': that is, an equivalent of tightness and uniform control on compact sets. This is possible for locally convex spaces satisfying Condition C, see Definition \ref{def:condB_condC_intro}.

\begin{remark} \label{remark:strict_satisfies_C}
	In Section \ref{section:strict_topology} below, we will consider the strict topology $\beta$ on the space $C_b(\cX)$. This topology, together with the supremum norm, satisfies Condition C, see Corollary 8.2 in \cite{Kr16}.
\end{remark}

For locally convex spaces satisfying Condition C, we have the following Hille-Yosida theorem. The approximation result can be found in the proof of the mentioned result in \cite{Kr16}.

\begin{theorem}[Theorem 6.4 in \cite{Kr16}] \label{theorem:Hille_Yosida}
Let $(E,\tau_E,\vn{\cdot}_E)$ satisfy Condition C. For a linear operator $(A,\cD(A))$ on $(E,\tau_E)$, the following are equivalent.
\begin{enumerate}[(a)]
\item $(A,\cD(A))$ generates a $\tau_E$-SCLE semigroup $\{T(t)\}_{t \geq 0}$ that is $\vn{\cdot}_E$, or equivalently $\tau_E$, bounded.
\item $(A,\cD(A))$ is $\tau_E$-closed, $\tau_E$-densely defined and for every $\lambda > 0$ one has $\lambda \in \rho(A)$ and for every semi-norm $p \in \cN$ and $\lambda_0 > 0$ there exists a semi-norm $q \in \cN$ such that for all $f \in E$ one has 
\begin{equation*}\label{eqn:Hille_Yosida_estimate_not_complex}
\sup_{n \geq 1} \sup_{0 < \lambda \leq \lambda_0} p\left( \left( R(n \lambda)\right)^n f \right) \leq  q(f).
\end{equation*}
\end{enumerate}
In this setting, we additionally have Yosida-approximation. Set $A_n = n \left( R(n^{-1},A) - \bONE \right)$, then we have for every $f \in E$ that
\begin{equation} \label{eqn:Yosida_approximation}
\e^{t A_n} f = \sum_{k \geq 0} \frac{\left(tn R(n^{-1},A)\right)^k}{k!}\e^{-tn} f \rightarrow T(t)f
\end{equation}
$\tau_E$-uniformly for $t$ in compact intervals.
\end{theorem}

\section{Strongly continuous semigroups for the strict topology and the martingale problem}\label{section:martingale_problem}

As mentioned in the introduction, the transition semigroup of a Feller process on a non-compact Polish space is usually not strongly continuous for the supremum norm. On the other hand, functional analytic theory can greatly aid the study of such semigroups. In this appendix, we collect some results that unify probabilistic and functional analytic methods that have been developed to deal with this disconnect, and that can be used for the purposes of this paper.

\subsection{The martingale problem}

The main technique for characterizing Markov processes on Polish spaces is the so called martingale problem. 

\begin{definition}[The martingale problem]
Let $A \colon \cD(A) \subseteq C_b(\cX) \rightarrow C_b(\cX)$ be a linear operator. For $(A,\cD(A))$ and a measure $\nu \in \cP(\cX)$, we say that $\PR \in \cP(D_\cX(\bR^+))$ solves the \textit{martingale problem} for $(A,\nu)$ if $\PR \eta(0)^{-1} = \nu$ and if for all $f \in \cD(A)$ 
\begin{equation*}
M_f(t) := f(\eta(t)) - f(\eta(0)) - \int_0^t Af(\eta(s)) \dd s
\end{equation*}
is a mean $0$ martingale with respect to its natural filtration $\cF_t := \sigma\left(\eta(s) \, | \, s \leq t \right)$ under $\PR$. 
	
\smallskip
	
We denote the set of all solutions to the martingale problem, for varying initial measures $\nu$, by $\cM_A$. We say that \textit{uniqueness} holds for the martingale problem if for every $\nu \in \cP(\cX)$ the set $\cM_\nu:= \{\PR \in \cM_A \, | \, \PR \eta(0)^{-1} = \nu\}$ is empty or a singleton. Furthermore, we say that the martingale problem is \textit{well-posed} if this set contains exactly one element for every $\nu$.
\end{definition}

\smallskip

The construction of solutions to the martingale problem can often be done via approximating processes. Classically, uniqueness for the martingale problem is proven via duality. \cite{CoKu15}, however, introduced a method based on viscosity solutions.

\subsection{The strict topology} \label{section:strict_topology}

The effectiveness of functional analytic theory for the study of probability on (locally) compact metric spaces comes from the Riesz-representation theorem. This result identifies the space of Borel measures as the dual space of $(C_0(\cX),\vn{\cdot})$. If we replace $\cX$ by a non-compact Polish space, this duality breaks down.

To reconnect the functional analytic theory to probability theory, in the form of re-establishing this duality a weaker topology on $C_b(\cX)$ is necessary. This weaker topology is the \textit{strict} topology. This topology was first introduced by Buck \cite{Bu58} in the locally compact setting. A comprehensive treatment in the setting of locally convex spaces is given in \cite{Se72}. Finally, \cite{Ca11} treats the strict topology in the context of Markov processes

\smallskip

Before we introduce the strict topology, we introduce the \textit{compact-open} topology $\kappa$ on $C_b(\cX)$  and $M_b(\cX)$. This locally convex topology is generated by the semi-norms $p_K(f) := \sup_{x \in K} |f(x)|$, where $K$ ranges over all compact sets in $\cX$.

The \textit{strict} topology $\beta$ on the space bounded continuous functions $C_b(\cX)$  and $M_b(\cX)$ is generated by the semi-norms 
\[
p_{K_n,a_n}(f) := \sup_n a_n \sup_{x \in K_n} |f(x)|
\]
varying over non-negative sequences $a_n$ converging to $0$ and sequences of compact sets $K_n \subseteq \cX$. We proceed discussing properties of the strict topology on $C_b(\cX)$.

\begin{remark}
We refer the reader to the discussion of the strict and substrict topology in \cite{Se72}, where it is shown that these two topologies coincide for Polish spaces. Because the definition of the substrict topology is more accessible, we use this as a characterisation of the strict topology.
\end{remark}

\begin{remark}
The strict topology can equivalently be given by the collection of semi-norms
\begin{equation*} 
p_g(f) := \vn{fg}
\end{equation*}
where $g$ ranges over the set 
\begin{equation*}
\left\{g \in C_b(\cX) \, | \, \forall \alpha > 0: \, \left\{x , | \,  |g(x)| \geq \alpha\right\} \text{ is compact in } \cX \right\}.
\end{equation*}
See \cite{Bu58} and \cite{Ca11}.
\end{remark}

The strict topology is the `right' generalisation of the norm topology on $C(\cX)$ for compact metric $\cX$ to the more general context of Polish spaces. To avoid further scattering of results, we collect some of the main properties of $\beta$.

\begin{theorem} \label{theorem:propertiesCbstrict} 
Let $\cX$ be Polish. The locally convex space $(C_b(\cX),\beta)$ satisfies the following properties.
\begin{enumerate}[(a)]
\item $(C_b(\cX),\beta)$ is complete, strong Mackey (i.e.~all weakly compact sets in the dual are equi-continuous) and the continuous dual space coincides with the space of Radon measures on $\cX$ of bounded total variation.
\item $(C_b(\cX),\beta)$ is separable.
\item For any locally convex space $(F,\tau_F)$ and $\beta$ to $\tau_F$ sequentially equi-continuous family $\{T_i\}_{i \in I}$ of maps $T_i \colon (C_b(\cX),\beta) \rightarrow (F,\tau_F)$,  the family $I$ is $\beta$ to $\tau_F$ equi-continuous.
\item The norm bounded and $\beta$ bounded sets coincide. Furthermore, on norm bounded sets $\beta$ and $\kappa$ coincide.
\item Stone-Weierstrass: Let $M$ be an algebra of functions in $C_b(\cX)$. If $M$ vanishes nowhere and separates points, then $M$ is $\beta$ dense in $C_b(\cX)$.
\item Arzel\`{a}-Ascoli: A set $M \subseteq C_b(\cX)$ is $\beta$ compact if and only if $M$ is norm bounded and $M$ is an equi-continuous family of functions.
\item $(C_b(\cX),\beta,\leq)$, where $\leq $ is defined as $f \leq g$ if and only if $f(x) \leq g(x)$ for all $x \in \cX$, is locally convex-solid.
\item Dini's theorem: If $\{f_\alpha\}_{\alpha}$ is a net in $C_b(\cX)$ such that $f_\alpha$ increases or decreases point-wise to $f \in C_b(\cX)$, then $f_\alpha \rightarrow f$ for the strict topology.
\end{enumerate}
\end{theorem}

Note that $(d)$ implies that a sequence $f_n \stackrel{\beta}{\rightarrow} f$ if and only if $\sup_n \vn{f_n} < \infty$ and $f_n \stackrel{\kappa}{\rightarrow} f$. 

\begin{proof}
(a) and (c) follow from Theorems 9.1 and 8.1 in \cite{Se72}, Theorem 7.4 in \cite{Wi81}, Corollary 3.6 in \cite{We68} and Krein's theorem\cite[24.5.(4)]{Ko69}. (b) follows from Theorem 2.1 in \cite{Su72}. (d) follows by Theorems 4.7, 2.4 in \cite{Se72} and 2.2.1 in \cite{Wi61}. (e) is proven in Theorem 2.1 and Corollary 2.4 in \cite{Ha76}. (f) follows by the Arzel\`{a}-Ascoli theorem for the compact-open topology, Theorem 8.2.10 in \cite{En89}, and (d). To conclude, (g) and (h) follow from Theorems 6.1 and 6.2 in \cite{Se72}.
\end{proof}

The following result is an analogue of the Banach-Steinhaus theorem for the setting of strong Mackey spaces and is usefull for the study of operators for the strict topology. The proof follows from Lemma 3.8 in \cite{Ku09} and the fact that the strict topology is strong Mackey. See also Lemma 3.2 in \cite{Kr16} for the use of a variant of this result in the setting of semigroups. 

\begin{proposition} \label{proposition:equi_continuity_via_strong_continuity}
	Suppose $\{T_n\}_{n \geq 1}$ is a family of strictly continuous linear operators from $C_b(\cX)$ to $C_b(\cX)$. Additionally assume that $T_n f \rightarrow f$ for all $f \in C_b(\cX)$. Then the family $\{T_n\}$ is equi-continuous: for every $\beta$ continuous semi-norm $p$ there is a $\beta$ continuous semi-norm $q$ such that
	\begin{equation*}
	\sup_{n \geq 1} p\left(T_n f \right) \leq q(f)
	\end{equation*}
	for all $f \in C_b(\cX)$.
\end{proposition}

\subsection[The transition semigroup is SCLE]{The transition semigroup is strongly continuous and locally equi-continuous with respect to the strict topology} \label{section:martingale_problem_SCLE}

In the setting that the martingale problem is well-posed, we obtain a strengthened version of Theorem 4.5.11, \cite{EK86}, showing that the transition semigroup of the solution is strongly continuous for the strict topology. For a overview of results on SCLE semigroups relevant for the results to follow, see Section \ref{section:semigroups_on_lcs}.

	\begin{theorem} \label{theorem:verification_conditions_semigroup}
		Let $A \subseteq C_b(\cX) \times C_b(\cX)$ and let the martingale problem for $A$ be well-posed. Suppose that the closed convex hull of $\cD(A)$ is $\beta$ dense in $C_b(\cX)$. Suppose that for all compact $K \subseteq \cP(\cX)$, $\varepsilon >0$ and $T > 0$, there exists a compact set $K' = K'(K,\varepsilon,T)$ such that for all $\PR \in \cM_A$, we have
		\begin{equation} \label{eqn:theorem_compact_containment_set}
		\PR\left[\eta(t) \in K' \text{ for all } t < T, \eta(0) \in K \right] \geq (1-\varepsilon) \PR\left[\eta(0) \in K \right].
		\end{equation}
		Then the measures $\PR \in \cM_A$ correspond to strong Markov processes with a $\beta$-SCLE semigroup $\{S(t)\}_{t \geq 0}$ on $C_b(\cX)$ defined by $S(t)f(x) = \bE[f(\eta(t))\, | \, \eta(0) = x]$. The generator of $\{S(t)\}_{t \geq 0}$ is an extension of $A$.
	\end{theorem}

\begin{proof}[Proof of Theorem \ref{theorem:verification_conditions_semigroup}]
	
The proof that the solutions are strong Markov and correspond to a semigroup 
\begin{equation*}
S(t)f(x) = \bE[f(\eta(t)) \, | \, \eta(0) = x]
\end{equation*}
that maps $C_b(\cX)$ into $C_b(\cX)$ follows as in the proof of (b) and (c) of Theorem 4.5.11 \cite{EK86}. We are left to show that $\{S(t)\}_{t \geq 0}$ is SCLE for $\beta$, which we do in Lemma \ref{lemma:semigroup_is_for_every_time_continuous} and Proposition \ref{proposition:semigroup_strongly_continuous} below. That the generator of the semigroup extends follows from Proposition \ref{proposition:generator_extends_martingale_problem_operator}.
\end{proof}

\begin{lemma} \label{lemma:semigroup_is_for_every_time_continuous}
Let $\{S(t)\}_{t \geq 0}$ be the semigroup introduced in Theorem \ref{theorem:verification_conditions_semigroup}. The family $\{S(t)\}_{t \geq 0}$ is locally equi-continuous for $\beta$.
\end{lemma}

\begin{proof}
Fix $T \geq 0$. We will prove that $\{S(t)\}_{t \leq T}$ is $\beta$ equi-continuous by using Theorem \ref{theorem:propertiesCbstrict} (c) and (d). Pick a sequence $f_n$ converging to $f$ with respect to $\beta$. It follows that $\sup_n \vn{f_n} \leq \infty$, which directly implies that $\sup_n \sup_{t \leq T} \vn{S(t)f_n} < \infty$.

We also know that $f_n \rightarrow f$ uniformly on compact sets. We prove that this implies the same for $S(t)f_n$ and $S(t)f$ uniformly in $t \leq T$. Fix $\varepsilon > 0$ and a compact set $K \subseteq \cX$, and let $\hat{K}$ be the set introduced in Equation \eqref{eqn:theorem_compact_containment_set} for $T$. Then we obtain that
\begin{align*}
& \sup_{t \leq T} \sup_{x \in K} \left|S(t)f(x) - S(t)f_n(x) \right| \\
& \leq \sup_{t \leq T}\sup_{x \in K}  \bE_x\left|f(\eta(t)) - f_n(\eta(t)) \right| \\
& \leq  \sup_{t \leq T} \sup_{x \in K}  \bE_x\left|\left(f(\eta(t)) - f_n(\eta(t)) \right) \bONE_{\{\eta(t) \in \hat{K}\}} \right. \\
& \qquad \qquad \qquad \qquad \qquad \left. + \left(f(\eta(t)) - f_n(\eta(t)) \right) \bONE_{\{\eta(t) \in \hat{K}^c\}} \right| \\
& \leq \sup_{t \leq T} \sup_{y \in \hat{K}} \left|f(y) - f_n(y)\right| +  \sup_n \vn{f_n - f} \varepsilon.
\end{align*}
As $n \rightarrow \infty$ this quantity is bounded by $\sup_n \vn{f_n - f} \varepsilon$ as $f_n$ converges to $f$ uniformly on compacts. As $\varepsilon$ was arbitrary, we are done.
\end{proof}

\begin{proposition} \label{proposition:semigroup_strongly_continuous}
Let $\{S(t)\}_{t \geq 0}$ be the semigroup introduced in Theorem \ref{theorem:verification_conditions_semigroup}. Then $\{S(t)\}_{t \geq 0}$ is $\beta$ strongly continuous.
\end{proposition}

For the proof, we  recall  the notion of a weakly continuous semigroup. A semigroup is \textit{weakly continuous} if for all $f \in C_b(\cX)$ and $\mu \in \cM(\cX)$ the trajectory $t \mapsto \ip{S(t)f}{\mu}$ is continuous in $\bR$.

\begin{proof}[Proof of Proposition \ref{proposition:semigroup_strongly_continuous}]
First, recall that $(C_b(\cX),\beta)$ is strong Mackey and complete. By Lemma \ref{lemma:semigroup_is_for_every_time_continuous} the semigroup $\{S(t)\}_{t \geq 0}$ is locally equi-continuous. Therefore, Proposition 3.5 in \cite{Kr16} implies that we only need to prove weak continuity. So let $f \in C_b(\cX)$ and $\mu \in \cM(\cX)$. Write $\mu$ as the Hahn-Jordan decomposition: $\mu = c^+ \mu^+ - c^-\mu^-$, where $c^+,c^- \geq 0$ such that $\mu^+, \mu^- \in \cP(\cX)$.
	
We show that $t \mapsto \ip{S(t)f}{\mu}$ is continuous, by showing that $t \mapsto \ip{S(t)f}{\mu^+}$ and $t \mapsto \ip{S(t)f}{\mu^-}$ are continuous. Clearly, it suffices to do this for either of the two.
	
\smallskip
	
Let $\PR$ be the unique measure in $\cM_{\mu^+}$. It follows by Theorem 4.3.12 in \cite{EK86} that $\PR[X(t) = X(t-)] = 1$ for all $t > 0$, so $t \mapsto X(t)$ is continuous $\PR$ almost surely. Fix some $t > 0$, we show that our trajectory is continuous at this specific $t$.
\begin{equation*}
\left|\ip{S(t)f}{\mu^+} - \ip{S(t+h)f}{\mu^+} \right|  \leq \bE^\PR\left|f(\eta(t) - f(\eta(t+h)) \right|. 
\end{equation*}
By the almost sure convergence of $X(t+h) \rightarrow X(t)$ as $h \rightarrow 0$, and the boundedness of $f$, we obtain by the dominated convergence theorem that this difference converges to $0$ as $h \rightarrow 0$. As $t > 0$ was arbitrary, the trajectory is continuous at all $t > 0$. Continuity at $0$ follows by the fact that all trajectories in $D_\cX(\bR^+)$ are continuous at $0$.
\end{proof}

\begin{proposition} \label{proposition:generator_extends_martingale_problem_operator}
Let $\{S(t)\}_{t \geq 0}$ be the semigroup introduced in Theorem \ref{theorem:verification_conditions_semigroup} and let $\hat{A}$ be the generator of this semigroup. Then $\hat{A}$ is an extension of $A$.
\end{proposition}

\begin{proof}
We show that if $f \in \cD(A)$, then $f \in \cD(\hat{A})$. We again use the characterisation of $\beta$ convergence as given in Theorem \ref{theorem:propertiesCbstrict} (d). From this point onward, we write $g := Af$ to ease the notation.
	
First, $\sup_{t} \vn{\frac{S(t) f - f}{t}} \leq \vn{g}$ as
\begin{equation*}
\frac{S(t)f(x) - f(x)}{t} = \bE_x\left[\frac{f(\eta(t)) - f(x)}{t}\right] = \bE_x \left[\frac{1}{t} \int_0^t g(\eta(s)) \dd s \right]
\end{equation*}
Second, we show that we have uniform convergence of $\frac{S(t) f - f}{t}$ to $g$ as $t \downarrow 0$ on compacts sets. So pick $K \subseteq \cX$ compact. Now choose $\varepsilon > 0$ arbitrary, and let $\hat{K} = \hat{K}(K,\varepsilon,1)$ as in \eqref{eqn:theorem_compact_containment_set}.
\begin{align}
& \sup_{x \in K} \left|\frac{S(t)f(x) - f(x)}{t} - g(x)\right|\notag \\ 
& \leq \sup_{x \in K}  \bE_x \left|\frac{1}{t} \int_0^t g(X(s)) - g(x) \dd s \right|\notag \\
& \leq \sup_{x \in K}  \bE_x \bONE_{\{\eta(s) \in \hat{K} \text{ for } s \leq 1\}} \left|\frac{1}{t} \int_0^t g(\eta(s)) - g(x) \dd s \right|\notag \\
& \quad + \sup_{x \in K}  \bE_x \bONE_{\{\eta(s) \notin \hat{K} \text{ for } s \leq 1\}} \left|\frac{1}{t} \int_0^t g(\eta(s)) - g(x) \dd s \right| \notag \\
& \leq \sup_{x \in K}  \bE_x \bONE_{\{\eta(s) \in \hat{K} \text{ for } s \leq 1\}} \left|\frac{1}{t} \int_0^t g(\eta(s)) - g(x) \dd s \right|  + 2\varepsilon  \vn{g}\label{eqn:bound_on_convergence_generator_martingale_problem}
\end{align}
Thus, we need to work on the first term on the last line. 
	
\bigskip
	
The function $g$ restricted to the compact set $\hat{K}$ is uniformly continuous. So let $\varepsilon' > 0$, chosen smaller then $\varepsilon$, be such that if $d(x,y) < \varepsilon'$, $x,y \in \hat{K}$, then $|g(x) - g(y)| \leq \varepsilon$.
	
\bigskip
	
By Lemma 4.5.17 in \cite{EK86}, the set $\{\PR_x \, | \, x \in K\}$ is a weakly compact set in $\cP(D_\cX(\bR^+))$. So by Theorem 3.7.2 in \cite{EK86}, we obtain that there exists a $\delta = \delta(\varepsilon') > 0$ such that
\begin{equation*}
\sup_{x \in K} \PR_x \left[\eta \in D_\cX(\bR^+) \, | \, \sup_{s \leq \delta} d(\eta(0),\eta(s)) < \varepsilon'\right] > 1 - \varepsilon' > 1 - \varepsilon.
\end{equation*}
Denote $S_\delta := \{\eta \in D_\cX(\bR^+) \, | \, \sup_{s \leq \delta} d(\eta(0),\eta(s)) < \varepsilon' \} $, so that we can summarise the equation as $\sup_{x \in K} \PR_x[S_\delta] > 1 - \varepsilon$.
	
\bigskip
	
We reconsider the term that remained in equation \eqref{eqn:bound_on_convergence_generator_martingale_problem}. 
\begin{align*}
& \sup_{x \in K}  \bE_x \bONE_{\{\eta(s) \in \hat{K} \text{ for } s \leq 1\}} \left|\frac{1}{t} \int_0^t g(\eta(s)) - g(x) \dd s \right|  + 2 \varepsilon  \vn{g} \\
& \leq \sup_{x \in K}  \bE_x \bONE_{\{\eta(s) \in \hat{K} \text{ for } s \leq 1\} \cap S_\delta} \left|\frac{1}{t} \int_0^t g(\eta(s)) - g(x) \dd s \right|  + 4 \varepsilon \vn{g} 
\end{align*}
On the set $\{\eta(s) \in \hat{K} \text{ for } s \leq 1\} \cap S_\delta$, we know that $d(\eta(s),x) \leq \eta$, so that by the uniform continuity of $g$ on $\hat{K}$, we obtain $|g(\eta(s)) - g(x)| \leq \varepsilon$. Hence:
\begin{equation*}
\sup_{t \leq 1 \wedge \delta(\varepsilon')} \sup_{x \in K} \left|\frac{S(t)f(x) - f(x)}{t} - g(x)\right| \leq \varepsilon + 4 \varepsilon \vn{g}.
\end{equation*}
As $\varepsilon > 0$ was arbitrary, it follows that $f \in \cD(\hat{A})$ and $Af = g = \hat{A}f$.
	
\end{proof}

\subsection{Transition semigroup for a process on a locally compact space} \label{section:semigroup_on_locally_compact_space}

In the case that $\cX$ is a locally compact space, both $(C_b(\cX),\beta)$ and $(C_0(\cX),\vn{\cdot})$ have the space of Radon measures as a dual. As such, the space of Radon measures caries two weak topologies. The first one is the one that probabilist call the \textit{weak} topology, i.e.~$\sigma(\cM(\cX),C_b(\cX))$, and the second is the weaker \textit{vague} topology, i.e.~$\sigma(\cM(\cX),C_0(\cX))$.
Thus, we expect that if a semigroup is strongly continuous on $(C_b(\cX),\beta)$ it is strongly continuous on $(C_0(\cX),\vn{\cdot})$, as long as the semigroup maps $C_0(\cX)$ into itself.

\begin{theorem} \label{theorem:transition_semigroup_is_strongly_continuous}
Let $\{S(t)\}_{t \geq 0}$ be a SCLE semigroup on $(C_b(\cX),\beta)$ such that $S(t) C_0(\cX) \subseteq C_0(\cX)$ for every $t \geq 0$. Then the restriction of the semigroup to $C_0(\cX)$, denoted by $\{\tilde{S}(t)\}_{t \geq 0}$ is $\vn{\cdot}$ strongly continuous.
	
\medskip
	
Conversely, suppose that we have a strongly continuous semigroup $\{\tilde{S}(t)\}_{t \geq 0}$ on $(C_0(\cX),\vn{\cdot})$ such that $\tilde{S}'(t) \cP(\cX) \subseteq \cP(\cX)$. Then the semigroup can be extended uniquely to a SCLE semigroup $\{S(t)\}_{t \geq 0}$ on $(C_b(\cX),\beta)$.

\medskip

In this setting, denote by $(A,\cD(A))$ the generator of $\{S(t)\}_{t \geq 0}$ on $(C_b(\cX),\beta)$ and by $(A,\tilde{A})$ the generator of $\{\tilde{S}(t)\}_{t \geq 0}$ on $(C_0(\cX),\vn{\cdot})$. Then $\tilde{A} \subseteq A$ and $A$ is the $\beta$ closure of $\tilde{A}$.
\end{theorem}

\begin{proof}
	\textit{Proof of the first statement.}

	 For a given time $t \geq 0$, the operator $S(t)$ is continuous on $(C_0(\cX),\vn{\cdot})$, because $S(t)$ is $\beta$ continuous and therefore maps $\beta$-bounded sets into $\beta$-bounded sets. Norm continuity of the restriction $\tilde{S}(t)$ on $C_0(\cX)$ then follows by the fact that the bounded sets for the norm and for $\beta$ coincide.
	
\bigskip
	
As $\{S(t)\}_{t \geq 0}$ is $(C_b(\cX), \beta)$ is strongly continuous, it is also weakly continuous, in other words, for every Radon measure $\mu$, we have that 
\begin{equation*}
t \mapsto \ip{S(t)f}{\mu}
\end{equation*}
is continuous for every $f \in C_b(\cX)$ and in particular for $f \in C_0(\cX)$. Theorem I.5.8 in Engel and Nagel \cite{EN00} yields that the semigroup $\{\tilde{S}(t)\}_{t \geq 0}$ is strongly continuous on $(C_0(\cX),\vn{\cdot})$.

\smallskip

	\textit{Proof of the second statement.}

 First note that such a $\beta$-continuous extension must be unique by the Stone-Weierstrass theorem, cf.~Theorem \ref{theorem:propertiesCbstrict} (e), which implies that $C_0(\cX)$ is $\beta$ dense in $C_b(\cX)$. We will show that $\tilde{S}(t)$ is $\beta$ to $\beta$ continuous, because we can then extend the operator by continuity to $C_b(\cX)$. In fact, we will prove a stronger statement, namely that $\{\tilde{S}(t)\}_{t \geq 0}$ is locally $\beta$ equi-continuous. 
\bigskip
	
First of all, by the completeness of $(C_b(\cX),\beta)$, the fact that $C_0(\cX)$ is dense in $(C_b(\cX),\beta)$ and 21.4.(5) in \cite{Ko69}, we have $(C_0(\cX),\beta)' = (C_b(\cX),\beta)' = \cM(\cX)$ and the equi-continuous sets in $\cM(\cX)$ with respect to $(C_0(\cX),\beta)$ and $(C_b(\cX),\beta)$ coincide. It follows by 39.3.(4) in \cite{Ko79} that $\{\tilde{S}(t)\}_{t \geq 0}$ is locally $\beta$ equi-continuous if for every $T \geq 0$ and $\beta$ equi-continuous set $K \subseteq \cM(\cX)$ we have that 
\begin{equation*}
\fS K := \{\tilde{S}'(t)\mu \, | \, t \leq T, \mu \in K \}
\end{equation*}
is $\beta$ equi-continuous. By Theorem 6.1 (c) in Sentilles \cite{Se72}, it is sufficient to prove this result for $\beta$ equi-continuous sets $K$ consisting of non-negative measures in $\cM(\cX)$. Furthermore, we can restrict to weakly closed $K$, as the weak closure of an equi-continuous set is $\beta$ equi-continuous.
	
\bigskip
	
Therefore, let $K$ be an arbitrary weakly closed $\beta$ equi-continuous subset of the non-negative Radon measures. We show that $\fS K$ is relatively weakly compact, as this will imply that $\fS K$ is $\beta$ equi-continuous as $\beta$ is a strong Mackey topology, cf. Theorem \ref{theorem:propertiesCbstrict} (a). This in turn would establish that $\{\tilde{S}(t)\}_{t \geq 0}$ is locally $\beta$ equi-continuous. 
	
\smallskip
	
By Theorem 8.9.4 in \cite{Bo07}, we obtain that the weak topology on the positive cone in $\cM(E)$ is metrisable. So, we only need to show sequential relative weak compactness of $\fS K$. Let $\nu_n$ be a sequence in $\fS K$. Clearly, $\nu_n = \tilde{S}'(t_n) \mu_n$ for some sequence $\mu_n \in K$ and $t_n \leq T$. As $K$ is $\beta$ equi-continuous, it is weakly compact by the Bourbaki-Alaoglu theorem, so without loss of generality, we restrict to a weakly converging subsequence $\mu_n \in K$ with limit $\mu \in K$ and $t_n \rightarrow t$, for some $t \leq T$.
	
\smallskip
	
Now there are two possibilities, either $\mu = 0$, or $\mu \neq 0$. In the first case, we obtain directly that $\nu_n = \tilde{S}'(t_n) \mu_n \rightarrow 0 \ni K \subseteq \fS K$ weakly. In the second case, one can show that
\begin{equation*}
\hat{\mu}_n := \frac{\mu_n}{\ip{\bONE}{\mu_n}} \rightarrow \frac{\mu}{\ip{\bONE}{\mu}} =: \hat{\mu}
\end{equation*}
weakly, and therefore vaguely. As $\{\tilde{S}(t)\}_{t \geq 0}$ is strongly continuous on $(C_0(\cX),\vn{\cdot})$, it follows that $\tilde{S}'(t_n)\hat{\mu}_n \rightarrow \tilde{S}'(t) \hat{\mu}$ vaguely. By assumption, all measures involved are probability measures, so by Proposition 3.4.4 in Ethier and Kurtz \cite{EK86} implies that the convergence is also in the weak topology. By an elementary computation, we infer that the result also holds without the normalising constants: $\nu_n \rightarrow \tilde{S}'(t) \mu$ weakly. 
	
So both cases give us a weakly converging subsequence in $\fS K$.
	
\smallskip
	
We conclude that $\{\tilde{S}(t)\}_{t \leq T}$ is $\beta$ equi-continuous. So we can extend all $\tilde{S}(t)$ by continuity to $\beta$ continuous maps $S(t) \colon C_b(\cX) \rightarrow C_b(\cX)$. Also, we directly obtain that $\{S(t)\}_{t \geq 0}$ is locally equi-continuous. The semigroup property of $\{S(t)\}_{t \geq 0}$ follows from the semigroup property of $\{\tilde{S}(t)\}_{t \geq 0}$. The last thing to show is the $\beta$ strong continuity of $\{S(t)\}_{t \geq 0}$.
	
\bigskip
	
By Proposition 3.5 in \cite{Kr16} it is sufficient to show weak continuity of the semigroup $\{S(t)\}_{t \geq 0}$. Pick $\mu \in \cM(\cX)$, and represent $\mu$ as the Hahn-Jordan decomposition $\mu = \mu^+ - \mu^-$, where $\mu^+,\mu^-$ are non-negative measures. By construction, the adjoints of $S(t)$ and $\tilde{S}(t)$ coincide, so $t \mapsto S'(t)\mu^+$ and $t \mapsto S'(t)\mu^-$ are vaguely continuous. The total mass of the measures in both trajectories remains constant by the assumption of the theorem, so by Proposition III.4.4 in \cite{EK86}, we obtain that $t \mapsto S'(t)\mu^+$ and $t \mapsto S'(t)\mu^-$ are weakly continuous. This directly implies that $\{S(t)\}_{t \geq 0}$ is weakly continuous and thus strongly continuous. 

\smallskip

	\textit{Proof of the third statement.}

 Let $(\tilde{A},\cD(\tilde{A}))$ be the generator of $\{\tilde{S}(t)\}_{t \geq 0}$ and $(A,\cD(A))$ the one of $\{S(t)\}_{t \geq 0}$. As the norm topology is stronger than $\beta$, it is immediate that $\tilde{A} \subseteq A$.

We will show that $\cD(\tilde{A})$ is a \textit{core} for $(A,\cD(A))$, i.e.~$\cD(\tilde{A})$ is dense in $\cD(A)$ for the $\beta$-graph topology on $\cD(A)$. The result of Proposition II.1.7 in \cite{EN00} is proven for Banach spaces but also holds for spaces satisfying Condition C by replacing the norm topology by the locally convex one. 

Thus, it suffices to prove $\beta$ density of $\cD(\tilde{A})$ in $C_b(\cX)$ and that $S(t) \cD(\tilde{A}) \subseteq \cD(\tilde{A})$. 

The first claim follows because $\cD(\tilde{A})$ is norm, hence $\beta$, dense in $C_0(\cX)$ by Theorem II.1.4 in \cite{EN00} and because $C_0(\cX)$ is $\beta$ dense in $C_b(\cX)$ by the Stone-Weierstrass theorem, cf Theorem \ref{theorem:propertiesCbstrict} (e). The second claim follows because $S(t) \cD(\tilde{A}) = \tilde{S}(t) \cD(\tilde{A}) \subseteq \cD(\tilde{A})$ by e.g.~Lemma II.1.3 in \cite{EN00} or Lemma 5.2 in \cite{Kr16}. 

\smallskip

We conclude that $\cD(\tilde{A})$ is a core for $\cD(A)$. As $A$ is $\beta$ closed, it follows that $A$ is the $\beta$ graph-closure of $\tilde{A}$.
\end{proof}

\textbf{Acknowledgement}
RK is supported by the Deutsche Forschungsgemeinschaft (DFG) via RTG 2131 High-dimensional Phenomena in Probability – Fluctuations and Discontinuity. MS is supported by the European Research Council under ERC Grant Agreement 320637.

\sloppy 
\printbibliography

\end{document}